\documentclass[11pt,letterpaper]{amsart}
\usepackage{amsthm}
\usepackage{times,amssymb,amsmath,amsfonts}
\usepackage{graphicx,color}
\usepackage{subcaption}
\usepackage{xfrac}
\usepackage{mathtools}
\usepackage{blindtext}

\newtheorem{theorem}{Theorem}[section]
\newtheorem{lemma}[theorem]{Lemma}
\newtheorem{corollary}[theorem]{Corollary}

\newtheorem{conjecture}[theorem]{Conjecture}

\newcommand{\Z}{{\mathbb Z}}
\DeclareMathOperator{\rank}{rank}

\theoremstyle{definition}
\newtheorem{definition}[theorem]{Definition}
\newtheorem{example}[theorem]{Example}
\newtheorem{remark}[theorem]{Remark}
\date{\today}

\begin{document}
		\title[Universal polar dual pairs]{Universal polar dual pairs of spherical codes found in $E_8$ and the Leech lattice}
		
		\author[S. V. Borodachov]{S. V. Borodachov}
		\address{ Department of Mathematics, Towson University\\ 
		7800 York Rd, Towson, MD, 21252, USA}
\email{sborodachov@towson.edu}

	\author[P. G. Boyvalenkov]{P. G. Boyvalenkov} 
\address{ Institute of Mathematics and Informatics, Bulgarian Academy of Sciences \\
8 G Bonchev Str., 1113  Sofia, Bulgaria}
\email{peter@math.bas.bg}

\author[P. D. Dragnev]{P. D. Dragnev}
\address{ Department of Mathematical Sciences, Purdue University \\
Fort Wayne, IN 46805, USA }
\email{dragnevp@purdue.edu}

\author{D. P. Hardin}
\address{ Center for Constructive Approximation, Department of Mathematics \\
Vanderbilt University, Nashville, TN 37240, USA }
\email{doug.hardin@vanderbilt.edu}

\author{E. B. Saff}
\address{ Center for Constructive Approximation, Department of Mathematics \\
Vanderbilt University, Nashville, TN 37240, USA }
\email{edward.b.saff@vanderbilt.edu}

\author[M. M. Stoyanova]{M. M. Stoyanova} 
\address{ Faculty of Mathematics and Informatics, Sofia University St. Kliment Ohridski\\
5 James Bourchier Blvd., 1164 Sofia, Bulgaria}
\email{stoyanova@fmi.uni-sofia.bg}

	 \begin{abstract}  

We identify universal polar dual pairs of spherical codes $C$ and $D$ such that for a large class of potential functions $h$ the minima of the discrete $h$-potential of $C$  on the sphere occur at the points of $D$ and vice versa. Moreover, the minimal values of their normalized potentials are equal. These codes arise from the known hierarchies of sharp codes embedded in the even unimodular extremal lattices $E_8$  and $\Lambda_{24}$ (Leech lattice). This embedding allows us to use the lattices' properties to find new universal polar dual pairs. In the process we extensively utilize the interplay between the binary Golay codes and the Leech lattice. 

As a byproduct of our analysis, we identify a new universally optimal  (in the sense of energy) code in the projective space $\mathbb{RP}^{21}$ with $1408$ points (lines). 
Furthermore, we extend the Delsarte-Goethals-Seidel definition of derived codes from their seminal $1977$ paper and generalize their Theorem 8.2 to show that if a $\tau$-design is enclosed in $k\leq \tau$ parallel hyperplanes, then each derived code is a $(\tau+1-k)$-design in the ambient subspace.
\end{abstract}

 		\maketitle

{\bf Keywords}: {\it Discrete potentials, sharp spherical configurations, linear programming, Gauss-Jacobi  quadrature, universal polarization bounds, polar duality, $E_8$ and Leech lattices}.
\vskip 3mm
{\bf MSC 2020}: {\it 05B30, 52C17, 74G65,  94B65; 05E30, 33C45, 52A40}

\section{Introduction}

A non-empty finite subset $C$ of the unit sphere $\mathbb{S}^{n-1}:=\{x\in \mathbb{R}^n\, :\, |x|=1\}$, $n\geq 2$,  in
Euclidean space is called a {\em spherical code}. 
For a function $h:[-1,1] \to (-\infty,+\infty]$, continuous in the extended sense on $[-1,1]$ and finite on $[-1,1)$, we consider the {\em discrete $h$-potential associated with $C$} given by
\[ U_h(x,C):=\sum_{y\in C}h(x \cdot y),\]
where $x \in \mathbb{S}^{n-1}$ is arbitrary and $x\cdot y$ denotes the inner product in $\mathbb{R}^n$.

\begin{definition} \label{PolarDuality}  Let $C\subset \mathbb{S}^{n-1}$ be a spherical code. We say that $D\subset \mathbb{S}^{n-1}$ is the {\em set of universal minima}\footnote{We emphasize that for most codes $C$ a universal minima set does not exist.} associated with $C$ if  
\[ D={\rm argmin}_{x\in \mathbb{S}^{n-1} }\ U_h(x,C) ,\]
for all strictly absolutely monotone potentials $h$ on $[-1,1)$ (that is $h^{(k)} (t) > 0, k=0,1,\dots$). 
A pair of spherical codes $(C,D)$, such that $D$ is the set of universal minima associated with $C$ and $C$ is the set of universal minima associated with $D$, is called {\em a universal polar dual pair}.
\end{definition} 

Geometrically, a set of universal minima must be contained in the set of deep holes; that is, the points in $\mathbb{S}^{n-1}$ whose distance to the code $C$ is maximal (see \cite[Proposition 14.4.1]{BHS} and \cite[Theorem 4.8(i)]{BDHSS-Sharp}). For example, the universal minima set of the regular icosahedron is its dual dodecahedron and vice versa. Hence, they form a universal polar dual pair (see \cite[Theorems 3.3 and 3.4]{Bor-new}) and their universal minima sets coincide with the sets of deep holes. The same situation occurs with the remaining Platonic solids. The octahedron and the cube form a universal polar dual pair, while the tetrahedron forms a universal polar dual pair with its own antipode.

The phenomenon of certain regular spherical codes forming pairs where each code is the set of minima of the potential of the other code was first observed for the case of power-law interactions by Stolarsky in 1975 (see \cite{Sto1975circle,Sto1975}) and subsequently by Nikolov and Rafailov \cite {NikRaf2011,NikRaf2013}. These authors demonstrated this phenomenon for the set of vertices of a regular $N$-gon on $\mathbb{S}^1$ (which forms such a pair with $N$ midpoints of the arcs joining its neighboring vertices), regular simplex on $\mathbb{S}^{n-1}$ (which forms such a pair with its antipode), and the cube--cross-polytope pair on $\mathbb{S}^{n-1}$. Hardin, Kendall, and Saff \cite{HarKenSaf2013}, established this phenomenon for the regular $N$-gon and a $\pi/N$ rotational counterpart for potential interactions given by a decreasing and convex function of the geodesic distance (this includes a certain class of absolutely monotone kernels). Later Borodachov \cite{B,Bor-2} showed that the regular $N$-gon pairs, the $n$-simplex and its antipode, as well as the cross-polytope and the $n$-cube form universal polar dual pairs on $\mathbb{S}^{n-1}$ (see also \cite{BDHSS-JMAA}).
 
Some of the codes mentioned above are also known solutions to the {\em max-min polarization} (maximizing the minimum value of the potential of $N$ points over the sphere) and {\em min-max polarization} (minimizing the maximum value of the potential of $N$ points over the sphere). The optimality of the set of vertices of a regular $N$-gon was proved for max-min polarization in the works by Ambrus \cite{Amb2009}, Nikolov and Rafailov \cite{NikRaf2011}, Ambrus, Ball, and Erdel\'yi \cite{AmbBalErd2013}, Erdel\'yi and Saff \cite{ErdSaf2013} for particular potentials, and for the general case in Hardin, Kendall, and Saff \cite{HarKenSaf2013}. For min-max polarization this optimality was derived by Stolarsky \cite{Sto1975circle}, Nikolov and Rafailov \cite{NikRaf2011}, and Borodachov \cite{B}. The set of vertices of a regular simplex on $\mathbb{S}^{n-1}$ was shown to have both max-min and min-max polarization properties by Su \cite{Su2014} and Borodachov \cite{B2,B}. While the max-min polarization property of the cross-polytope is still open, it was established by Boyvalenkov et al. in \cite[Proposition 6.7]{BDHSS-JMAA} for all centered codes (i.e.,codes for which there is a point on the sphere whose inner products to points of the cross-polytope are in the interval $[-1/\sqrt{n},1/\sqrt{n}]$; this class includes all antipodal $2n$-point codes). 

A common property for the regular $N$-gon on $\mathbb{S}^1$, the regular simplex, and the cross-polytope in $\mathbb{S}^{n-1}$ is that they are {\em sharp codes} in the sense of Cohn and Kumar \cite{CK}; that is, they are spherical $(2m-1)$- or $2m$-designs (see Definition \ref{def-designs-1}) that have $m$ different inner products between distinct points in the code. All such codes known to the authors are listed in \cite[Table 1]{CK}. 

In this article we shall verify the following Claim highlighting prominently the connection between sharp codes and universal polar dual pairs.

{\bf Claim:} {\it All spherical sharp codes known to date (aside from the generalized quadrangles family), as well as the $24$-cell generate universal polar dual pairs.}

\begin{remark} To illustrate the term ``generate" above, in Theorem \ref{E8-pairs-1-5}(e) we consider the Clebsch code $C_{16}=(5,16,3)$, whose set of universal minima is the cross-polytope $C_{10}$ in $\mathbb{R}^5$. However, the set of universal minima of $C_{10}$ is the {\em symmetrized} Clebsch code, where for a code $C$ we refer to $C\cup(-C)$ as the symmetrization of $C$ (see the caption of Table \ref{PULB-pairs}). In fact, more generally, for any sharp code in the above Claim, we establish that its symmetrization is a member of a universal polar dual pair. In addition, if a sharp code $C$ is a spherical design of even strength, then the universal polar dual pair is $(C,-C)$ (see Remark \ref{EvenTight}).  
\end{remark}

\begin{remark}
As our definition of universal polar dual pairs refers to finite sets $C$ and $D$, we exclude degenerate simplexes in the above Claim. For example, if $C$ is the set of vertices of an equilateral triangle on the equator in $\mathbb{S}^2$, the set of universal minima $D$ consists of the two poles. However, the set of minimal points of the potential of $D$ is the entire equator.  
\end{remark}

We next provide additional examples illustrating the concept of universal polar dual pairs in the context of lattices. 
 
\begin{figure}[ht]
\begin{center}
\includegraphics[scale=1.3]{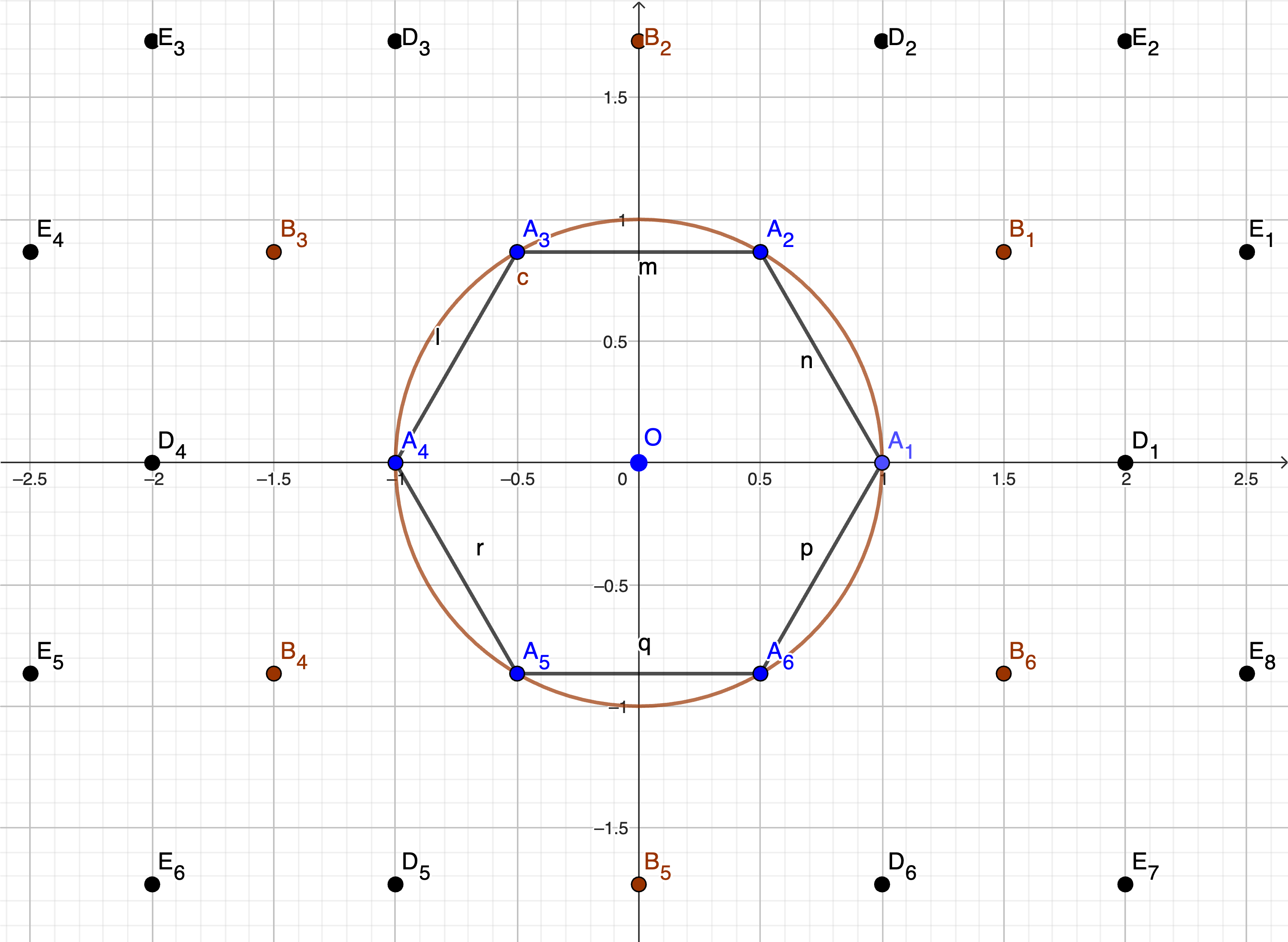} \quad \quad 
\includegraphics[scale=1.3]{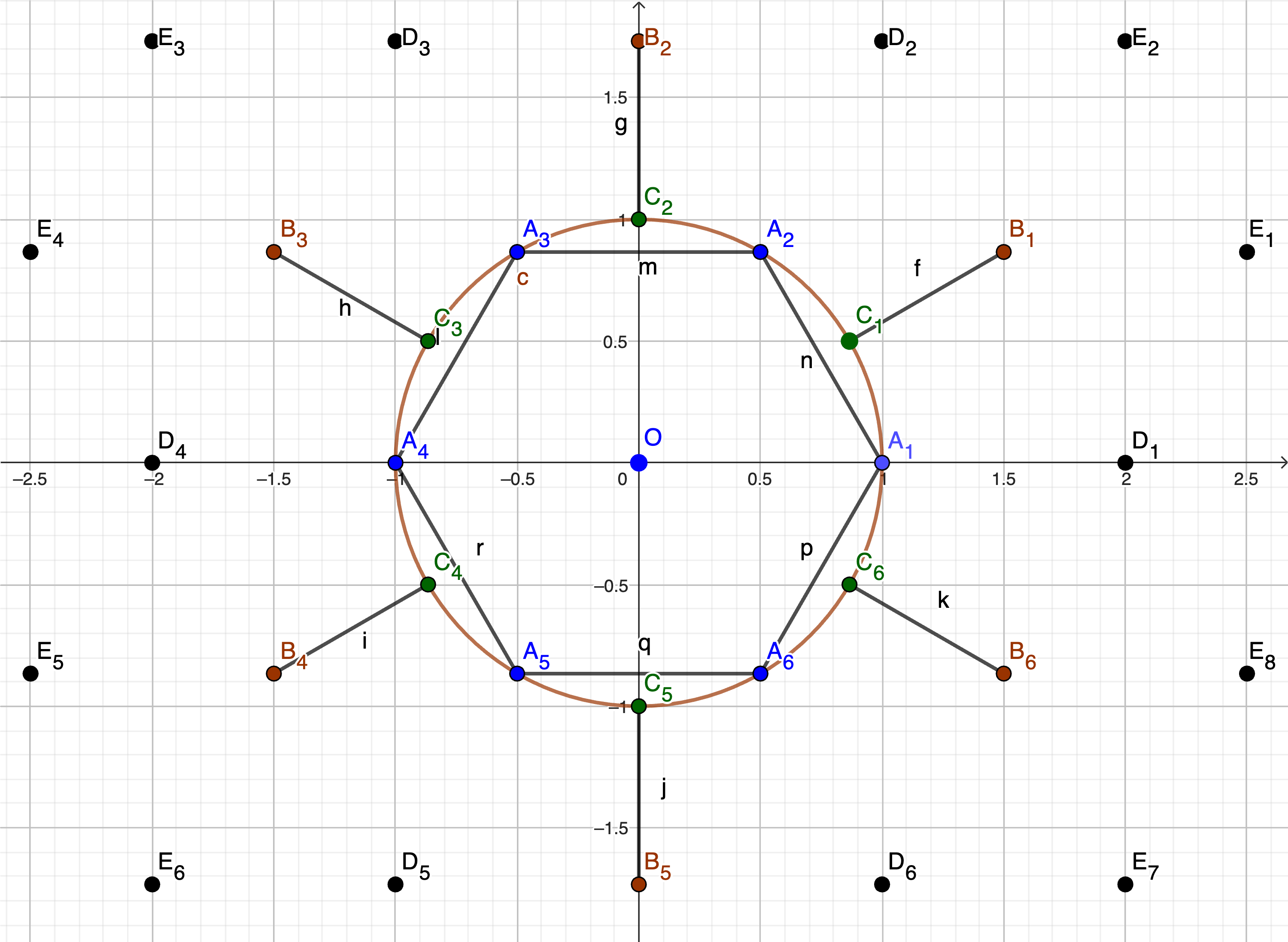}
\caption{The universal polar dual pair $(A,B)$ formed from projections of the first and second layers of the hexagonal lattice.} 
\label{Hex_lat}
\end{center}
\end{figure}

\begin{example} We highlight the polar duality with the case of the regular hexagon (see Figure \ref{Hex_lat}). Let us embed the regular hexagon $A=\{A_i\}_{i=1}^6$ in a hexagonal lattice with $0$ at the center of the hexagon. The global minima of the potential $U_h(x,A)$, where $x$ belongs to the circumscribing circle, are the midpoints of the arcs between the six vertices of $A$ (see \cite{HarKenSaf2013}), which are also projections of the second layer $B=\{B_i\}_{i=1}^6$ onto the circle, and vice versa, the global minima of $U_h(y,B)$, where $y$ belongs to the circumscribing circle of $B$, are the projections of $A_i$ onto that circle. Thus, we can think of the universal polar dual pair $(A,B)$ geometrically as the first and second layers of the hexagonal lattice. 

\end{example}

We shall see in Theorems \ref{E8-pairs-1-5}(a) and \ref{LeechPULBpair} that the first two layers of the celebrated $E_8$ and the Leech lattices form universal polar dual pairs. The next example illustrates that the same feature holds for another remarkable lattice.

\begin{example}
Recall that the $D_4$ lattice is the collection of points in $\mathbb{R}^4$ with integer coordinates, whose sum is even. In \cite{Bor-2} Borodachov showed that the first layer $A:=\{((\pm 1)^2,0^2)\}$ and second layer $B:=\{ ((\pm 2)^1,0^3), ((\pm 1)^4)\}$ of $D_4$, projected onto $\mathbb{S}^3$ form a universal polar dual pair $(A,B)$.
\end{example}

For the purposes of verification and derivation of universal polar dual pairs we make use of properties of {\em spherical designs} \cite{Ban09, Ban17, BBIT, CS, DGS, EZ, Lev}, as well as lower bounds for minima of discrete potentials which we refer to as {\em polarization universal lower bounds} (PULB) of spherical codes and designs \cite{B2, Bor2022talk, B, Bor-new, Bor-2, Bor-FL, BDHSS-JMAA, BDHSS-Sharp}. 

We remind the reader of a few equivalent definitions of spherical designs as they were originally introduced in 1977 by Delsarte, Goethals, and Seidel \cite[Section 5]{DGS} (for a comprehensive survey see \cite{Ban09,BBIT}). 

\begin{definition} \label{def-designs-1} \cite[Definition 5.1]{DGS} A spherical code $C \subset \mathbb{S}^{n-1}$ of cardinality $|C|=N$ is a {\em spherical $\tau$-design} if
for any homogeneous polynomial $p (x_1,\dots,x_n)$ of total degree at most $ \tau$ 
and any orthogonal transformation $U$ on $\mathbb{S}^{n-1}$ the following equality holds:
\begin{equation}\label{Def51}
\sum_{x\in C} p(Ux)=\sum_{x\in C} p(x).
\end{equation} 
The maximal such $\tau$ for a code $C$ is called its {\em design strength}. Hereafter $Ux$ denotes the image of the vector $x$ under the element $U$ of the orthogonal
group $O(n)$. Equivalently, for all polynomials $q(x) = q(x_1,x_2,\ldots,x_n)$ of degree at most $\tau$ we have
\begin{equation*}\label{DefQR}
\frac{1}{\mu(\mathbb{S}^{n-1})} \int_{\mathbb{S}^{n-1}} q(x) d\mu(x) =
                  \frac{1}{N} \sum_{x \in C} q(x), \end{equation*}
where $\mu$ is the surface area measure (i.e., the
average of $q$ over the set $C$ is equal to the average of $q$ over
$\mathbb{S}^{n-1}$). We shall refer to $C$ as a $(n,N,\tau)$-code.
\end{definition}

The concept was further extended to spherical $T$-designs by Delsarte and Seidel \cite{DS89} in 1989 (see also \cite[Section 6.1]{Ban17}). 
Given a spherical code $C  \subset   \mathbb{S}^{n-1}$, its {\em $\ell$-th moment}, $\ell\in \mathbb{N}$, is defined as
\begin{equation*}\label{Moments}
{\mathcal M}_\ell^n(C):=\sum_{x,y\in C} P_\ell^{(n)} (x\cdot y),
\end{equation*}
where $P_\ell^{(n)} (t)$ are the {\em Gegenbauer polynomials}. Recall that these are Jacobi\footnote{The Jacobi polynomials 
$P_\ell^{(\alpha,\beta)}(t)$ are orthogonal on $[-1,1]$ with respect to a weight function $(1-t)^\alpha (1+t)^\beta$.} 
polynomials $P_\ell^{(\alpha,\beta)}(t)$ with parameters
$\alpha=\beta=(n-3)/2$ normalized so that $P_\ell^{(n)} (1)=1$. 
The positive definiteness of the Gegenbauer polynomials (see e.g. \cite{Sch42}, \cite[Chapter 5]{BHS}) implies 
that $\mathcal{M}_\ell^n(C) \geq 0$ for every positive integer $\ell$. 

\begin{definition}\cite[Definition 6.1]{Ban17} \label{T-design}
Given an {\em index set} $T\subset \mathbb{N}$, we call a spherical code $C \subset \mathbb{S}^{n-1}$ a {\em $T$-design} if $ \mathcal{M}_\ell^n(C)=0$ 
for every $\ell\in T$. If $T=\{1,2,\dots, \tau\}$, then $C$ is a spherical $\tau$-design. When $T=\{1,2,\ldots,\tau+3\} \setminus \{\tau+1\}$, then $C$ is called $(\tau_{\sfrac{1}{2}})$-design (see  Venkov \cite[p. 44]{Ve}).  
\end{definition}

We define the {\em max-min} polarization quantities
\begin{equation*} \label{max-min}
m^h(C):=\min  \left \{ U_h(x,C) : x \in \mathbb{S}^{n-1}\right\}, \ \ m_N^h (\mathbb{S}^{n-1}):=\max \left\{ m^h(C) : |C|=N, \ C \subset \mathbb{S}^{n-1} \right\}. 
\end{equation*}

We shall also consider the counterpart of the latter quantity when we restrict to codes $C$ that are $T$-designs of cardinality $|C|=N$
\begin{equation*} \label{max-min_T}
m^h_{N,T}(\mathbb{S}^{n-1}):=\max \{ m^h(C)\, :\, |C|=N, \ \text{$C$ is a $T$-design on $\mathbb{S}^{n-1}$} \}.
\end{equation*}
In the case $T=\{1,2,\ldots,\tau\}$ we write $\tau$ instead of $T$ in the above notation. The definitions imply that if a $T$-design $C$ with $|C|=N$ exists, then 
\[ m_N^h (\mathbb{S}^{n-1}) \geq m^h_{N,T}(\mathbb{S}^{n-1}). \]
Thus, the polarization bounds below (Theorems \ref{PULB} and \ref{PULB2} and the exact values of $m^h(C)$ throughout the paper) provide bounds for the quantities
$m_{N,T}^h(\mathbb{S}^{n-1})$ for respective $T$, $n$,  $N$, and $h$, and hence for $m_N^h (\mathbb{S}^{n-1})$.

We briefly summarize the polarization LP framework from \cite{BDHSS-JMAA}. Recall that the Gegenbauer polynomials are orthogonal on $[-1,1]$ with orthogonality measure 
\begin{equation} \label{mu_n} d\mu_n(t):=\gamma_n (1-t^2)^{(n-3)/2}\, dt , \end{equation}
where the normalization constant $\gamma_n$ is chosen to make $\mu_n$ a probability measure. Thus, any real polynomial $f$ can be written as
\begin{equation} \label{geg-exp} 
f(t)=\sum_{\ell=0}^{\deg (f)} f_\ell P_\ell^{(n)}(t) 
\end{equation}
with Gegenbauer coefficients $f_\ell$ given by
\[f_\ell := \int_{-1}^1 f(t) P_\ell^{(n)}(t)\, d \mu_n(t)/\|P_\ell^{(n)}\|^2_{\mu_n}, \quad \ell=0,\dots, \deg (f). \]

The following equivalent definition of a spherical design 
facilitates our approach (see \cite{DGS}, \cite[Equation (1.10)]{FL}).

\begin{definition} \label{def-des}  A code
$C \subset \mathbb{S}^{n-1}$ is a spherical $\tau$-design if and
only if for any point $x \in \mathbb{S}^{n-1}$ and any real 
polynomial $f(t)$ of degree at most $\tau$, the equality
\begin{equation}
\label{defin_f}
U_f(x,C)=\sum_{y \in C}f(x \cdot y) = f_0|C|
\end{equation}
holds, where $f_0=\int_{-1}^1 f(t) \, d \mu_n(t)$ is the constant term in the Gegenbauer expansion  \eqref{geg-exp} of $f$. Similarly, $C$ is a spherical $T$-design if and only if \eqref{defin_f} holds for any $f\in \mathcal{P}_T$, where
\begin{equation*}\label{P_T}
\mathcal{P}_T:=\text{span } \{P_\ell^{(n)} :\ell\in T\cup \{0\}\}.
\end{equation*}
\end{definition}

For $x\in \mathbb{S}^{n-1}$ and a code $C \subset \mathbb{S}^{n-1}$, let 
\begin{equation}\label{I(x,C)} 
I(x,C):=\{x\cdot y :y\in C\}=:\{u_i\}_{i=1}^k
\end{equation} 
and let $r_i$ denote the {\em relative frequency} of occurrence of $u_i$; i.e., $u_i = x\cdot y$ for $|C|r_i$ many distinct $y\in C$.  Observe that 
\begin{equation} \label{QRcode}
U_f(x,C)=\sum_{y \in C}f(x \cdot y) = |C|\sum_{i=1}^k r_i f(u_i).
\end{equation}
 
Note that \eqref{defin_f} asserts that the $f$-potential for a $\tau$-design $C$ is constant on ${\mathbb S}^{n-1}$, whenever $\deg (f) \leq \tau$. This fact serves as the foundation in obtaining universal minima and maxima in \cite{Bor-2} and lower and upper linear programming (LP) bounds for polarization, referred to as PULB and PUUB (polarization universal lower/upper bound) in \cite{BDHSS-JMAA}. In this article, we are interested in lower bounds and as such, will be utilizing the following PULB result from \cite{BDHSS-JMAA}.

\begin{theorem}\label{PULB} {\rm (\cite[Theorem 3.4, Corollary 3.9]{BDHSS-JMAA}, \cite[Theorem 4.3]{Bor-2})} Suppose $C$ is   a spherical $\tau$-design of cardinality $N$ on $\mathbb{S}^{n-1}$, where $\tau=:2k-1+\epsilon$, $\epsilon\in \{0,1\}$, and that the potential $h:[-1,1]\to (-\infty,\infty]$ is continuous on $[-1,1]$ (in the extended sense), finite on $(-1,1)$,
and has a (strictly) positive derivative $h^{(2k+\epsilon)}$ on $(-1,1)$.
Then 
\begin{equation}\label{PolarizationULB}
m^h(C) \geq N \sum_{i\in I} \rho_i h(\alpha_i),
\end{equation}
where the index set $I$, the quadrature nodes $\{\alpha_i\}_{i\in I} $, and the positive weights $\{\rho_i\}_{i\in I}$   are determined as follows:
$I:=\{1-\epsilon,\dots,k\}$, $\{\alpha_i\}_{i\in I}$ are the zeros of the (possibly adjacent\footnote{The adjacent Gegenbauer polynomials are, in fact, Jacobi polynomials  $P_k^{(\alpha+a,\beta+b)}(t)$, where $a,b \in \{0,1\}$ and $\alpha=\beta=(n-3)/2$. We denote them by $P_k^{(a,b)}(t)$ for short and normalize by $P_k^{(a,b)}(1)=1$.}) Gegenbauer polynomials $(1+t)^\epsilon P_k^{(0,\epsilon)}(t)$,
the weights $\{\rho_i\}_{i\in I}$ are positive, sum to 1, and are  given by
\begin{equation*} \label{RhoWeights}
\rho_i:= \int_{-1}^1 \ell_i (t) \, d \mu_n(t)=\int_{-1}^1 \ell_i^2 (t) \, d \mu_n(t),
\end{equation*} where
$\ell_i(t)$ denote the Lagrange basic polynomials\footnote{ $\ell_i(\alpha_j)=\delta_{ij}$, the Kronecker delta and ${\rm deg}\ \!\ell_i=\#I-1$.} associated with the nodes $\{ \alpha_i \}_{i \in I}$.

Moreover, the bound \eqref{PolarizationULB} is the best that can be attained by linear programming via polynomials $f$
of degree at most $\tau$ for which $f\leq h$ on $[-1,1]$.
 
In addition, if a spherical $\tau$-design $C$, $|C|=N$, attains the bound \eqref{PolarizationULB}, then
there exists a point $\widetilde{x} \in \mathbb{S}^{n-1}$ such that the set $I(\widetilde{x},C)$ of all inner products between $\widetilde{x}$ and the points of $C$ coincides with the set
 $\{\alpha_i\}_{i\in I}$, and the multiplicities of these inner products are $\{N\rho_i \}_{i\in I}$, respectively. In particular, the numbers $N\rho_i$, $i\in I$, are positive integers. 
 
Conversely, if $\widetilde{x}\in \mathbb{S}^{n-1}$ is such that $I(\widetilde{x},C)\subseteq \{\alpha_i\}_{i\in I}$, then $I(\widetilde{x},C) = \{\alpha_i\}_{i\in I}$ and the multiplicities of these inner products between $\widetilde{x}$ and the points of $C$ are $\{N\rho_i \}_{i\in I}$ and bound \eqref{PolarizationULB} is attained at $\widetilde{x}$; i.e.,it is a universal minimum.
\end{theorem}

The existence/nonexistence problem for codes attaining \eqref{PolarizationULB} was considered in \cite{BKN26}.

Motivated by the case of icosahedron, and the codes formed by the shortest vectors of the $E_8$ and Leech lattices, all of which are $(\tau_{\sfrac{1}{2}})$-designs, in \cite[Section 4]{BDHSS-Sharp} Boyvalenkov et al. extended the PULB \eqref{PolarizationULB} for such designs when $\tau$ is odd (the so-called skip 1-add 2 framework).

\begin{theorem}\label{PULB2} {\rm (\cite[Theorem 4.14]{BDHSS-Sharp})} 
For every $(\tau_{\sfrac{1}{2}})$-design $C \subset \mathbb{S}^{n-1}$, $\tau=2k-1$, of cardinality $N$ and every potential with $h^{(2k)}(t)>0$, $h^{(2k+1)}(t)>0$, and
$h^{(2k+2)}(t)>0$, $t \in (-1,1)$, the following bound holds
\begin{equation}\label{PolarizationULB2}
m^h(C) \geq N \sum_{i=1}^{k+1} r_i h(\beta_i).
\end{equation}
Here the quadrature nodes $\{\beta_1,\dots,\beta_{k+1}\}$ are symmetric about the origin 
and are the $k+1$ zeros of the equation 
\[ P_{k+1}^{(n)}(t)+bP_{k-1}^{(n)}(t)=0, \]
where the constant $b$ is found as the positive root of the quadratic equation (for $n>2$)
 \begin{equation*} \label{b-equ}
X^2 +\frac{(k+1)^2(n-2)(n+2k-4)}{(n+k-2)(n+k-3)(n+4k)}\cdot X -\frac{k(k+1)(n+2k-4)}{(n+k-2)(n+k-3)(n+2k)}=0.
\end{equation*}
The weights $\{r_1,\dots,r_{k+1}\}$ are positive and uniquely determined from the Lagrange basic polynomials associated with the $\beta_i$'s.

In addition, if a spherical $\tau$-design $C$, $|C|=N$, attains the bound \eqref{PolarizationULB2}, then
there exists a point $\widetilde{x} \in \mathbb{S}^{n-1}$ such that the set $I(\widetilde{x},C)$ of all inner products between $\widetilde{x}$ and the points of $C$ coincides with the set
 $\{\beta_i\}_{i=1}^{k+1}$, and the multiplicities of these inner products are $\{Nr_i \}_{i=1}^{k+1}$, respectively. In particular, the numbers $N r_i$, $i=1,\dots,k+1$,
are positive integers.

Conversely, if $\widetilde{x}\in \mathbb{S}^{n-1}$ is such that $I(\widetilde{x},C)\subseteq \{\beta_j\}_{j=1}^{k+1}$, then $I(\widetilde{x},C) = \{\beta_j\}_{j=1}^{k+1}$ and the multiplicities of these inner products between $\widetilde{x}$ and the points of $C$ are $\{Nr_j\}_{j=1}^{k+1}$ and bound \eqref{PolarizationULB2} is attained at $\widetilde{x}$, i.e.,it is a universal minimum.

\end{theorem}

In the same article the authors considered all known sharp codes except for the codes from isotropic subspaces (or generalized quadrangles, see Remark \ref{srg}), and showed that they attain the PULB 
\eqref{PolarizationULB} or \eqref{PolarizationULB2} 
by identifying at least one universal minimum in each case and characterizing the rest as orthogonal transformations between the unique facets associated with the 
universal minima, which are either (unique) sharp codes or unique strongly regular graphs (see \cite[Theorem 6.1]{BDHSS-Sharp}). We unify the optimal cases for Theorems \ref{PULB} and \ref{PULB2} in the following definition.

\begin{definition}\label{PULB-optimalDef}
We say that  a $\tau$- or $(\tau_{\sfrac{1}{2}})$-design $C$ of cardinality $N$ is {\em PULB-optimal} if it attains the lower bound \eqref{PolarizationULB} or \eqref{PolarizationULB2}, respectively, for all potentials $h$ satisfying the conditions of Theorems \ref{PULB} or \ref{PULB2}, respectively. In this case $D:={\rm argmin \ } U_h (x,C)$ is the set of universal minima associated with $C$, because Theorems \ref{PULB} and \ref{PULB2} are true for all absolutely monotone potentials. 
\end{definition}

In a contemporaneous article  \cite{Bor-2}, Borodachov, utilizing his notion of {\em $m$-stiff configurations} (that is, spherical $(2m-1)$-designs, which are contained in $m$ parallel hyperplanes), considered an alternative, more direct approach to determine the universal minima of some sharp codes as well as some other remarkable configurations (see \cite[Theorems 4.3 and 4.5]{Bor-2}). In the same paper, the author identified three ``pairs of mutually dual stiff configurations", namely, the symmetrized simplex in $\mathbb{R}^n$, and the shortest vectors of the $E_6$ and $E_7$ lattices, along with their dual stiff counterparts (see \cite[Table 3]{Bor-2}). Note that $m$-stiff configurations attain \eqref{PolarizationULB} and are thus PULB-optimal. Universal minima of the icosahedron, dodecahedron, shortest vectors of $E_8$ lattice and the Gosset  polytope $2_{41}$ are found in \cite{Bor-new} and \cite{Bor-FL}.

As we combine forces, our goal in this article is to determine the universal polar dual pairs generated by all sharp codes known to date (except for the codes from isotropic subspaces, which will be considered elsewhere) and some other remarkable spherical codes. Thus, for each of these codes we shall explicitly identify the collection of all universal minima, and in turn, will reveal new PULB-optimal configurations. All such pairs will have the same minimum value of their {\em normalized discrete $h$-potentials} (that is $U_h(x,C)/|C|$), given by the quadrature on the right-hand side of \eqref{PolarizationULB} and \eqref{PolarizationULB2}. This prompts the following unifying definition.

\begin{definition}\label{PULB-pair}
We say that a pair $(C,D)$ of PULB-optimal codes $C, D \subset \mathbb{S}^{n-1}$ is an {\em optimal polarization pair} or {\em PULB-optimal pair} if for all strictly absolutely monotone potentials $h$, the points of $D$ are minima for $U_h (x,C)$ and vice versa. A PULB-optimal pair that is a universal polar dual will be called {\em maximal}.
\end{definition}

We now state our main theorem.

\begin{theorem} \label{main-th} The codes listed in Table \ref{PULB-pairs} form maximal PULB-optimal pairs (and a fortiori universal polar dual pairs) embedded in $E_8$ and $\Lambda_{24}$. 
\end{theorem}

\begin{center}
\begin{table} 
\scalebox{0.75}{
\begin{tabular}{|c|c|c|c|c|c|}
\hline\noalign{\smallskip}
 Dimension & \multicolumn{2}{|c|}{Maximal PULB-optimal pair of codes $(C, D)$} & Lattice & Reference \\
\noalign{\smallskip}\hline\noalign{\smallskip}
 $n$ & Sharp code $C=(n,N,\tau)$ & $D=(n,N,\tau)$ & $E_8 \ | \ \Lambda_{24}$ & \\
\noalign{\smallskip}\hline\hline\noalign{\smallskip}
 & First layer of $E_8$ & Second layer $E_8$& & Th. \ref{E8-pairs-1-5} (a)\\
 $8$ & $C_{240}=(8,240,7)$ & $C_{2160}=(8,2160,7)$ & $E_8$ &  \cite{ Bor-new,Bor-FL}\\ 
\noalign{\smallskip}\hline\noalign{\smallskip}
 & kissing configuration, tight design & root system of $E_7$-lattice & &Th. \ref{E8-pairs-1-5} (b) \\
 $7$ & $C_{56}=(7,56,5)$ & $C_{126}=(7,126,5)$ & $E_8$ & \cite{Bor-2} \\
\hline\noalign{\smallskip}
 & Shl\"{a}fli code & antipodal to Shl\"{a}fli code& & Rmk. 2.4\\
 ${6}$ & $C_{27}=(6,27,4)$ & $-C_{27}=(6,27,4)$  & $E_8$ & \cite{Bor-2, BDHSS-Sharp}\\
\noalign{\smallskip}\hline\noalign{\smallskip} 
 & Symmetrized Shl\"{a}fli code & root system of $E_6$-lattice & & Th. \ref{E8-pairs-1-5} (c)\\
 $6$ & $C_{54}:=C_{27}\cup (-C_{27})=(6,54,5)$ & $C_{72}=(6,72,5)$ & $E_8$ & \cite{Bor-2}\\ 
\noalign{\smallskip}\hline\noalign{\smallskip}
 & Symmetrized 5-dim simplex & Equator of the $6$-dim cube& & Th. \ref{E8-pairs-1-5} (d)\\
 $5$ & $C_{12}:=C_6 \cup (-C_{6})=(5,12,3)$ & $C_{20}=(5,20,3)$ & $E_8$ &  \cite{Bor-2}\\ 
\noalign{\smallskip}\hline\noalign{\smallskip}
 & Symmetrized Clebsch code ($5$-cube) & $5$-dim cross-polytope& & Th. \ref{E8-pairs-1-5} (e)\\
 $5$ & $C_{32}:=C_{16}\cup (-C_{16})=(5,32,3)$ & $C_{10}=(5,10,3)$ & $E_8$ & \ \cite{Bor-2}\\ 
 \noalign{\smallskip}\hline\noalign{\smallskip}
 & First layer of $\Lambda_{24}$ & Second layer of $\Lambda_{24}$  &  & New\\
 $24$ & $\left(\sfrac{1}{2}\right)\Lambda(2) = (24,196560,11)$ & $\left(\sfrac{1}{\sqrt{6}}\right)\Lambda(3) = (24,16773120,11)$ & $\Lambda_{24}$ &\S 5.1 \\ 
\noalign{\smallskip}\hline\noalign{\smallskip}
 & Kissing configuration of $\Lambda(2)$, tight design & Symmetrized $C_{47104}=(23,47104,7)$ &  & New\\
 ${23}$ & $C_{4600} = (23,4600,7)$ & $C_{94208} = C_{47104}\cup (-C_{47104})$ & $\Lambda_{24}$ & \S 5.2\\ 
\noalign{\smallskip}\hline\noalign{\smallskip}
 & Symmetrized $C_{891}$ &  three isometric disjoint copies of &  & New \\
 ${22}$ & $C_{1782}:=C_{891}\cup (-C_{891})=(22,1782,5)$ & $C_{2816}= (22,2816,5)$ & $\Lambda_{24}$ &\S 5.3\\ 
\noalign{\smallskip}\hline\noalign{\smallskip}
 &  tight design & Symmetrized $C_{11178}=(23,11178,7)$  &  & New \\
 ${23}$ & $C_{552}=(23,552,5)$ & $C_{22356} = C_{11178}\cup (-C_{11178})$  & $\Lambda_{24}$ &\S 5.4\\ 
\hline\noalign{\smallskip}
 & McLaughlin code & antipodal to McLaughlin code& &Rmk. 2.4 \\
 ${22}$ & $C_{275}=(22,275,4)$ & $-C_{275}=(22,275,4)$  & $\Lambda_{24}$ & \cite{Bor-2, BDHSS-Sharp} \\
\hline\noalign{\smallskip}
 & Symmetrized McLaughlin code & Symmetrized $C_{7128}=(22,7128,5)$ & & New \\
 $22$ & $C_{550}=C_{275}\cup(-C_{275})=(22,550,5)$ & $C_{14256}=C_{7128}\cup (-C_{7128})$  & $\Lambda_{24}$ & \S 5.5 \\\noalign{\smallskip}\hline\noalign{\smallskip}
 & Symmetrized Higman-Sims code & two isometric disjoint copies of &  & New \\
 ${22}$ & $C_{200}:=C_{100}\cup (-C_{100})=(22,200,3)$ & $C_{352}=(22,352,3)$ & $\Lambda_{24}$ & \S 5.6 \\ 
 \noalign{\smallskip}\hline\noalign{\smallskip}
 &  Symmetrized $C_{112}=(21,112,3)$  & four isometric disjoint copies of &  & New \\
${21}$ & $C_{224}=C_{112}\cup (-C_{112})$ & $C_{162} =(21,162,3)$ & $\Lambda_{24}$ & \S 5.7 \\ 
\hline
\end{tabular}
}
\bigskip
\caption{Universal polar dual pairs embedded in $E_8$ and $\Lambda_{24}$. By symmetrized code we mean $C\cup (-C)$. The proofs provided for configurations embedded in $E_8$ are new. 
}
\label{PULB-pairs}
\end{table}
\end{center}

\begin{remark} \label{srg} The last line in \cite[Table 1]{CK} (see also Tables \ref{EnergyULB_Table} and \ref{PolarizationULB_Table} here) is a remarkable family of spherical sharp codes associated with generalized quadrangles. They come from {\em strongly regular graphs} ${\rm srg}(v,k,\lambda,\mu)$, that is $k$-regular graphs with $v$ vertices, for which a pair of vertices have $\lambda$ neighbors if connected by an edge and $\mu$ neighbors if not connected. For any  power of prime $q$, the parameters of the said family are $v=(q+1)(q^3+1)$, $k=q(q^2+1) $, $\lambda=q-1$, and $\mu=q^2+1$. It can be embedded as a spherical code with cardinality $v$ in $\mathbb{S}^{n-1}$, where $n= q(q^2-q+1)$ (see \cite{CGS,Lev87}). For $q=2$ this is the Schl\"{a}fli configuration whose symmetrization is considered in Theorem \ref{E8-pairs-1-5}(c) and for $q=3$ this is the $C_{112}$ code considered in Subsection \ref{112-162}. This family is the only one from the known sharp codes that is not derived from regular polytopes, or codes embedded in $E_8$ or Leech lattice (except for $q=2$ and $q=3$), and will be considered in a subsequent  work.
\end{remark}

\begin{remark}\label{OtherPD}
In \cite{Bor-new} Borodachov  established that the icosahedron and the dodecahedron form a universal polar dual pair. Further, in \cite{Bor-2} he proved that the symmetrized regular simplex on $S^{n-1}$, $n\geq 3$ odd, together with the set of points on $S^{n-1}$ that form with it two distinct dot products, form a universal polar dual pair. In Table \ref{OtherPULB-pairs} we list the universal polar dual pairs we are aware of, which are not covered by Theorem \ref{main-th} (see also \cite[Table 3]{Bor-2}).
\end{remark}

In their Nature Physics paper \cite{JIBA}, Jain, Ioshu, Barg, and Albert introduced a framework to construct {\em quantum spherical codes} (used for quantum error-correction) from suitable combinations of spherical designs (see also \cite{AAB}). Universal polar dual pairs exhibit the desired properties sought in \cite{JIBA} and we believe these pairs will find applications in quantum coding. Note that a substantial number of the constructed examples in \cite[Tables A.1, A.2, B.1, B.2]{JIBA} are among our list of universal polar dual pairs as presented in Tables \ref{PULB-pairs} and \ref{OtherPULB-pairs}. It would be an interesting future investigation to explore how the potential-theoretical properties of universal polar dual pairs relate to quantum spherical codes.

The following result is a direct consequence of Theorem \ref{main-th} and Remark \ref{OtherPD}.

\begin{corollary}
The universal polar dual pairs listed in Tables \ref{PULB-pairs} and \ref{OtherPULB-pairs} verify the Claim.
\end{corollary}

\begin{center}
\begin{table} 
\scalebox{0.75}{
\begin{tabular}{|c|c|c|c|c|c|}
\hline\noalign{\smallskip}
 Dimension & \multicolumn{2}{|c|}{Maximal PULB-optimal pair of codes $(C, D)$} & Lattice & Remarks \\
\noalign{\smallskip}\hline\noalign{\smallskip}
 & regular $N$-gon & regular $N$-gon rotated by $\frac{\pi}{N}$ rad & $\mathbb{Z}^2$ $(N=4)$ & non-lattice \\
 ${2}$ & $C_{N}=(2,N,N-1)$ & $C_{N}=(2,N,N-1)$ & $A_2 \ (N=3, 6)$ & for other $N$ \\ 
\noalign{\smallskip}\hline\noalign{\smallskip}
 & icosahedron & dodecahedron &  & Platonic\\
 ${3}$ & $C_{12}=(3,12,5)$ & $C_{20}=(3,20,5)$ & Non-lattice &  solids\\\noalign{\smallskip}\hline\noalign{\smallskip}
 & 24-cell & 24-cell & & first and\\
 ${4}$ & $C_{24}=(4,24,5)$ & $C_{24}=(4,24,5)$ & $D_4$ &  second layer\\
 \noalign{\smallskip}\hline\noalign{\smallskip}
  & simplex & antipodal to simplex &sublattice & \\ 
 $n \geq 3$ & $C_{n+1}=(n,n+1,2)$ & $-C_{n+1}=(n,n+1,2)$ & of $A_n$ & \\
 \noalign{\smallskip}\hline\noalign{\smallskip}
  & regular cross-polytope & cube & & first and\\
 $n \geq 3$ & $C_{2n}=(n,2n,3)$ & $C_{2^n}=(n,2^n,3)$ & $\mathbb{Z}^n$ &$n$-th layer \\ 
 \noalign{\smallskip}\hline\noalign{\smallskip}
  &symmetrized regular simplex &  & sublattice& \\ 
 $n \geq 3, {\rm odd}$ & $C_{2n+2}=(n,2n+2,3)$ & $C_\kappa =(n,\kappa,3)$, $\kappa={n+1 \choose (n+1)/2}$ & of $A_n$ & \\
\hline
\end{tabular}
}
\bigskip
\caption{Other known universal polar dual pairs.}
\label{OtherPULB-pairs}
\end{table}
\end{center}

Our approach of embedding the universal polar dual pairs into even unimodular extremal lattices, such as the  $E_8$ and Leech lattices, and utilizing the basic properties of these lattices to find a split of a spherical code into subcodes contained in parallel hyperplanes (thereby satisfying the hypotheses of Theorem \ref{PULB} or \ref{PULB2})  is new and even for known PULB-optimal codes provides simpler and shorter proofs of their optimality and their sets of minima. Numerous new PULB-optimal configurations are found in the process as well. Moreover, the embedding allows us to prove the maximality of the respective PULB-optimal pairs utilizing the self-duality of $E_8$ and $\Lambda_{24}$. This approach empowers us to derive and identify the universal polar dual pairs in Table \ref{PULB-pairs}. We make extensive use of the interplay between the binary Golay codes and the Leech lattice. A prominent role in our proofs of the maximality of the PULB-optimal pairs is played by the Smith normal form (see Subsection \ref{lattbasics}). Our tools also include an extension of Theorem 8.2 by Delsarte-Goethals-Seidel \cite{DGS}; namely, we prove that if a spherical $\tau$-design is contained in $k \leq \tau $ parallel hyperplanes, then the subcodes in each of the hyperplanes are spherical $(\tau+1-k)$-designs in the lower dimension, a result of independent interest.

All of the examples of PULB-optimal configurations so far support the following conjecture.
\begin{conjecture}
Every PULB-optimal code $C$ whose set of universal minima $D$ is in general position; that is, $D$ is not contained in a hyperplane, forms with $D$ a PULB-optimal pair $(C,D)$ and generates a maximal PULB-optimal pair, and hence a universal polar dual pair.
\end{conjecture}

This article is structured as follows. In Section \ref{prel} we introduce the necessary preliminaries regarding Energy ULBs (see Subsection \ref{2.2}), PULBs, and the related quadratures, which play an important role in our analysis. The corresponding bounds are collected in Tables \ref{EnergyULB_Table} and \ref{PolarizationULB_Table}. Some basic facts about 
the binary Golay codes, the lattices $E_8$ and $\Lambda_{24}$, and the Smith normal form are also presented. In Section \ref{Derived} we prove an extension of   \cite[Theorem 8.2]{DGS} by Delsarte, Goethals, and Seidel along with a corollary concerning $m$-stiff configurations. Section \ref{E8} lists the maximal PULB pairs of codes (universal polar dual pairs) embedded in the $E_8$ lattice, while Section \ref{Leech} contains the results related to the maximal PULB pairs (universal polar dual pairs) found in the Leech lattice. In Section \ref{B1408} we establish the universal optimality (in the energy sense) of a code of $1408$ points in the projective space $\mathbb{RP}^{21}$, generated from a PULB-optimal code considered in Section \ref{891code}.

\section{Preliminaries} \label{prel}

\subsection{Quadrature rules and spherical designs} Here we recall the notion of spherical $T$-designs, where $T$ is a set of indices. 
\begin{definition}
\label{QuadRule}
Let $\Pi$ be a linear space of real univariate polynomials that contains the constant polynomials.  We say that $\{\alpha_i\}_{i\in I}\subset [-1,1]$ and $\{\rho_i\}_{i\in I}\subset (0,1)$,  for some finite index set $I$,   form a {\em quadrature rule exact on $\Pi$}   if
\begin{equation}\label{QR1}
f_0=\int_{-1}^1 f(t) \, d \mu_n(t) = \sum_{i\in I} \rho_i f(\alpha_i),\qquad \forall \ f \in \Pi,\end{equation}
where $d\mu_n (t)$ is given in \eqref{mu_n}.
\end{definition}
If $C$ is a $T$-design, from \eqref{defin_f} and \eqref{QRcode} it follows that for each $x\in \mathbb{S}^{n-1}$ the set of nodes given by $I(x,C)=\{\alpha_1,\ldots,\alpha_\ell\}$ with relative frequencies $\{\rho_1,\ldots, \rho_\ell\}$ forms a quadrature rule that is exact on $\mathcal{P}_T$. In particular, in this article we restrict ourselves to $T=\{1,\dots,\tau\}$ or $T=\{1,\dots,2k-1,2k+1,2k+2\}$, the cases of $\tau$-designs or of $(\tau_{\sfrac{1}{2}})$-designs ($\tau=2k-1$). 

Related to the results in this article is the work \cite{BDHSS} of Boyvalenkov et al. from 2016, where universal lower bounds on energy (or Energy ULB) were introduced. The quadrature rules of the type \eqref{QR1} with the choice of $I(x,C)$ for a point $x\in C$ play an important role and lead to Radau or Lobatto quadratures (when one or both endpoints are quadrature nodes). We introduce the setup for the Energy ULB in the next subsection. 

\subsection{Delsarte-Goethals-Seidel bound, Levenshtein $1/N$-quadrature and bound on maximal codes, and Energy ULB}\label{2.2}

As noted before, the polarization quantities $m_N^h (\mathbb{S}^{n-1})$
for general spherical codes are bounded by the corresponding 
quantities for designs $m^h_{N,T}(\mathbb{S}^{n-1})$ when they 
exist. The cardinality of spherical $\tau$-designs is bounded below by the following Fisher-type bound, cf.
\cite[Theorems 5.11, 5.12]{DGS}. If $C\subset\mathbb{S}^{n-1}$ is a $\tau$-design, 
$\tau=2k-1+\epsilon$, $k \in \mathbb{N}$, $\epsilon \in \{0,1\}$, then 
\begin{equation}
\label{DGS-bound}
|C| \geq D(n,\tau):={n+k-2+\epsilon \choose n-1}+{n+k-2 \choose n-1}.
\end{equation}

The existence of $\tau$-designs on $\mathbb{S}^{n-1}$ with a given cardinality $N \geq D(n,\tau)$ is not guaranteed and Yudin \cite{Y-designs} showed
that the bound \eqref{DGS-bound} can be improved in some cases (see also \cite{BBD99,NN}). On the other hand, Seymour and Zaslavsky \cite{SZ}
showed that there exist  $\tau$-designs on $\mathbb{S}^{n-1}$ for all large enough cardinalities, and Bondarenko, Radchenko, and Viazovska
\cite{BRV13,BRV15} solved a long-standing conjecture by proving that there exist spherical $\tau$-designs on $\mathbb{S}^{n-1}$ 
for all cardinalities $N \geq C_n \tau^{n-1}$, where $C_n$ depends only on the dimension $n$. 

Denote the maximal possible cardinality of a spherical code $C\subset \mathbb{S}^{n-1}$ of prescribed upper bound $s$ for the maximal
inner product with
$$A(n,s):=\max\{|C| \colon C \subset \mathbb{S}^{n-1}, \ x\cdot y \leq s, \,   x\neq y \in C\}.$$
In the (second) proof and investigation \cite{Lev92} of his bound on the quantity $A(n,s)$, Levenshtein utilized Gauss-type $1/N$-quadrature rules 
that we now briefly review (cf. \cite[Section 4]{Lev92}, \cite[Section 5]{Lev}). Given a
real number (possibly cardinality) $N$, there exists a unique $\tau=2k-1+\epsilon$, $\epsilon\in \{0,1\}$, such that  $ N \in (D(n,\tau),D(n,\tau+1)]$, where $D(n,\tau)$ are the Delsarte-Goethals-Seidel numbers \eqref{DGS-bound}. Let $\alpha_k=s$ be the maximal (unique in a certain sub-interval of $[-1,1)$) 
solution of the equation $N=L_{\tau}(n,t)$, where $L_{\tau}(n,t)$ is the Levenshtein function \cite[Section 5]{Lev}. Then there exist
uniquely determined quadrature nodes and nonnegative weights 
\[ -1 \leq \alpha_{1-\epsilon} < \cdots <\alpha_k < 1,\quad \rho_{1-\epsilon},\ldots,\rho_k \in \mathbb{R}^+, \]
such that the Radau/Lobatto $1/N$-quadrature (see, e.g., \cite{BBMQ,Dav1975}) holds:
\begin{equation} 
\label{defin_qf}
f_0= \frac{f(1)}{N}+\sum_{i=1-\epsilon}^{k} \rho_i f(\alpha_i), \ \ \mbox{ for all}\  f\in \mathcal{P}_{\tau},
\end{equation}
where $\mathcal{P}_\tau$ denotes the collection of polynomials of a single real variable of degree at most $\tau$.
When $\epsilon=1$, then $\alpha_0=-1$ and \eqref{defin_qf} is Lobatto quadrature; otherwise it is Radau quadrature.
The nodes $\alpha_i$, $i=1,\ldots,k$, are the roots of the equation
\begin{equation}\label{LevNodes}
P_k^{(1,\epsilon)}(t)P_{k-1}^{(1,\epsilon)}(\alpha_k) - P_{k}^{(1,\epsilon)}(\alpha_{k})P_{k-1}^{(1,\epsilon)}(t)=0, 
\end{equation}
and the weights are determined so as to meet the required accuracy. 

It turns out that   Levenshtein's $1/N$-quadrature \eqref{defin_qf} plays an important role in bounding potential energy. Given a code $C\subset \mathbb{S}^{n-1}$ with cardinality $|C|=N$, the potential energy (or $h$-energy) of $C$ is defined as 
\[ E_h (n,C):=\sum_{x,y \in C,x\not= y} h(x\cdot y).\]
The optimization quantity 
\[ \mathcal{E}(n,N;h):=\inf_{|C|=N} E_h (n,C) \]
arises  in many areas such as crystallography, material science, information theory, etc. The following theorem holds.

\begin{theorem}{\rm (}\cite[Theorems 2.3 and 3.1]{BDHSS}{\rm )}\label{THM_ULB}
Let $h$ be an absolutely monotone potential function on $[-1,1)$, $\{(\alpha_i, \rho_i)\}_{i=1-\epsilon}^{k}$, $
\epsilon \in \{0,1\}$, be the parameters of the Levenshtein's $1/N$-quadrature rule \eqref{defin_qf}, where $\tau=2k-1+\epsilon$, $\epsilon \in \{0,1\}$, is selected such that $D(n,\tau)<N\leq D(n,\tau+1)$. Then we have the universal lower bound
\begin{equation}\label{ULB1}
\mathcal{E}(n,N;h)\ge  N^2\sum_{i=1-\epsilon}^{k} \rho_i h(\alpha_i),
\end{equation}
which is attained if and only if there exists a sharp code $C$ with $|C|=N$.  
\end{theorem}

\begin{center}
\begin{table}
\scalebox{0.85}{
\begin{tabular}{|c|c|c|c|c|c|c|}
\hline
dim & Cardinality & Strength & Energy universal lower bound  \\ 
$n$ & $N$ & $\tau$ & $\mathcal{E}(n,N;h)/N$ ($h$ - absolutely monotone) \\
\noalign{\smallskip}\hline\noalign{\smallskip}
$2$ & $N=2k$ & $2k-1$ & $h(-1) + 2\sum\limits_{j=1}^{k-1} h(\cos(2j\pi/N))$  \\ 
\noalign{\smallskip}\hline\noalign{\smallskip}
$2$ & $N=2k+1$ & $2k$ & $2\sum\limits_{j=1}^k h(\cos(2j\pi/N))$  \\ 
\noalign{\smallskip}\hline\noalign{\smallskip}
$n$ & $N \leq n$ & 1 & $(N-1)h(-1/(N-1))$  \\ 
\noalign{\smallskip}\hline\noalign{\smallskip}
$n$ & $n+1$ & 2 & $nh(-1/n)$  \\ 
\noalign{\smallskip}\hline\noalign{\smallskip}
$n$ & $2n$ & 3 & $h(-1) + 2(n-1)h(0)$  \\ 
\noalign{\smallskip}\hline\noalign{\smallskip}
3 & 12 & 5 & $h(-1) + 5h(-1/\sqrt{5}) + 5h(1/\sqrt{5})$  \\ 
\noalign{\smallskip}\hline\noalign{\smallskip}
5 & 16 & 3 & $5h(-3/5) + 10h(1/5)$  \\ 
\noalign{\smallskip}\hline\noalign{\smallskip}
6 & 27 & 4 & $10h(-1/2) + 16h(1/4)$  \\ 
\noalign{\smallskip}\hline\noalign{\smallskip}
7 & 56 & 5 & $h(-1) + 27h(-1/3) + 27h(1/3)$  \\ 
\noalign{\smallskip}\hline\noalign{\smallskip}
8 & 240 & 7 & $h(-1) + 56(h(-1/2)+h(1/2)) + 126h(0)$  \\ 
\noalign{\smallskip}\hline\noalign{\smallskip}
21 & 112 & 3 & $30h(-1/3) + 81h(1/9)$  \\ 
\noalign{\smallskip}\hline\noalign{\smallskip}
21 & 162 & 3 & $56h(-2/7) + 105h(1/7)$  \\ 
\noalign{\smallskip}\hline\noalign{\smallskip}
22 & 100 & 3 & $22h(-4/11) + 77h(1/11)$  \\ 
\noalign{\smallskip}\hline\noalign{\smallskip}
22 & 275 & 4 & $112h(-1/4) + 162h(1/6)$  \\ 
\noalign{\smallskip}\hline\noalign{\smallskip}
22 & 891 & 5 & $42h(-1/2) + 512h(-1/8) + 336h(1/4)$  \\ 
\noalign{\smallskip}\hline\noalign{\smallskip}
23 & 552 & 5 & $h(-1) + 275h(-1/5) + 275h(1/5)$  \\ 
\noalign{\smallskip}\hline\noalign{\smallskip}
23 & 4600 & 7 & $h(-1) + 891(h(-1/3)+h(1/3)) + 2816h(0)$  \\ 
\noalign{\smallskip}\hline\noalign{\smallskip}
24 & 196560 & 11 & $h(-1) + 4600(h(-1/2)+h(1/2)) + 47104(h(-1/4)+h(1/4)) + 93150h(0)$  \\
\noalign{\smallskip}\hline\noalign{\smallskip}
$\frac{q(q^3+1)}{q+1}$ & $(q^3+1)(q+1)$ & $3$ & $q(q^2+1)h(-1/q)+q^4h(1/q^2)$ \\
$q$ -- a prime &&($4$, $q=2$)& \\
power &&& \\
\hline
\end{tabular}
}
\bigskip
\caption{The Energy ULB bound for known sharp codes}
\label{EnergyULB_Table}
\end{table}
\end{center}

\subsection{Sharp codes, Energy ULB, and PULB} Recall that sharp codes are spherical $\tau$-designs with $[(\tau+1)/2]$ distinct inner products among distinct points in the code. We denote sharp codes on $\mathbb{S}^{n-1}$ of cardinality $N$ and design strength $\tau$ as $C_N=(n,N,\tau)$. Tables of  sharp 
codes appear in the literature. For example, Levenshtein \cite[Table 9.1]{Lev92}, \cite[Table 6.2]{Lev} exhibited them as all known codes 
attaining his upper bound on $A(n,s)$, and Cohn and Kumar \cite[Table 1]{CK} showed that they are universally optimal, i.e., they possess, 
for their dimension $n$ and cardinality $N$, the minimum possible $h$-energy for all absolutely monotone potentials $h$. Except 
for codes from the last row\footnote{Tables \ref{EnergyULB_Table} and \ref{PolarizationULB_Table} appear also in \cite{BDHSS-Sharp}. } of Tables  \ref{EnergyULB_Table} and \ref{PolarizationULB_Table}, all known sharp codes can be found among 
the examples in the paper of Delsarte, Goethals, and Seidel \cite{DGS}.

\begin{center}
\begin{table}
\scalebox{0.65}{
\begin{tabular}{|c|c|c|c|c|c|c|}
\hline
dim & Cardinality & Strength & Polarization (PULB bound) \\ 
$n$ & $N$ & $\tau$ & $h^{(\tau+1)} \geq 0$ \\
\noalign{\smallskip}\hline\noalign{\smallskip}
$2$ & $N=2k$ & $2k-1$ & $2\sum\limits_{j=1}^k h(\cos((2j-1)\pi/N))$ \\
\noalign{\smallskip}\hline\noalign{\smallskip}
$2$ & $N=2k+1$ & $2k$ & $h(-1) + 2\sum\limits_{j=1}^k h(\cos(2j\pi/N))$ \\
\noalign{\smallskip}\hline\noalign{\smallskip}
$n$ & $N \leq n$ & 1 & $Nh(0)$ \\ 
\noalign{\smallskip}\hline\noalign{\smallskip}
$n$ & $n+1$ & 2 & $h(-1) + nh(1/n)$ \\ 
\noalign{\smallskip}\hline\noalign{\smallskip}
$n$ & $2n$ & 3 & $nh(-1/\sqrt{n}) + nh(1/\sqrt{n})$ \\ 
\noalign{\smallskip}\hline\noalign{\smallskip}
$3^{*}$ & 12 & 5 & ${3h\left(-\frac{\sqrt{1+2/\sqrt{5}}}{\sqrt{3}}\right) + 3h\left(-\frac{\sqrt{1-2/\sqrt{5}}}{\sqrt{3}}\right) + 
3h\left(\frac{\sqrt{1-2/\sqrt{5}}}{\sqrt{3}}\right) + 3h\left(\frac{\sqrt{1+2/\sqrt{5}}}{\sqrt{3}}\right)}$ \\
\noalign{\smallskip}\hline\noalign{\smallskip}
5 & 16 & 3 & $8h(-1/\sqrt{5}) + 8h(1/\sqrt{5})$ \\ 
\noalign{\smallskip}\hline\noalign{\smallskip}
6 & 27 & 4 & $h(-1) + 16h(-1/4) + 10h(1/2)$ \\ 
\noalign{\smallskip}\hline\noalign{\smallskip}
7 & 56 & 5 & $12h(-1/\sqrt{3}) + 32h(0) + 12h(1/\sqrt{3})$ \\ 
\noalign{\smallskip}\hline\noalign{\smallskip}
$8^{*}$ & 240 & 7 & $14h\left(-\frac{\sqrt{2}}{2}\right) + 64h\left(-\frac{\sqrt{2}}{4}\right) + 84h(0) + 64h\left(\frac{\sqrt{2}}{4}\right) + 
14h\left(\frac{\sqrt{2}}{2}\right)$ \\
\noalign{\smallskip}\hline\noalign{\smallskip}
21 & 112 & 3 & $56h(-1/\sqrt{21}) + 56h(1/\sqrt{21})$ \\ 
\noalign{\smallskip}\hline\noalign{\smallskip}
21 & 162 & 3 & $81h(-1/\sqrt{21}) + 81h(1/\sqrt{21})$ \\ 
\noalign{\smallskip}\hline\noalign{\smallskip}
22 & 100 & 3 & $50h(-1/\sqrt{22}) + 50h(1/\sqrt{22})$ \\ 
\noalign{\smallskip}\hline\noalign{\smallskip}
22 & 275 & 4 & $h(-1) + 162h(-1/6) + 112h(1/4)$ \\ 
\noalign{\smallskip}\hline\noalign{\smallskip}
22 & 891 & 5 & $162h(-1/\sqrt{8}) + 567h(0) + 162h(1/\sqrt{8})$ \\ 
\noalign{\smallskip}\hline\noalign{\smallskip}
23 & 552 & 5 & $100h(-\sqrt{3}/5) + 352h(0) + 100h(\sqrt{3}/5)$ \\ 
\noalign{\smallskip}\hline\noalign{\smallskip}
23 & 4600 & 7 & $275h(-\sqrt{5}/5) + 2025h(-\sqrt{5}/15) + 2025h(\sqrt{5}/15) + 275h(\sqrt{5}/5)$ \\ 
\noalign{\smallskip}\hline\noalign{\smallskip}
$24^{*}$ & 196560 & 11 & $552h\left(-\frac{\sqrt{6}}{4}\right) +11178h\left(-\frac{\sqrt{6}}{6}\right)+48600 h\left(-\frac{\sqrt{6}}{12}\right)+75900h(0)+48600 h\left(\frac{\sqrt{6}}{12}\right)+11178h\left(\frac{\sqrt{6}}{6}\right)+552h\left(\frac{\sqrt{6}}{4}\right)$ \\\noalign{\smallskip}\hline\noalign{\smallskip}
$\frac{q(q^3+1)}{q+1}$& $(q^3+1)(q+1)$ & 3 & $\frac{(q^3+1)(q+1)}{2}[h(-\sqrt{\frac{1}{q^3-q^2+q}}) + h(\sqrt{\frac{1}{q^3-q^2+q}})]$ \\
   &      &   4 ($q=2$)   &  \\
\hline
\end{tabular}
}
\bigskip
\caption{Polarization ULB \eqref{PolarizationULB} and \eqref{PolarizationULB2} for sharp codes. Codes that attain the enhanced  PULB \eqref{PolarizationULB2} (the icosahedron, the kissing configurations of $E_8$ and the Leech lattice) are indicated with *.}
\label{PolarizationULB_Table}
\end{table}
\end{center}

All sharp codes attain the Energy ULB for any absolutely monotone potential $h$ (see \cite[Theorem 3.1]{BDHSS}). Table  \ref{EnergyULB_Table}  displays the scaled energy $\mathcal{E}(n,N;h)/N$ of all sharp codes, which coincides with the ULB \eqref{ULB1} divided by the cardinality $N$. The  inner products $\alpha_i$ between a fixed point of a sharp code and the rest of the points of that code are also the Levenshtein nodes, and the numbers of occurrences of these inner products are the positive
integers $N\rho_i$, where $\rho_i$'s are the Levenshtein weights.

In \cite{BDHSS-Sharp} Boyvalenkov et al. showed that the sharp codes are also PULB-optimal. Table \ref{PolarizationULB_Table} lists the parameters for the sharp codes $C_N=(n,N,\tau)$, which satisfy either the PULB \eqref{PolarizationULB} when $h^{(\tau+1)} \geq 0$ or the enhanced PULB \eqref{PolarizationULB2} when $h^{(\tau+1)} \geq 0, h^{(\tau+2)} \geq 0, h^{(\tau+3)} \geq 0$. We note that the PULB-optimality of some of the sharp codes was earlier obtained by Borodachov (see \cite{Bor2022talk,B,Bor-new}). 

\begin{theorem} {\rm (\cite[Theorems 3.1, 4.15, 5.1]{BDHSS-Sharp})}\label{sharp-pos-level-1}
All unmarked sharp codes from Table \ref{PolarizationULB_Table} attain the bound \eqref{PolarizationULB} of Theorem \ref{PULB}, with inner products and distance distributions following the data from that table. The marked sharp codes attain the second level PULB \eqref{PolarizationULB2} of Theorem \ref{PULB2} (skip 1-add 2 framework).
\end{theorem}

\begin{remark}  \label{EvenTight} When a sharp code $C$ is a tight $2k$-design, as is the case of the regular $(2k+1)$-gon, the regular simplex, the Schl\"{a}fli code $C_{27}:=(6,27,4)$, and the McLaughlin code $C_{275}=(22,275,4)$, the universal polar dual pair is simply $(C,-C)$, which is easy to see by comparing the respective entries in Tables \ref{EnergyULB_Table} and \ref{PolarizationULB_Table} (see \cite{Bor-2,BDHSS-Sharp}). Indeed, by Theorem \ref{PULB} the nodes of the (polarization) quadrature are the zeros of $P_k^{(0,1)}$, while in Theorem \ref{THM_ULB} we have $\alpha_k=t_k^{1,0}$, the largest zero of $P_k^{(1,0)}$, so from \eqref{LevNodes} we conclude the nodes are the zeros of $P_k^{(1,0)}$. Thus, the nodes of both quadratures are symmetric to each other about the origin. As the four sharp codes are unique, the set of universal minima for $U_h(x,C)$ is clearly $-C$ and vice versa. 
\end{remark}

Unlike the case with the sharp codes, where a configuration attains the Energy ULB if and only if it is sharp, the variety of spherical codes 
attaining the PULB \eqref{PolarizationULB} or \eqref{PolarizationULB2} is much broader. Among these we list the symmetrized simplex and the unit cube in $\mathbb{R}^n$, the $24$-cell on $\mathbb{S}^3$, the shortest vectors of the $\mathbb{E}_6$ and $\mathbb{E}_7$ lattices and their mutually dual stiff configurations on $\mathbb{S}^5$ and $\mathbb{S}^6$, respectively. We should note though that any spherical code that is PULB-optimal with respect to the bound \eqref{PolarizationULB} also attain the Fazekas-Levenshtein universal bound on covering radius of spherical $\tau$-designs (see \cite{FL}, \cite{Bor-FL}, \cite[Corollary 3.9]{BDHSS-JMAA}).

\subsection{The binary Golay code}

The binary Golay code $G_{23}$, ``probably the most important of all codes for both practical and theoretical reasons"\ \cite{MWS}, 
was introduced by Golay \cite{Gol49} in 1949 in a one-page paper (together with the ternary Golay code). In coding theory terminology, 
it is a $[23,12,7]$ perfect binary code. The parameters 23, 12, and 7 correspond to the length of the codewords of $G_{23}$, its dimension as a 
subspace of $\mathbb{F}_2^{23}$, and the minimum Hamming distance between two distinct codewords, respectively.

We will need the extended Golay code\footnote{The extended Golay code was famously used in
the missions of Voyager 1 and 2 (1979-81) to send color pictures of Jupiter and Saturn.} of length 24 which is obtained from $G_{23}$
by a parity check. We will call it again a Golay code and will denote it by $\mathcal{G}$. The information below is extracted 
from \cite[Chapters 2, 16, 20]{MWS}.

The binary Golay codes can be defined in many different ways. One algebraic definition is the following. 

\begin{definition}\label{GolayCode}
Let $B_{11}$ be the $11 \times 11$ matrix with rows the vector $[11011100010]$ with $1$'s corresponding to the quadratic residues modulo $11$, and its ten cyclic right shifts (see \cite[p. 84, Figure 3.4]{CS}). Let $B_{12}$ be 
the $12 \times 12$ matrix obtained from $B_{11}$ by adding as a last $1\times 12$ row vector $[111111111110]$ and its transpose as a last column
(this operation corresponds to the parity check extension). The $12$-dimensional subspace of $\mathbb{F}_2^{24}$ with a basis of the rows 
of the $12 \times 24$ matrix $I_{12}B_{12}$, where $I_{12}$ is the $12 \times 12$ identity matrix, is the Golay code $\mathcal{G}$. 
The code $G_{23}$ can be obtained from $\mathcal{G}$ by deleting the last coordinate. 
\end{definition}

In Section \ref{Leech} we shall use a different basis, namely the rows of the $12\times 24$ matrix $GC$, which follows the Leech lattice MOG (Miracle Octad Generator) construction in \cite[pp. 132-133, Figure 4.12]{CS}
\setcounter{MaxMatrixCols}{24}
\begin{equation}\label{GC_matrix}
\setlength\arraycolsep{3.75pt}
\small
GC:=\begin{bmatrix}
 1 & 1 & 1 & 1 & 1 & 1 & 1 & 1 & 0 & 0 & 0 & 0 & 0 & 0 & 0 & 0 & 0 & 0 & 0 & 0 & 0 & 0 & 0 & 0 \\
 1 & 1 & 1 & 1 & 0 & 0 & 0 & 0 & 1 & 1 & 1 & 1 & 0 & 0 & 0 & 0 & 0 & 0 & 0 & 0 & 0 & 0 & 0 & 0 \\
 1 & 1 & 0 & 0 & 1 & 1 & 0 & 0 & 1 & 1 & 0 & 0 & 1 & 1 & 0 & 0 & 0 & 0 & 0 & 0 & 0 & 0 & 0 & 0 \\
 1 & 0 & 1 & 0 & 1 & 0 & 1 & 0 & 1 & 0 & 1 & 0 & 1 & 0 & 1 & 0 & 0 & 0 & 0 & 0 & 0 & 0 & 0 & 0 \\
 1 & 0 & 0 & 1 & 1 & 0 & 0 & 1 & 1 & 0 & 0 & 1 & 1 & 0 & 0 & 1 & 0 & 0 & 0 & 0 & 0 & 0 & 0 & 0 \\
 1 & 0 & 1 & 0 & 1 & 0 & 0 & 1 & 1 & 1 & 0 & 0 & 0 & 0 & 0 & 0 & 1 & 1 & 0 & 0 & 0 & 0 & 0 & 0 \\
 1 & 0 & 0 & 1 & 1 & 1 & 0 & 0 & 1 & 0 & 1 & 0 & 0 & 0 & 0 & 0 & 1 & 0 & 1 & 0 & 0 & 0 & 0 & 0 \\
 1 & 1 & 0 & 0 & 1 & 0 & 1 & 0 & 1 & 0 & 0 & 1 & 0 & 0 & 0 & 0 & 1 & 0 & 0 & 1 & 0 & 0 & 0 & 0 \\
 0 & 1 & 1 & 1 & 1 & 0 & 0 & 0 & 1 & 0 & 0 & 0 & 1 & 0 & 0 & 0 & 1 & 0 & 0 & 0 & 1 & 0 & 0 & 0 \\
 0 & 0 & 0 & 0 & 0 & 0 & 0 & 0 & 1 & 1 & 0 & 0 & 1 & 1 & 0 & 0 & 1 & 1 & 0 & 0 & 1 & 1 & 0 & 0 \\
 0 & 0 & 0 & 0 & 0 & 0 & 0 & 0 & 1 & 0 & 1 & 0 & 1 & 0 & 1 & 0 & 1 & 0 & 1 & 0 & 1 & 0 & 1 & 0 \\
 1 & 1 & 1 & 1 & 1 & 1 & 1 & 1 & 1 & 1 & 1 & 1 & 1 & 1 & 1 & 1 & 1 & 1 & 1 & 1 & 1 & 1 & 1 & 1 \\
 \end{bmatrix}
\end{equation}

We will need the following properties of $\mathcal{G}$ (cf. \cite[Chapter 2.6]{MWS}). 

\begin{lemma} \label{golay-properties}
\begin{itemize}
\item[(a)] The code $\mathcal{G}$ is selfdual, i.e.,$\mathcal{G}=\mathcal{G}^{\perp}$ with the Euclidean inner product over $\mathbb{F}_2$. 
\item[(b)] Every codeword of $\mathcal{G}$ has weight either $0$, $8$, $12$, $16$, or $24$.
\item[(c)] The weight distribution of $\mathcal{G}$ is $(1,759,2576,759,1)$, where the entries give
the number of words of weights $0$, $8$, $12$, $16$, and $24$, respectively. 
\item[(d)] The code $\mathcal{G}$ is distance regular; i.e., the distance distribution of $\mathcal{G}$ with respect to any of its points is the same, namely $(1,759,2576,759,1)$ for (Hamming) distances $0$, $8$, $12$, $16$, and~$24$, respectively.
\item[(e)] The subset $\mathcal{O}$ of $759$ codewords of weight $8$ in $\mathcal{G}$ 
forms a Steiner system $S(5, 8, 24)$. In other words, any binary vector of length $24$ and weight $5$ is covered by exactly one 
codeword in~$\mathcal{O}$. 
\item[(f)] The code $\mathcal{G}$ is antipodal, i.e.,$c \in \mathcal{G}$ implies $\bar{c} \in \mathcal{G}$, where $\bar{c}$ is obtained from $c$ by bit inversion.
\item[(g)] The code $G_{23}$ is cyclic.
\end{itemize}
\end{lemma}

We shall refer to codewords in $\mathcal{G}$ of weights $8$ and $12$ as octads and dodecads, respectively.  The following properties of the set of octads are well known. 

\begin{lemma} \label{octads-p}
\begin{itemize}
\item[(a)] For each coordinate position, there are exactly $253$ octads with $1$ in that position. 
\item[(b)] For each two coordinate positions, there are exactly $77$ octads with two $1$'s in these two positions. 
\item[(c)] For each three coordinate positions, there are exactly $21$ octads with three $1$'s in these three positions. 
\item[(d)] For each four coordinate positions, there are exactly $5$ octads with four $1$'s in these four positions. 
\item[(e)] For each two coordinate positions, there are exactly $176$ octads with $10$ in these two positions and this order. 
\item[(f)] For each three coordinate positions, there are exactly $56$ octads with $110$ in these three positions and this order. 
\end{itemize}
\end{lemma}

It is clear from Definition \ref{GolayCode} that the set of octads  $\mathcal{O}=\{[1^8, 0^{16}]\}$ generates $\mathcal{G}$. Indeed, the first eleven rows of the matrix $I_{12}B_{12}$ are linearly independent with $12$-th coordinate zero, so any octad with non-zero $12$th coordinate will add to a basis. Hereafter, the notation $a^\ell$ in a vector of length $n$ will mean that the number $a$ appears at the $\ell$ appropriate positions of the vector as will become clear from the context. In Section \ref{4600code} we will need to show that only a certain subset of $\mathcal{O}$ generates $\mathcal{G}$. For this purpose we introduce the notion of {\em tetrads} and {\em sextets} (see \cite[Chapter 10]{CS}). As any five positions in a binary $24$-length code in $\mathbb{F}_2^{24}$ identify uniquely an octad $\omega\in \mathcal{O}$, and as any five positions in a given octad identify the same octad, we obtain that $|\mathcal{O}|={24 \choose 5}/{8 \choose 5}=759$ as pointed out in Lemma \ref{golay-properties}(c). Furthermore, given an arbitrary set $T_1$ of four positions, each of the remaining $20$ positions combined with $T_1$ will define uniquely an octad, and these $20$ positions will be split uniquely into five disjoint subsets $T_i$, $i=2\dots,6$. The sets $T_i$ are called tetrads and the collection $\{ T_i\}_{i=1}^6$ is called a sextet. Given a certain tetrad $T$, we refer to the binary vector $v_T$ with ones exactly at the positions of $T$ as a tetrad as well. With this in mind the tetrads in the sextet $\{T_i\}_{i=1}^6$ given above, identify ${6\choose 2}=15$ distinct octads $v_{T_i}+v_{T_j}$, $1\leq i< j\leq 6$ (note that addition is in $\mathbb{F}_2$). We shall use these representations in the next lemma.

\begin{lemma}\label{O_1} The $352$ octads in $\mathcal{O}_{10}:=\{ [1,0,1^7,0^{15}]\}$ and $\mathcal{O}_{01}:=\{ [0,1,1^7,0^{15}]\}$ generate the Golay code. 
\end{lemma}
\begin{proof} In a similar fashion we define the octad subsets $\mathcal{O}_{00}$ and  $\mathcal{O}_{11}$ and let $\omega\in  \mathcal{O}_{11}$. Select four positions of the eight where $\omega$ is one such that the first position is included and the second is excluded. Consider the tetrad $T_1$ defined by these four positions and identify the corresponding sextet $\{T_i\}_{i=1}^6$. Without loss of generality assume $\omega = v_{T_1}+v_{T_2}$. Let $\omega_1:=v_{T_1}+v_{T_3}$ and $\omega_2: = v_{T_2}+v_{T_3}$. It is clear that $\omega_1 \in \mathcal{O}_{10}$ and $\omega_2 \in \mathcal{O}_{01}$ and that $\omega=\omega_1+\omega_2$. This proves that $\mathcal{O}_{11}\subset {\rm span}_{\mathbb{F}_2} (\mathcal{O}_{10},\mathcal{O}_{01})$.

If $\omega\in  \mathcal{O}_{00}$, we select any four positions where $\omega$ has $1$'s and consider the corresponding tetrad $T_1$ and denote the associated sextet $\{T_i\}_{i=1}^6$, where we again assume $\omega = v_{T_1}+v_{T_2}$. Let $T_3$ be the tetrad covering the first position. Then $\omega_1:=v_{T_1}+v_{T_3}$ will either belong to $\mathcal{O}_{11}$ or $\mathcal{O}_{10}$ depending on whether $T_3$ covers the second position  or not. The same is true about $\omega_2: = v_{T_2}+v_{T_3}$. As $\omega=\omega_1+\omega_2$, we conclude the proof of the lemma.
\end{proof}

\subsection{Lattices, generators, and the Smith normal form}
\label{lattbasics}
A discrete set $\Lambda\subset \mathbb{R}^n$ is a {\em lattice in $\mathbb{R}^n$}  if there is a finite set $\{b_1,\dots,b_k\}\subset \mathbb{R}^n$ such that
 \begin{equation*}
  \Lambda=\left\{\sum_{i=1}^k x_ib_i\mid x_1,x_2,\ldots,x_k\in \mathbb{Z}\right\}=B\Z^k,
  \end{equation*}
where $B$   is the  $n\times k$ matrix   with columns $b_1, \ldots, b_k$.    We also write  $\Lambda={\rm ispan}(b_1,\dots,b_k)$ where ${\rm ispan}$ denotes the set of linear combinations  of the vectors $b_1,\dots,b_k$ with integer coefficients.   If  ${\rm rank}(B)=n$, then $\Lambda$ is called a {\em full-rank lattice}.   If $\Lambda$ is a full-rank lattice in $\mathbb{R}^n$, then one can find $n$ linearly independent vectors $a_1,\ldots, a_n$ such that $\Lambda=A\Z^n={\rm ispan}(a_1,\ldots, a_n)$, where  $A$ is the $n\times n$ non-singular matrix with columns  $a_1,\ldots, a_n$, in which case we call $A$ a  {\em generator} for $\Lambda$.  
   
    The {\em co-volume} of a lattice $\Lambda$ with generator $A$ is $\left|\Lambda\right|:=|\det A|$   which is the volume of any choice of measurable fundamental domain $\Omega_\Lambda=\mathbb{R}^n/\Lambda$  and, in particular, the volume of the fundamental parallelotope $A[0,1)^n$.  A lattice $\Phi$ that is a subset of a lattice $\Lambda$  is called a {\em sublattice} of $\Lambda$.  The {\em index} $| \Lambda:\Phi |$ of $\Phi$ in $\Lambda$ is the cardinality of the quotient $\Lambda/\Phi$ and is easily shown to equal the ratio of the co-volumes $| \Lambda:\Phi |=|\Phi|/|\Lambda|$.   A lattice is called {\em unimodular} if its co-volume equals one; i.e., if $|\det A|= 1$.
   
 The {\em dual lattice} $\Lambda^*$ of a full-rank lattice $\Lambda$ is defined to be
   $\Lambda^*:=\{v\in \mathbb{R}^n \mid w\cdot v\in \mathbb{Z}, \ \forall \ w\in \Lambda    \}$ and has  generator   $A^{-t}:=(A^t)^{-1}$ if $\Lambda$ is generated by $A$.
   It follows that the co-volume of $\Lambda^*$ is $|\Lambda^*|=1/|\Lambda|$. A general reference to the subject is \cite{CS}.
  
    The lattices of main interest in this paper, namely the $E_8$ and Leech lattices, can be scaled to be sublattices of $\Z^n$ for $n=8,24$, respectively.   We will find the Smith normal form \cite{Smith1861, Stanley} useful for computing a generator and the co-volume of a sublattice of $\Z^n$ from a  set of linearly dependent vectors whose integer linear combinations span that sublattice.
    \begin{lemma}[Smith normal form]
   Let $B$ be an $n\times k$ matrix with integer entries and $k\ge n$.     Then there are matrices    $S, \widetilde{D}$ and $T$   with integer entries such that $S$ and $T$ are square unimodular (determinant is $\pm 1$) matrices,   $ \widetilde{D}={\rm diag}(d_1,\dots , d_n)$ is a diagonal $n\times k$ matrix with non-negative diagonal entries  $d_1,\ldots , d_n$,  and
     \begin{equation}\label{SNF}
   B=S\widetilde{D}T.
   \end{equation}
   Furthermore, there is a unique choice $\widetilde{D}$  such that $d_j\mid d_{j+1}$ for $1\le j<r$ where $r:= \rank B$.
       \end{lemma}
   The factorization  in  \eqref{SNF} (or just the unique $\widetilde{D}$) is called the {\em Smith normal form} of $B$. If  $B\in\Z^{n\times k}$ has rank $n$, then we can find a  generator $A\in \Z^{n\times n}$ for $\Lambda=B\Z^k\subset \Z^n$ as follows.  Let $B=S\widetilde{D}T$ be the Smith normal form of $B$.   Since $B$ is rank $n$, we have $k\ge n$ and so $\widetilde{D}$ can be written in a block form as $\widetilde{D}= \left[ D\, \mathbf{0}\right]$ where $D$ is $n\times n$.    It follows from $|\det T|=1$ that $T^{-1}$ also has integer entries and so $\mathbb{Z}^k= T\mathbb{Z}^k$.   Therefore,
$$\Lambda=B\mathbb{Z}^k=S\left[ D\, \mathbf{0}\right]T\mathbb{Z}^k=S\left[ D\, \mathbf{0}\right]\mathbb{Z}^k=S  D\mathbb{Z}^n, $$
showing that $A:=SD$ is a generator for $\Lambda$ with co-volume $|\Lambda|=|\det D|=d_1d_2\cdots d_n$ where $d_1, d_2, \ldots, d_n$ are the diagonal entries of $D$.

\begin{remark}\label{Mathematica}
  Our application of the Smith normal form will consist of two steps. Given a set $B$ of vectors in $E_8$ or $\Lambda_{24}$ lattices (scaled to be sublattices of $\mathbb{Z}^8$ or $\mathbb{Z}^{24}$), we consider the sublattice $\Phi:={\rm ispan}(B)$. We determine the diagonal part ${\rm diag}(d_1,\dots , d_n)$, $n=8, 24$, of the Smith normal form of $B$. To do so we identify $B$ using the coordinate representations of $E_8$  and $\Lambda_{24}$ as described in \cite{CS}, and derive the (exact) Smith normal form of $B$ via standard algorithms implemented in Mathematica (our code is available at arXiv:2512.25037 for the readers interested  in verification of this step). Since the co-volumes of $E_8$ and $\Lambda_{24}$ as embedded in $\mathbb{Z}^8$ and $\mathbb{Z}^{24}$ are $|E_8|=2^8$ and $|\Lambda_{24}|=2^{36}$, the respective indices of $\Phi$ in $E_8$ or $\Lambda_{24}$ are found to be $|E_8 : \Phi|=d_1 d_2\cdots d_8/2^8$ and $|\Lambda_{24} : \Phi|=d_1d_2\cdots d_{24}/2^{36}$ (see Remark \ref{LeechPoints} below).  
\end{remark}

 \subsection{The $E_8$ and Leech ($\Lambda_{24}$) lattices} \label{E8-Leech}

A lattice is called {\em integral} if the inner product of any
two vectors is an integer. An integral lattice is called {\em even} if the squared norms of its vectors are even integers. Even integral unimodular lattices only occur in dimensions divisible by 8, with the unique example in 8 dimensions being the $E_8$ lattice. Such lattices in $\mathbb{R}^n$ are called extremal if the squared norm of their minimal non-zero vector(s) equals $2+2\lfloor{n/24}\rfloor$. Both $E_8$ and $\Lambda_{24}$ are the unique extremal even unimodular lattices in their respective dimensions. 
 
The information below is extracted from \cite[Chapters 4, 5, 10, 14]{CS}.
 
The lattice $E_8$ consists of all points in $\mathbb{R}^8$ such that either all coordinates are integers or all coordinates are
half-integers, and the sum of all coordinates is even. The kissing configuration of $E_8$ (i.e., the set of shortest nonzero vectors) consists of $240$
vectors and can be described as follows: there are $2^7=128$ vectors $[(\pm 1/2)^8]$ with even number of negative signs and
$4 \cdot {8\choose 2}=112$ vectors with two non-zero coordinates $[(\pm 1)^2,0^6]$.
 
The number of (minimal) vectors of squared length $4$ for the Leech lattice is $196560$
and one explicit description (via the binary Golay code) is given in the next lemma. We will refer to the squared length of a vector in the Leech lattice simply as {\em norm}. 
 
\begin{lemma} \label{lem-196560} {\rm (see page 133 in \cite{CS})}
The standard coordinate representation of the $196560$ vectors of norm 4 in  $\Lambda_{24}$ is the following:
\begin{itemize}
\item[(A)] $\Lambda(2)_2 = 2^{7} \cdot 759=97152$ vectors of the form $(1/\sqrt{8})[(\pm 2)^{8},0^{16}]$, where there is an even minus signs and the $\pm 2$'s are at the $759$ octads of the Golay code;
\item[(B)] $\Lambda(2)_3 = {24 \choose 1} \cdot 2^{12}=98304$ vectors of the form $(1/\sqrt{8})[(\mp 3)^{1},(\pm 1)^{23}]$, where the upper signs follow the $2^{12}$ Golay codewords;
\item[(C)] $\Lambda(2)_4 = {24 \choose 2}\cdot 2^2  =1104$ vectors of the form $(1/\sqrt{8})[(\pm 4)^2, 0^{22}]$.
\end{itemize}
\end{lemma}

\begin{lemma} \label{lem-leech} {\rm (see Theorem 5 in Chapter 12 in \cite{CS})}
The vectors of the Leech lattice multiplied by $\sqrt{8}$, which we call Leech points, are the integer vectors $(x_1,\dots,x_{24})$ satisfying the conditions
\begin{eqnarray}\label{LeechDef}
x_i &\equiv& m \pmod 2 \nonumber\\
(x_i-m)/2 \pmod{2} &{\rm is}& {\rm in\ the\ Golay\ code}\\
\sum x_i &\equiv& 4m \pmod 8,\nonumber
\end{eqnarray}
where $m \in \{0,1\}$.
\end{lemma}
 
 \begin{remark} \label{LeechPoints} The Leech points form a sub-lattice of $\mathbb{Z}^{24}$ with co-volume $2^{36}$. Along with the Smith normal form, this fact plays a significant role in determining the indices of various sub-lattices in Section \ref{Leech}.
 \end{remark}
 
The vectors in the layers of lattices (also called shells) form (after scaling to the unit sphere) spherical designs. In the cases of the $E_8$ and Leech lattices the
first layer designs are tight; i.e., they attain the Delsarte-Goethals-Seidel bound \eqref{DGS-bound} (cf. \cite{DGS,BS81,Ban09}). Moreover, each of the layers of $E_8$ and $\Lambda_{24}$ are actually $(7_{\sfrac{1}{2}})$- and $(11_{\sfrac{1}{2}})$-designs, respectively (see \cite[Theorem 3.12]{GS}, \cite[Theorem~1]{Ve}).

\begin{lemma} \label{lem-des-7-11}
{\rm (a)} The $240$ vectors in the first layer of $E_8$
define a tight spherical $7$-design which is unique up to isometry. Moreover, the vectors in each layer of the $E_8$ lattice define a spherical  $(7_{\sfrac{1}{2}})$-design.

{\rm (b)} The $196560$ vectors in the first layer of the Leech lattice
define a tight spherical $11$-design which is unique up to isometry. Moreover, the vectors in each layer of the Leech lattice define a spherical  $(11_{\sfrac{1}{2}})$-design.
\end{lemma}
 
\begin{remark}\label{CodeLatticeNotation}
When working with codes from lattices we will sometimes use the same notation for the scaled version on $\mathbb{S}^{n-1}$ and the original
set of vectors (points) in the lattice.  Our PULB-optimal pairs found in $E_8$ and $\Lambda_{24}$ will be projected onto the same hypersphere, where by projection we mean either radial projection, orthogonal projection followed by scaling (also referenced as {\em geodesic projection}), or a composition of translations and scaling, which will become clear from the context.
\end{remark}

\section{Design strength of derived codes}\label{Derived}

The data from Tables \ref{EnergyULB_Table} and \ref{PolarizationULB_Table} encode important information about the sharp codes, which we shall use repeatedly throughout the article. In particular, every row of Table \ref{EnergyULB_Table} presents a decomposition of a given sharp code into parallel hyperplanes that are perpendicular to the radius-vector of a given point in the code. The hyperplanes identify derived codes in lower dimensions. Delsarte, Goethals, and Seidel in their seminal paper from $1977$ established that the derived codes are spherical designs of strength $\tau+1-k$ (if $\tau \geq k$), where $\tau$ is the design strength of the original code and $k$ is the number of hyperplanes excluding the original point and possibly its antipode (if it belongs to the original code) \cite[Theorem 8.2]{DGS}. 

As an illustration consider the $(24,196560,11)$-code of the shortest vectors of $\Lambda_{24}$. Fixing a point of the code, there are $5$ inner products in $(-1,1)$, so the ``kissing" configuration $(23,4600,7)$, as well as each of the other derived codes, is indeed a $7$-design. Our goal in this section is to extend this property to such a split by parallel hyperplanes that are perpendicular to any point/vector $\widetilde{x}\in \mathbb{S}^{n-1}$. 

We broaden the Delsarte-Goethals-Seidel definition of derived codes \cite[Definition 8.1]{DGS} as follows. 

\begin{definition}\label{derived_codes}
Let $n\geq 3$ be a positive integer, $C\subset \mathbb{S}^{n-1}$ be a spherical code, $\widetilde{x}\in \mathbb{S}^{n-1}$ be a fixed vector, and $\alpha\in I(\widetilde{x},C)\setminus \{-1,1\}$ be a real number. The code
\begin{equation*}
C_\alpha(\widetilde{x}):=\left\{z:= \frac{y-(y\cdot \widetilde{x})\widetilde{x}}{\sqrt{1-\alpha^2}} \, :\, y\in C, \ y\cdot\widetilde{x}=\alpha\right\}
\end{equation*} 
is called a \emph{derived code of $C$ w.r.t. $\widetilde{x}$}. As all $z\in C_\alpha(\widetilde{x})$ belong to $\{\widetilde{x}\}^\perp$ and have norm $\left|z\right|=1$, $C_\alpha (\widetilde{x})$ can be naturally identified with a code on $\mathbb{S}^{n-2}$, which we denote with $C_\alpha(\widetilde{x})$ as well. The derived codes with the largest and second largest $\alpha$ will be referred to as {\em first} and {\em second kissing configurations/layers} of $C$ relative to  $\widetilde{x}$.
\end{definition}

We now state the extension of \cite[Theorem 8.2]{DGS}, which we find to be interesting in its own right, especially applied to PULB-optimal codes\footnote{In \cite{SXY} the authors announced a similar result independently.}.

\begin{theorem}\label{derived_codes_thm}
Let $C\subset \mathbb{S}^{n-1}$ be a spherical $\tau$-design, $n\geq 3$. Suppose there is $\widetilde{x}\in\mathbb{S}^{n-1}$ and $k\leq \tau$ real numbers $-1<\alpha_1<\dots<\alpha_k<1$, such that $I(\widetilde{x},C)\setminus \{-1,1\} =\{\alpha_1,\ldots,\alpha_k\}$. Then the derived 
codes $C_{\alpha_i}(\widetilde{x})$, $i=1,2,\ldots,k$, are spherical $(\tau+1-k)$-designs on $\mathbb{S}^{n-2}$.
\end{theorem}

\begin{remark} It is illustrative to trace the conclusion about derived codes  in the above theorem in the context of Table \ref{PolarizationULB_Table}. For example, as the $196560$ shortest vectors of the Leech lattice form an $11$-design,  supported by $7$ hyperplanes, each of the derived codes in this split forms a spherical $5$-design on $\mathbb{S}^{22}$, including the $(23,552,5)$ sharp code of $276$ equiangular lines, a well-known tight $5$-design. In some cases, though, derived codes may have a higher design strength.
\end{remark}

\begin{corollary} Let $C \subset \mathbb{S}^{n-1}$, $n\geq 3$, be an $m$-stiff configuration and $\widetilde{x}\in\mathbb{S}^{n-1}$ be such that $I(\widetilde{x},C)\setminus \{-1,1\} =\{\alpha_1,\ldots,\alpha_m\}$. Then the derived codes $C_{\alpha_i}(\widetilde{x})$, $i=1,2,\ldots,m$, are spherical $m$-designs.
\end{corollary}

\begin{proof} Recall that an $m$-stiff code $C$ is a spherical $(2m-1)$-design that is embedded in $m$ parallel hyperplanes. Then $\widetilde{x}$ is a unit vector perpendicular to these hyperplanes and the subcodes in each of the hyperplanes after normalization form derived codes in $\mathbb{S}^{n-2}$, whose strength according to Theorem \ref{derived_codes_thm} is exactly $(2m-1)+1-m=m$.
\end{proof}

\begin{proof} [Proof of Theorem \ref{derived_codes_thm}] The case when $\widetilde{x}\in C$ or $-\widetilde{x}\in C$ follows from Theorem 8.2 in \cite{DGS}. So, suppose that for all $y\in C$, $y\cdot \widetilde{x} \in \{\alpha_1,\dots,\alpha_k\}\subset (-1,1)$. Without loss of generality we may assume 
$\widetilde{x}=(0,\dots,0,1)\in \mathbb{S}^{n-1}$. Let $T$ be any orthogonal transformation on $\mathbb{S}^{n-2}$. Without ambiguity, denote the corresponding orthogonal matrix with $T$ as well. Define
\[ T^\prime :=\begin{bmatrix} T& \mathbf{0}^t \\
\mathbf{0}& 1 \end{bmatrix},
\]
where $\mathbf{0}$ is the zero vector-row in $\mathbb{R}^{n-1}$ and  $\mathbf{0}^t$ denotes its transpose. Let $r\leq \ell$ be nonnegative integers and let $P_r(x_1,\dots,x_{n-1})$ be a homogeneous polynomial in $x_1,\dots,x_{n-1}$ of degree $r$. Then 
\[ Q_{\ell,r}(x_1,\dots,x_{n-1},x_n):=x_n^{\ell-r}P_r(x_1,\dots,x_{n-1}) \]
is a homogeneous polynomial in $x_1,\dots,x_n$ of degree $\ell$. In the following we shall utilize $x=(x_1,\dots,x_n)$ to abbreviate notation where convenient. Since $C$ is a $\tau$-design, for any $r\leq\tau+1-k$ and $r\leq \ell\leq \tau$, we have from \eqref{Def51} in Definition \ref{def-designs-1} 
\begin{equation} \label{DGS_1}\sum_{x \in C} Q_{\ell,r}(x)=\sum_{x \in C} Q_{\ell,r}(T^\prime x), \quad \ell=r,\dots, r+k-1.
\end{equation}
From the definitions of $P_r$ and $Q_{\ell,r}$ and the homogeneity of $P_r$ the left-hand side of \eqref{DGS_1} may be computed as
\begin{eqnarray} \label{DGS_2}
\sum_{x \in C} Q_{\ell,r}(x)&=&
\sum_{i=1}^k \alpha_i^{\ell-r} \sum_{z\in C_{\alpha_i}(\widetilde{x})} P_{r}(z\sqrt{1-\alpha_i^2})\nonumber \\
&=& \sum_{i=1}^k \alpha_i^{\ell-r} (1-\alpha_i^2)^{r/2}\sum_{z\in C_{\alpha_i}(\widetilde{x})} P_{r}(z), \quad \ell=r,\dots, r+k-1,
\end{eqnarray}
where $x=(z\sqrt{1-\alpha_i^2},\alpha_i)$. The right-hand side is found similarly
\begin{eqnarray} \label{DGS_3}
\sum_{x \in C} Q_{\ell,r}(T^\prime x)&=&
\sum_{x\in C} Q_{\ell,r}(T\, (x_1,\dots,x_{n-1}),x_n) \nonumber \\
&=& \sum_{i=1}^k \alpha_i^{\ell-r} (1-\alpha_i^2)^{r/2}\sum_{z\in C_{\alpha_i}(\widetilde{x})} P_{r}(Tz),\quad \ell=r,\dots, r+k-1,
\end{eqnarray}
where in the last equation we use the definition of $Q_{\ell,r}$, the linearity of $T$, and the homogeneity of $P_r$. 
Looking at \eqref{DGS_2} and \eqref{DGS_3}, we evaluate the $k\times k$ determinant 
\[\det ([\alpha_i^{\ell-r}(1-\alpha_i^2)^{r/2}]_{i=1,\, \ell=r}^{k\quad \ r+k-1} )=\left( \prod_{i=1}^k (1-\alpha_i^2)^{r/2}\right) \det([\alpha_i^{\ell-r}]_{i=1,\, \ell=r}^{k\quad \ r+k-1}) \not=0,\]
where in the last step we use that $|\alpha_i|<1$,  $i=1,\dots,k$, and the Vandermonde determinant is non-zero as $\alpha_i$ are pairwise distinct. We can now 
 conclude from \eqref{DGS_1}, \eqref{DGS_2}, and \eqref{DGS_3}, that for every fixed $i=1,\dots,k$ we have
\[ \sum_{z\in C_{\alpha_i}(\widetilde{x})} P_{r}(z)=\sum_{z\in C_{\alpha_i}(\widetilde{x})} P_{r}(Tz),\]
for all homogeneous polynomials $P_r$ of degree $r=0,1,\dots,\tau+1-k$.  This concludes the proof.
\end{proof}

\section{Maximal PULB-optimal pairs of codes embedded in the $E_8$ lattice}\label{E8}

In this section we describe the maximal PULB-optimal pairs (and a fortiori universal polar dual pairs) of codes embedded in the $E_8$ lattice. That these codes generate maximal PULB-optimal pairs was previously established in \cite{Bor-FL,Bor-2} (referenced as {\em mutually dual stiff configurations}, see \cite[Table 3]{Bor-2}). However, as we shall use an alternative unified approach leading to new proofs, we choose to present one of the proofs completely here and only to list the remaining four pairs. 

We remind the reader that the $E_8$ lattice is the unique positive-definite, integral, even, and unimodular lattice of rank $8$ in $\mathbb{R}^8$ (see, e.g., \cite{CS}).

\begin{theorem} \label{E8-pairs-1-5} {\rm (a)} \cite[Theorem 3.5]{Bor-new}, \cite[Theorem 4.2]{Bor-FL}  For any potential $h$ with $h^{(8)}(t)> 0$, $h^{(9)}(t)> 0$, and $h^{(10)}(t)> 0$, $t\in (-1,1)$, the codes
$C_{240}$ and $C_{2160}$ form a maximal PULB-optimal pair; i.e., they attain the bound \eqref{PolarizationULB2} as given by 
\[ 
m^h(C_{240}) \geq 14h\left(-\sqrt{2}/2\right) + 64h\left(-\sqrt{2}/4\right) + 84h(0) + 64h\left(\sqrt{2}/4\right) + 
14h\left(\sqrt{2}/2\right), \]
\[ 
m^h(C_{2160}) \geq 126h\left(-\sqrt{2}/2\right) + 576h\left(-\sqrt{2}/4\right) + 756h(0) + 576h\left(\sqrt{2}/4\right) + 
126h\left(\sqrt{2}/2\right), \]
and form a universal polar dual pair. Moreover,
$m^h(C_{240})/240=m^h(C_{2160})/2160$. 

{\rm (b)} \cite[Theorem 10.3]{Bor-2} For any potential $h$ with $h^{(6)}(t)> 0$, $t\in (-1,1)$, the codes
$C_{56}$ and $C_{126}$ form a maximal PULB-optimal pair; i.e., they attain the bound \eqref{PolarizationULB} as given by
\begin{equation} \label{C56PULB}
m^h(C_{56})\geq 12h(-1/\sqrt{3}) + 32h(0) + 12h(1/\sqrt{3}),
\end{equation}
\begin{equation}\label{C126PULB}
 m^h(C_{126})\geq 27h(-1/\sqrt{3}) + 72h(0) + 27h(1/\sqrt{3}),
 \end{equation} and form a universal polar dual pair.
 Moreover, $m^h(C_{56})/56=m^h(C_{126})/126$. 

{\rm (c)} \cite[Theorem 9.3]{Bor-2} For any potential $h$ with $h^{(6)}(t)> 0$, $t\in (-1,1)$, the codes
$C_{54}$ and $C_{72}$ form a maximal PULB-optimal pair; i.e., they attain the bound \eqref{PolarizationULB} as given by
\begin{equation}\label{C54QR}
m^h(C_{54})\geq 12h(-\sqrt{6}/4) + 30h(0) + 12h(\sqrt{6}/4),\end{equation}
\begin{equation}\label{C72QR}
 m^h(C_{72})\geq 16h(-\sqrt{6}/4) + 40h(0) + 16h(\sqrt{6}/4),
\end{equation}
and form a universal polar dual pair. Moreover, $m^h(C_{72})/72 = m^h(C_{54})/54$.

{\rm (d)} \cite[Theorem 8.5]{Bor-2} For any potential $h$ with $h^{(4)}(t)> 0$, $t\in (-1,1)$, the codes
$C_{12}$ and $C_{20}$ form a maximal PULB-optimal pair; i.e., they attain bound \eqref{PolarizationULB} as given by 
\begin{equation*}\label{C12QR}
m^h(C_{12})\geq 6h(-1/\sqrt{5}) + 6h(1/\sqrt{5}), \ \ 
 m^h(C_{20})\geq  10h(-1/\sqrt{5}) + 10h(1/\sqrt{5}),
\end{equation*} respectively, and $(C_{12},C_{20})$ is a universal polar dual pair. Moreover,
$m^h(C_{12})/12=m^h(C_{20})/20$.

{\rm (e)} \cite[Theorem 3.4]{Bor-FL} For any potential $h$ with $h^{(4)}(t)> 0$, $t\in (-1,1)$, the symmetrized Clebsch code $C_{32}:=C_{16}\cup (-C_{16})$ and the cross-polytope $C_{10}$ in $\mathbb{R}^5$  form a maximal PULB-optimal pair; i.e., they attain bound \eqref{PolarizationULB} as given by 
\begin{equation*}\label{C10QR}
m^h(C_{10})\geq 5h(-1/\sqrt{5}) + 5h(1/\sqrt{5}), \ \ m^h(C_{32})\geq  16h(-1/\sqrt{5}) + 16h(1/\sqrt{5}),
\end{equation*} respectively, and $(C_{32},C_{10})$ is a universal polar dual pair. Moreover,
$m^h(C_{10})/10=m^h(C_{32})/32$.
\end{theorem}

\begin{proof} As said above, we prove only (c). 

(c) Consider the three derived codes in the PULB \eqref{C126PULB}. Given $v\in C_{240}$, we fix $y\in \widetilde{C}_{56}$ and consider the split of the minimal vectors in $\widetilde{C}_{126}$ in three subcodes $C_1\cup D\cup C_2$, where $y\cdot z=1$ for $z\in C_1$, $y\cdot z=0$ for $z\in D$ and $y\cdot z=-1$ for $z\in C_2$. The codes $C_1$ and $C_2$ are scaled versions of the Shl\"{a}fli sharp code $C_{27}=(6,27,4)$. As $C_{126}$ is antipodal, $C_1$ and $C_2$ are antipodal to each other. Since the vectors in $D$ are orthogonal to both $v$ and $y$, they are orthogonal to the (hexagonal) $A_2$-sublattice $\langle v,y\rangle \subset E_8$ generated by these vectors, and hence are in $E_6$ (see \cite[Subsection 4.8.3, p. 125]{CS}). Actually, they are the minimal vectors of the $E_6$ lattice. If we project the vectors of $C_1$ and $C_2$ onto the six-dimensional affine subspace determined by $D$ and scale them accordingly, we will obtain the antipodal $5$-design $C_{54}=C_{27}\cup (-C_{27})=(6,54,5)$, the so-called symmetrized Shl\"{a}fli code. Similarly, we define $C_{72}:=(1/\sqrt{2})D=(6,72,5)$.

\begin{figure}[htbp]
\centering
\includegraphics[width=2.9 in]{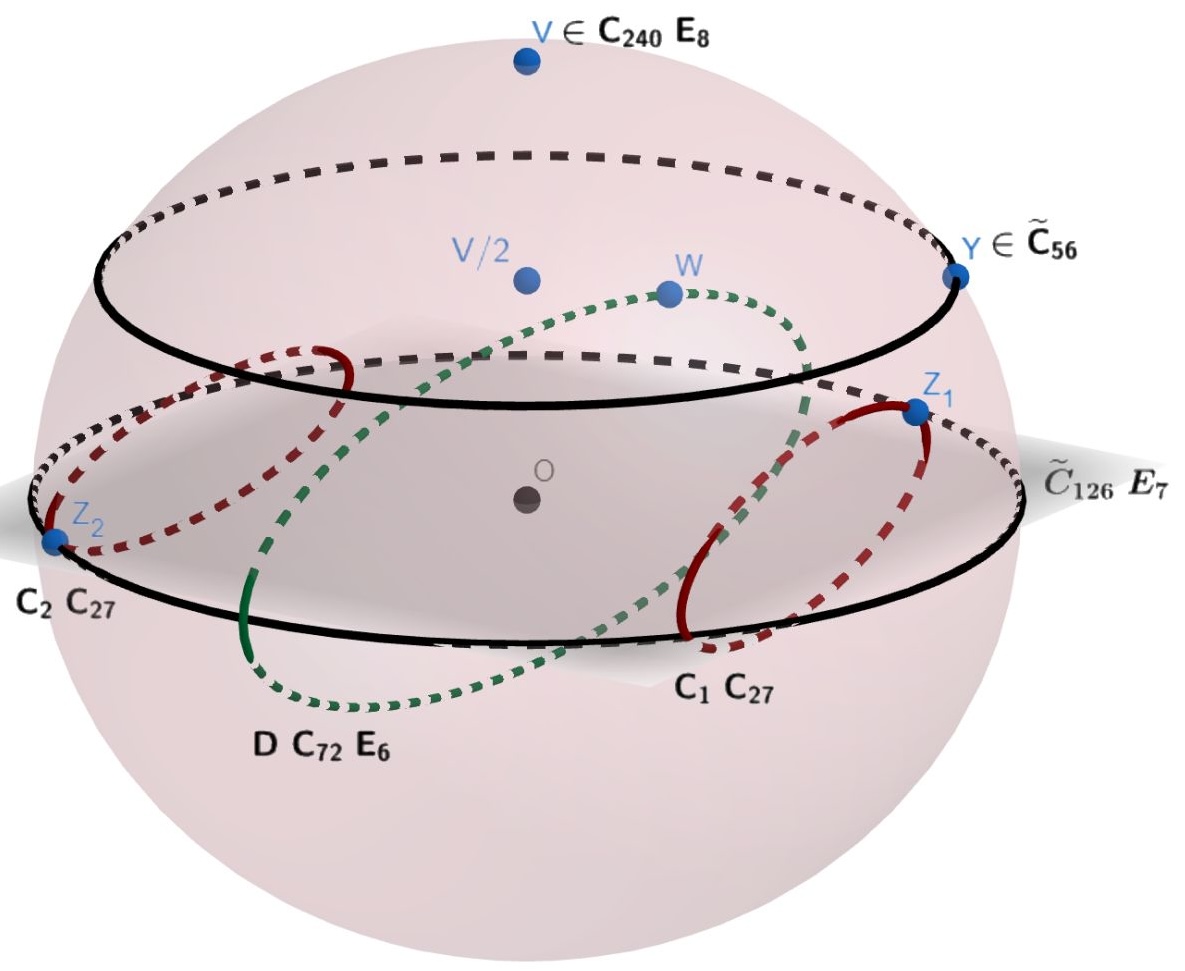}
\caption{The PULB-optimal pair $(C_{54},C_{72})$ embedded in the $E_8$ lattice.}
\label{fig:4}
\end{figure}

Let $z\in C_1$ be fixed and $w_1$ be a nearest vector from $D$ to $z$; i.e.,$|w_1-z|=\sqrt{2}$. Then $z\cdot w_1=1$. Note that $-z\in C_2$, so a nearest to $-z$ vector $w_2$ from $D$ will satisfy $(-z)\cdot w_2=1$. This shows that for any $w\in D$, $z\cdot w\in \{-1,0,1\}$. 

To determine the cosines between vectors in $C_{54}$ and $C_{72}$, let us denote by $m_1$ the center of mass of $C_1$. Then $m_1=\sqrt{2/3}(y-v/2)/|y-v/2|=(2/3)(y-v/2)$. Given $w\in D$, let $\phi$ be the angle between $w$ and $z-m_1=z-(2/3)(y-v/2)$, $z\in C_1$. We compute 
\[ \cos(\phi) =\frac{w\cdot (z-(2/3)(y-v/2))}{|w||z-(2/3)(y-v/2)|}=\frac{\sqrt{6}w\cdot z}{4}\quad \in \quad \left\{-\frac{\sqrt{6}}{4},0,\frac{\sqrt{6}}{4} \right\},\]
where we use the fact that $w\cdot v=w\cdot y=0$ and $|z-m_1|=2/\sqrt{3}$. We remind the reader that $P_3^{(6)}(t)=t(8t^2-3)/5$. Clearly, if $z\in C_2$, then the inner products $z\cdot w$ will be in the same set of zeros of $P_3^{(6)}(t)$. As the Schl\"{a}fli configuration $C_{27}$ is a spherical $4$-design, the symmetrized configuration $C_{54}=C_{27}\cup -C_{27}$ will be an antipodal spherical $5$-design. We now conclude that the points of $C_{54}$ are universal minima of $C_{72}$ and vice versa. The corresponding PULBs are 
\eqref{C54QR} and \eqref{C72QR},  
both holding true for any potential $h$ with $h^{(6)}(t)> 0$, $t\in (-1,1)$. 

 To prove that the points of $C_{54}$ comprise all of the universal minima of $U_h(x,C_{72})$, we assume $u$ is a universal minimum projected onto the circumscribing hypersphere of $D$. Consider $L:={\rm ispan}(D,v, y)$, the sub-lattice formed by the vectors of $D$ along with $v$ and $y$. Let 
\[ \widetilde{u}:= m_1+\frac{2}{\sqrt{3}}\frac{u}{\sqrt{2}}=\frac{2y-v}{3}+\frac{2}{\sqrt{6}}u \]
be the projection of $u$ onto the hypersphere containing $C_1$. Note that for $w\in D$ we have $u\cdot w\in \{\pm \sqrt{6}/2,0\}$ and $m_1 \cdot w = (2y-v)/3\cdot w=0$, which yields that $ \widetilde{u} \cdot w\in \{\pm 1,0\}$. Clearly, $\widetilde{u}\cdot v=0$. We find that $\widetilde{u}\cdot y= (\widetilde{u}-m_1)\cdot y+m_1 \cdot y=1$ 
, from which we conclude that $\widetilde{u}\in L^*$, where $L^*:=\{ x\in \mathbb{R}^8 \, :\, x\cdot a\in \mathbb{Z}, a\in L\}$ is the dual lattice of the sub-lattice $L$. The Smith normal form is ${\rm diag}(1,2^6,12)$ which allows us to find the index of the sub-lattice $L$, namely $|E_8 : L|=3$, from which we conclude that $|L^* : E_8|=3$. Since $m_1\cdot w=0$, $w\in D$, $m_1 \cdot v=0$, and $m_1\cdot y=1$, we derive $m_1 \in L^*\setminus E_8$. This implies that $L^*$ splits into three cosets
\[ L^*= E_8 \cup \left( m_1+E_8 \right) \cup (2m_1 + E_8).\]

If $\widetilde{u}\in E_8$, then $\widetilde{u} \in C_{1}$. As $|\widetilde{u}-m_1|^2=4/3$, it is impossible for $\widetilde{u}$ to be in $ (m_1+E_8)$. Finally, as $|\widetilde{u}-2m_1|^2=2$, $\widetilde{u}$ may belong to the class $(2m_1+E_8)$. Indeed, if $\widetilde{u}=2m_1+q$, $q\in E_8$, then $-q=2m_1-\widetilde{u}$; i.e.,$q\in -C_{1}=C_2$. As the points of $C_{27}$ and $-C_{27}$ already identify universal minima of $U_h(x,C_{72})$, we have found all of them.

We next derive that the points of $C_{72}$ constitute all of the universal minima of $U_h(x,C_{54})$. Without loss of generality we may assume that $v=[1,1,0^6]$, $y=[1,0,1,0^5]$. Then $C_{1}=\{ z_1,\dots,z_{27}\}$ consists of the vector $[1,-1,0^6]$, $10$ vectors $[0,0,1,(\pm 1)^1,0^4]$, and $16$ vectors $[\frac{1}{2},-\frac{1}{2},\frac{1}{2},(\pm \frac{1}{2})^5]$ with an even total number of minus signs. We also have $D=\{ w_1, \dots, w_{72}\}$ as the following set of vectors; $40$ vectors of type $[0,0,0,(\pm 1)^2,0^3]$, $16$ vectors $[\frac{1}{2}, - \frac{1}{2},-\frac{1}{2},(\pm \frac{1}{2})^5]$ with an even number of minus signs, and $16$ vectors $[-\frac{1}{2},  \frac{1}{2},\frac{1}{2},(\pm \frac{1}{2})^5]$ again with an even number of minus signs. The code $C_{54}$ is obtained as a scaled projection of $C_{1}$ and $-C_{1}$ onto the hyperspace orthogonal to $v$ and $y$. For our purposes though, we shall embed $C_{54}$ into the hypersphere circumscribing $C_{1}$ and also call it $C_{54}$; i.e.,$C_{54}=C_{1}\cup (2m_1-C_{1})$ (note that $2m_1 - C_{1}$ is antipodal to $C_{1} $ w.r.t. its center of mass $m_1$).

Let $u$ be a universal minimum of $U_h(x,C_{54})$ (as embedded above) and let $\widetilde{u}:=\sqrt{3/2}(u-m_1)$. Then $\widetilde{u}$ lies on the hypersphere $S_2:=S(0,\sqrt{2})\cap \left( {\rm span}\{ v,y \} \right)^\perp$ circumscribed about $D=E_8 \cap S_2$. Since $u$ is a universal minimum of $U_h (x,C_{54})$, we have
\[ \frac{u-m_1}{2/\sqrt{3}} \cdot \frac{z_i-m_1}{2/\sqrt{3}} \in \left\{ \pm \sqrt{\frac{3}{8}},0\right\},\]
from which we obtain
\[ (u-m_1) \cdot (z_i-m_1)=(u-m_1)\cdot z_i \in \left\{ \pm \sqrt{\frac{2}{3}},0\right\}, \quad i=1,\dots,27.\]
Therefore, $\widetilde{u}\cdot z_i \in \{ \pm 1,0\}$. Since $\widetilde{u}\cdot v=0$, we derive that $\widetilde{u}\in L^*$, the dual of the sub-lattice $L:={\rm ispan}(v,C_{1})$. Observe that $v\cdot z_i=0$ and $v\cdot v=2$, so if $a\in L$ then $v\cdot a \in 2\mathbb{Z}$, which yields that $L$ is a proper sub-lattice of $E_8$. As the columns of the matrix $E$, 
\setcounter{MaxMatrixCols}{8}

\[
\setlength\arraycolsep{3.75pt}
\small
E:=\begin{bmatrix}
1 & 1 & 0 & 0 & 0 & 0 & 0 & 1/2 \\
-1 & 1 & 0 & 0 & 0 & 0 & 0 & -1/2\\
0 & 0 & 1 & 1 & 1 & 1 & 1 & 1/2\\
0 & 0 & 1 & 0 & 0 & 0 & 0 & -1/2\\
0 & 0 & 0 & 1 & 0 & 0 & 0 & 1/2\\
0 & 0 & 0 & 0 & 1 & 0 & 0 & 1/2\\
0 & 0 & 0 & 0 & 0 & 1 & 0 & 1/2\\
0 & 0 & 0 & 0 & 0 & 0 & 1 & 1/2\\
 \end{bmatrix}
,\]
belong to $\{v,C_{1}\}$, the sub-lattice generated by the $8$ column vectors will be a sub-lattice of $L$. However, since $\det (E)=2$, we conclude this sub-lattice is exactly $L$, and moreover, we find the index $|L^* : E_8| =2$. The vector $v/2$ does not belong to $E_8$ and belongs to $L^*$, thus $L^*=E_8 \cup (v/2+E_8)$. We compute that $|\widetilde{u}-v/2|^2=2+2/4=5/2\notin 2\mathbb{Z}$, so $\widetilde{u}$ does not belong to the coset $(v/2+E_8)$, which implies that $\widetilde{u}\in E_8$, or $\widetilde{u}\in D$. This concludes the proof of the maximality of the PULB-optimal pair $(C_{72},C_{54})$. 
\end{proof}

\section{PULB-optimal pairs of codes embedded in the Leech lattice}\label{Leech}

In this section we consider the universal polar dual pairs arising from sharp codes embedded in the Leech lattice $\Lambda_{24}$, see Table \ref{PULB-pairs}. Embedding the codes in $\Lambda_{24}$ allows us to find an alternative (simpler) proof of their PULB-optimality, utilizing that the Leech lattice is the unique even unimodular extremal lattice in $\mathbb{R}^{24}$ and its shortest vectors form a spherical $11\sfrac{1}{2}$-design. We shall identify the maximal PULB-optimal pairs (and thus universal polar dual pairs) and in the process discover new PULB-optimal codes. We start   by showing that as in the $E_8$ case, the first and the second layers of the Leech lattice form a PULB-optimal pair (when projected onto the unit sphere).

\subsection{The maximal PULB-optimal pair of the first and second layers of the Leech lattice.}  \label{level-1-leech}Recall that the Leech lattice  $\Lambda_{24}$ is the unique even unimodular extremal lattice of rank $24$ in $\mathbb{R}^{24}$ with no minimal vectors (i.e., its shortest vectors have length $2$). Using Conway-Sloane's notation \cite[Subsection 4.11, p. 131]{CS}, let us denote with $\Lambda(k)$, $k=2,3,\dots$, the layers of the lattice with vectors with squared length (or norm in lattice terminology) $2k$. Then $\Lambda(2)$ are the $196560$ shortest vectors of the Leech lattice, whose length is $2$. The second layer consists of the lattice vectors of length $\sqrt{6}$ and is denoted by $\Lambda(3)$. We know that $|\Lambda(3)|=16773120$. Projecting radially on the unit sphere we get $(1/2)\Lambda(2)=(24,196560,11)$ and $(1/\sqrt{6})\Lambda(3)=(24,16773120,11)$. 
 
 \begin{figure}[htbp]
\includegraphics[width=3.8 in]{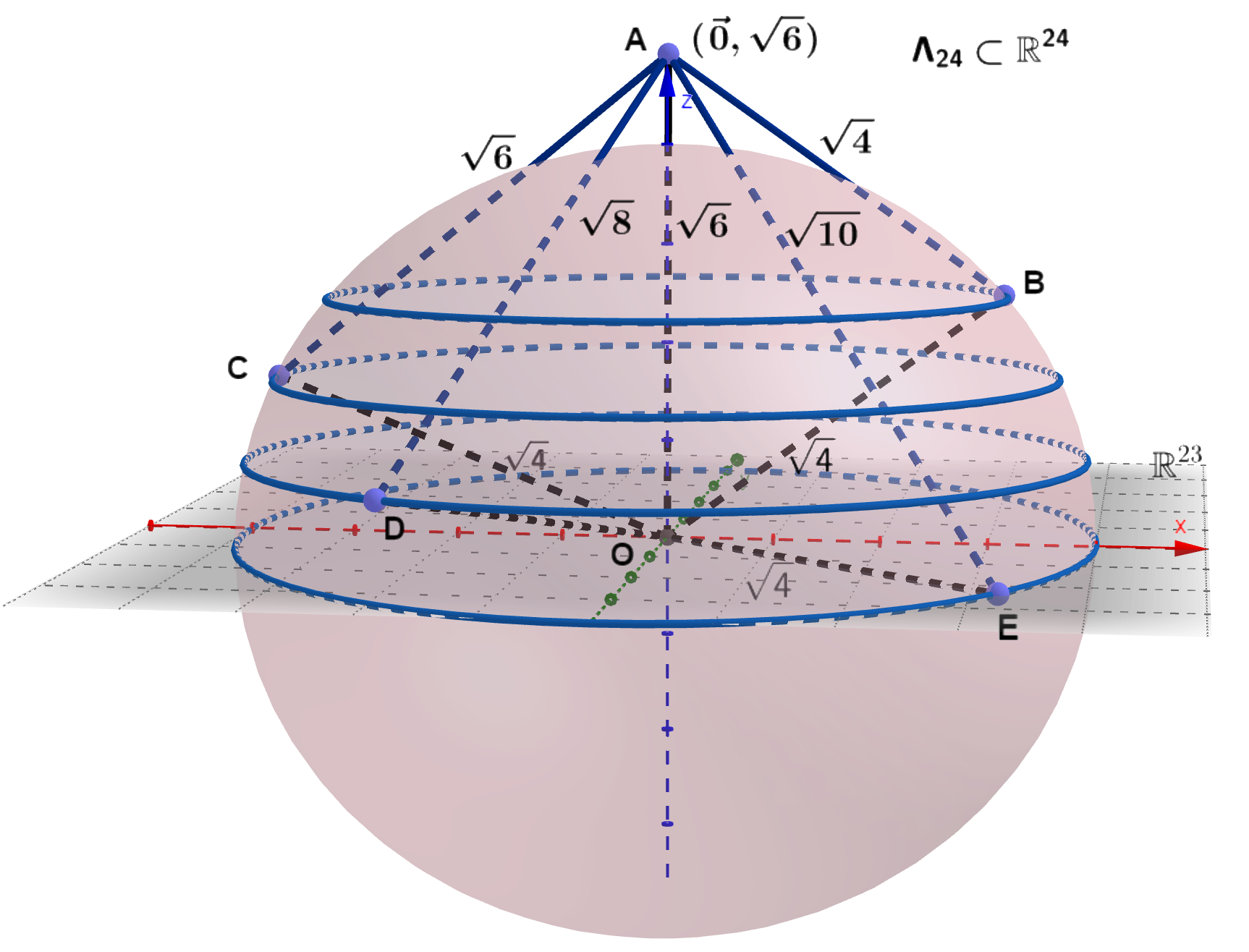}
\caption{Leech lattice - the PULB-optimal pair $\Big( \Lambda(2),\Lambda(3)\Big)$.}
\label{fig-Leech}
\end{figure}

Let $O$ be the origin and  select a point $A\in \Lambda(3)$. Let $B, C, D, E \in \Lambda(2)$ whose distance from $A$ is $\sqrt{4}, \sqrt{6}, \sqrt{8}, \sqrt{10}$ , respectively.This is illustrated  in Figure \ref{fig-Leech}. The corresponding circles illustrate hyperspheres obtained by intersecting $S(O,\sqrt{4})$ with $S(A,\sqrt{4})$, $S(A,\sqrt{6})$, $S(A,\sqrt{8})$, and $S(A,\sqrt{10})$, respectively, where $S(y,r):=\{x\in \mathbb{R}^{24}\, :\, |x-y|=r\}$. The Law of Cosines applied to $\triangle BOA$ implies that $\cos (\angle BOA)=\sqrt{6}/4$. Similarly, the Law of Cosines applied to triangles $\triangle COA, \triangle DOA$ yields that $\cos (\angle COA)=\sqrt{6}/6$ and $\cos (\angle DOA)=\sqrt{6}/12$. Obviously, $\triangle EOA$ is a right triangle and $\cos (\angle EOA)=0$. By symmetry we obtain that $A$ splits $\Lambda(2)$ into seven hyperplanes with cosine distribution
\begin{equation}\label{CosLeech} 
\cos(\phi) \in \left\{ \pm \sqrt{6}/4, \pm \sqrt{6}/6, \pm \sqrt{6}/12,  0 \right\},\end{equation}
which are exactly the roots of $P_7^{(24)}(t)+(4/81)P_5^{(24)}(t)$ and, simultaneously, the nodes of the PULB \eqref{PolarizationULB2} for the sharp code $(24,196560,11)=(1/2)\Lambda(2)$ from Table \ref{PolarizationULB_Table}. As the weights of the skip 1-add 2 quadrature are uniquely determined from the Lagrange basis polynomials, we have established that the projection of $A$ onto the unit sphere will be a universal minimum for the sharp code $(1/2)\Lambda(2)$, which implies that all points in $(1/\sqrt{6})\Lambda(3)$ are universal minima for the discrete potential $U_h(x,(1/2)\Lambda(2))$. 

To determine the number of universal minima of $U_h(x,(1/2)\Lambda(2))$, we note that from Theorem \ref{derived_codes_thm} the facet corresponding to such a minimum is a $5$-design of $552$ points on $\mathbb{S}^{22}$, and hence a tight design, namely  the sharp code $(23,552,5)$. The number of such facets of the Leech contact polytope, formed by the convex hull of $\Lambda(2)$, was found in \cite[p. 907]{DSV} to be $|\Lambda(3)|$. An alternative argument is that the $(23,552,5)$-facet embedded in the Leech lattice is antipodal with respect to its center of mass and the sum of any two antipodal vectors of such a facet determines  uniquely a point in $\Lambda(3)$, whose projection coincides with the chosen minimum.

\begin{remark}
We note that as each $(23,552,5)$ facet is identified with a point in $\Lambda(3)$ and thus the stabilizer of the said facet (and point) is the Conway group $Co_3$. Since $Co_0$ acts transitively on $\Lambda(3)$, the number of universal minima can also be expressed as
\begin{equation*}\label{2-3}\frac{|Co_0|}{ | Co_3 |}=\frac{2^{22} .\, 3^9 .\, 5^4 .\, 7^2 .\, 11\, .\, 13\, .\, 23}{2^{10} .\, 3^7 . \, 5^3 .\, 7\, .\, 11\, .\, 23}= 2^{12} .\, 3^2 . \, 5\, .\, 7\, .\, 13=16773120.\end{equation*}
\end{remark}

On the other hand the code $(1/\sqrt{6})\Lambda(3)$ has the same vanishing moments as $(1/2)\Lambda(2)$ and is therefore an antipodal 
spherical $11\sfrac{1}{2}$-design, so the same skip 1-add 2 bound \eqref{PolarizationULB2} holds with equality for it. Clearly, the cosine distribution of the angles between vectors in $\Lambda(2)$ and $\Lambda(3)$ will be the same as \eqref{CosLeech} and a similar argument shows that every point in $(1/2)\Lambda(2)$ will be a universal minimum of $U_h(x,(1/\sqrt{6})\Lambda(3))$. 

To show that all minima of  $U_h(x,(1/\sqrt{6})\Lambda(3))$ are the points  in $(1/2)\Lambda(2)$ we proceed as follows. Suppose $W\in \mathbb{S}^{23}$ is a universal minimum of $U_h(x,(1/\sqrt{6})\Lambda(3))$. We shall show that $U=2W$ is a point in the Leech lattice, and hence belongs to $\Lambda(2)$. We observe that the inner products of $U$ with all points of $\Lambda(3)$ are integers. Since $\Lambda(3)$ spans the Leech lattice as established by the lemma below, we conclude that $U$ belongs to the Leech lattice (and to $\Lambda(2)$) by the self-duality property of the Leech lattice.

\begin{lemma}\label{Span} The vectors in $\Lambda(3)$ span the Leech lattice; i.e.,${\rm ispan}(\Lambda (3))=\Lambda_{24}$ (see Subsection \ref{lattbasics} for the definition of ispan).
\end{lemma}
\begin{proof} To simplify the argument, we shall multiply all vectors by $\sqrt{8}$. We recall the description of the Leech points (see Lemma \ref{lem-leech}), and proceed with exhibition of the vectors in $\sqrt{8}\Lambda(3)$ (see, e.g., \cite[Table 4.13]{CS}):
\begin{itemize}
\item[(i)] $\Lambda(3)_2 = 2^{11} . 2576$ vectors of the form $[(\pm 2)^{ 12},0^{12}]$, where there is an even number of "$-$" signs and the $\pm 2$'s are at the $2576$ dodecads of the Golay code;
\item[(ii)] $\Lambda(3)_3 = {24 \choose 3} . 2^{12}$ vectors of the form $[(\mp 3)^{3},(\pm 1)^{21}]$, where the upper signs follow the $2^{12}$ Golay codewords;
\item[(iii)] $\Lambda(3)_4 = 2^{8} . 759 . 16$ vectors of the form $[(\pm 4)^1, (\pm 2)^{8},0^{15}]$, where the $\pm 2$'s follow the $759$ octads of the Golay code with odd number "$-$" signs;
\item[(iv)] $\Lambda(3)_5 = 24 . 2^{12}$ vectors of the form $[(\pm 5)^{1},(\pm 1)^{23}]$, where the plus signs follow the $2^{12}$ Golay codewords.
\end{itemize}
Let $(x_1,\dots,x_{24})$ be arbitrary Leech point (i.e., a point which satisfies \eqref{LeechDef}). By adding $[5,1^{23}]$ if needed, we may assume $m=0$ in Lemma \ref{lem-leech}. As the octads span the Golay code we can add vectors of type $[(\pm 4)^1, (\pm 2)^{8},0^{15}]$ and assure $4| x_i$, $i=1,\dots,24$.

There are $21$ Golay octads that start with three $1$'s; i.e.,of the type $[1,1,1,\dots]$. Selecting such a Golay word, we have that $[5,1,1,\dots]-[-3,-3,-3,\dots]=[8,4,4,0^{21}]\in {\rm ispan}(\sqrt{8}\Lambda(3))$. Similarly, $[4,8,4,0^{21}]\in {\rm ispan}(\sqrt{8}\Lambda(3))$, from which we get $[4,-4,0^{22}]\in {\rm ispan}(\sqrt{8}\Lambda(3))$. Of course, this can be done for any triple of indexes and hence all permutations of $[4,-4,0^{22}]$ belong to $ {\rm ispan}(\sqrt{8}\Lambda(3))$. By adding integer multiples of these vectors we can assure it suffices to find integer representations of $(x_1, \ldots,x_{24})$ when only one coordinate $x_i$ is non-zero. Observe, that this coordinate has to be divisible by $8$. 
\vskip 0mm
On the other hand, if we select vectors $V, U \in \Lambda(3)_4$, which differ only in the sign of $4$, then $V-U=[8^1,0^{23}]\in  {\rm ispan}(\sqrt{8}\Lambda(3))$, which completes the proof.
\end{proof}

This implies the following theorem.

\begin{theorem}\label{LeechPULBpair} For any potential $h$ with $h^{(12)}(t)> 0$, $h^{(13)}(t)> 0$ , and $h^{(14)}(t)> 0$, $t\in (-1,1)$, the code 
$\left(1/\sqrt{6}\right)\Lambda(3)\subset \mathbb{S}^{23}$ is PULB-optimal; i.e., it attains the bound \eqref{PolarizationULB2}
\begin{equation}\label{Lambda3PULB}
\begin{split}
m^h((1/\sqrt{6})\Lambda(3)) =& 47104[h(\sqrt{6}/4)+h(-\sqrt{6}/4)] +953856[h(-\sqrt{6}/6)+h(\sqrt{6}/6)] \\
&+4147200 [h(-\sqrt{6}/12)+h(\sqrt{6}/12)]+6476800h(0)
\end{split}
\end{equation}
with universal minima at the points of $\left(1/2\right)\Lambda(2)$.

Moreover, the first and the second layers of the Leech lattice, $\Lambda(2)$ and $\Lambda(3)$, when projected onto the unit sphere $\mathbb{S}^{23}$ form a maximal PULB-optimal pair, and thus, a universal polar dual pair, with normalized discrete potentials achieving the same extremal value
\begin{equation*}
\frac{m^h(\left(1/2\right)\Lambda(2))}{196560}= \frac{m^h(\left(1/\sqrt{6}\right)\Lambda(3))}{16773120}.\end{equation*}
\end{theorem}

\subsection{The maximal PULB-optimal pair generated by the sharp code $C_{4600}$ and the second kissing configuration $C_{47104}$ of $\Lambda(2)$.}\label{4600code}  Let $K_1$ denote the sharp code $C_{4600}=(23,4600,7)$ on $\mathbb{S}^{22}$. A scaled version of the code, which we refer to as $K_1$ as well, can be embedded in $\Lambda_{24}$ as the closest  $4600$ vectors in $\Lambda(2)$ to a fixed vector $A\in \Lambda(2)$. If $B\in K_1$ is a sample point then $|A-B|=2$. Let $K_2:=C_{47104}=(23,47104,7)$ be the second kissing configuration to $A$, that is the $47104$ points $C\in  \Lambda(2)$ at a distance $\sqrt{6}$ from $A$; i.e., $|A-C|=\sqrt{6}$. Let $\overline{K}_2$ be the symmetric image of $K_2$ with respect to the center of mass $m_1$ of $K_1$, and define $\widetilde{K}_2:=K_2\cup \overline{K}_2$.  In this subsection, we shall establish that $(K_1,\widetilde{K}_2)$ is a maximal PULB-optimal pair (we shall refer to $K_1$ and $\widetilde{K}_2$ as both, vectors in hyperspheres in the Leech lattice and spherical codes normalized to lie on $\mathbb{S}^{22}$ by projecting $K_1$ and $\widetilde{K}_2$ onto the equatorial hyperplane in $\mathbb{R}^{24}$ with $A$ as a North Pole and scaling accordingly). 

\begin{figure}[htbp]
\includegraphics[width=3.6 in]{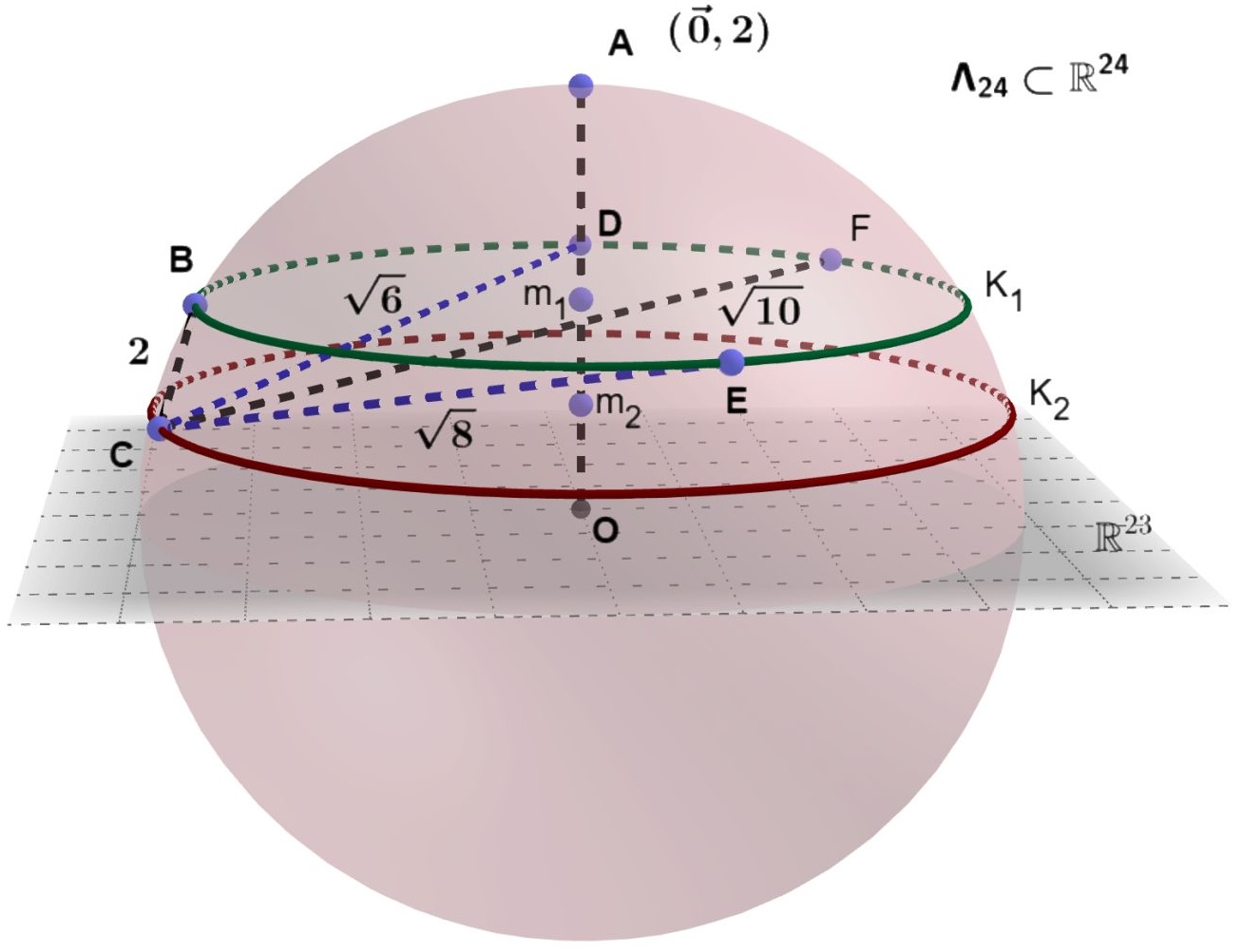}
\caption{First and second kissing arrangements $K_1$ and $K_2$ on $\Lambda(2)$, where for convenience, $A=[0^{23},2]$.}
\label{fig-4600}
\end{figure}
 
We start by showing that points in $K_2$ are universal minima of $U_h(x,K_1)$.
Let $C\in K_2$ be fixed and let us select $B, D, E, F \in K_1$, such that $|C-B|=2$, $|C-D|=\sqrt{6}$, $|C-E|=\sqrt{8}$ and $|C-F|=\sqrt{10}$ (see Figure \ref{fig-4600}). Recall the Energy ULB row for the $(24,196560,11)$ sharp code from Table \ref{EnergyULB_Table}
\[ \frac{\mathcal{E}(24,196560;h)}{196560} = h(-1) + 4600[h(-1/2)+h(1/2)] + 47104[h(-1/4) + h(1/4)]+93150h(0) .\]
Applying \cite[Theorem 8.2]{DGS} (see also Theorem \ref{derived_codes_thm}), we conclude that the derived codes are $7$-designs. The formula encodes the respective cardinalities of $|K_1|=4600$ and $|K_2|=47104$, but also the coordinates of their respective centers of mass, namely $m_1=A/2$ and $m_2=A/4$. The vectors $C-m_2$ and $B-m_1$ may be viewed as positioned in the Equatorial hyperplane with respect to $A$. We find the respective lengths as $|B-m_1|=\sqrt{3}$ and $|C-m_2|=\sqrt{15}/2$. We can now compute the cosine of the angle between $B-m_1$ and $C-m_2$ as follows
\[ \cos(B-m_1,C-m_2)=\frac{(B-m_1)\cdot(C-m_2)}{ |B-m_1|  |C-m_2|}=\frac{B\cdot C-(A\cdot B) /4}{\sqrt{3}(\sqrt{15}/2)}=\frac{\sqrt{5}}{5}.\]
Here we used that $A\cdot B=B\cdot C=2$.
Similarly, using this time the fact that $A\cdot D=2$ and $C\cdot D=1$ (recall that $|C-D|=\sqrt{6}$) we find that 
\[ \cos(D-m_1,C-m_2)=\frac{(D-m_1)\cdot(C-m_2)}{ |D-m_1|  |C-m_2|}=\frac{D\cdot C-(A\cdot D) /4}{\sqrt{3}(\sqrt{15}/2)}=\frac{\sqrt{5}}{15}.\]
Utilizing that $C\cdot E=0$ and $C\cdot F=-1$ we conclude analogously that $\cos(E-m_1,C-m_2)=-\sqrt{5}/15$ and $\cos(F-m_1,C-m_2)=-\sqrt{5}/5$. 

We remark that there exist no lattice points on $K_1$ whose distance to $C$ is greater than $\sqrt{10}$. Indeed, if $G\in K_1$ were a point for which $12\leq |C-G|^2=|C|^2+|G|^2-2C\cdot G=4+4-2C\cdot G$ we would get that $C\cdot G\leq -2$. Since $K_1$ is antipodal with respect to $m_1$, let $H\in K_1$ be the vector antipodal to $G$ in the hypersphere containing $K_1$; i.e.,$H=A/2+(A/2-G)=A-G$. Since $C\cdot A=1$, we have $C\cdot H\geq 3$ or $|C-H|^2\leq 2$. Since the Leech lattice has no minimal vectors of length smaller than $2$, we get $C=H$ which is absurd. Thus, $K_1$ is $4$-stiff.

Since the cosine distribution between points in $K_1$ and $K_2$ is $( \pm \sqrt{5}/15, \pm \sqrt{5}/5 )$, the zeros of the Gegenbauer polynomial 
$P_4^{(23)}(t)=(45t^2-1)(5t^2-1)/176$ and the quadrature nodes of \eqref{PolarizationULB}, we can obtain the quadrature 
weights and determine the PULB \eqref{PolarizationULB}
\begin{equation*}\label{4600PULB} m^h (K_1)\geq 275h(-\sqrt{5}/5) + 2025h(-\sqrt{5}/15) + 2025h(\sqrt{5}/15) + 275h(\sqrt{5}/5), \end{equation*}
valid for any potential $h$ with $h^{(8)}(t)> 0$, $t\in (-1,1)$. 
Thus, the projections of the points of $K_2$ onto the hyperplane containing $K_1$, normalized to be unit vectors are universal minima of $U_h(x,K_1)$.

Determining all the minima of  $U_h(x,K_1)$ is complicated by the fact that $K_2$ is not antipodal with respect to $m_2$, while the set of universal minima is antipodal with respect to $m_1$. One may be tempted to simply add the antipodal code of $K_2$ with respect to $m_2$. However, since this code does not belong to $\Lambda_{24}$, it is more convenient for our proof to add $\overline{K}_2\subset \Lambda_{24}$, a $47104$-point facet of the convex hull of $\Lambda(3)$, see Corollary \ref {CorK2}. This set is antipodal to $K_2$ with respect to $m_1$, and as such its projection onto the hypersphere circumscribing $K_1$ is antipodal 
with respect to $m_1$ to the projection of $K_2$. Hence, each point of $\overline{K}_2$ corresponds to a universal minimum for $U_h(x,K_1)$ thus giving another set of $47104$ universal minima. We note that the code $\overline{K}_2 - A= - K_2\in \Lambda(2)$ could also be used to describe these additional $47104$ universal minima.

The following corollary of Theorem \ref{LeechPULBpair} is interesting on its own.

\begin{corollary}\label{CorK2} The set of vertices $\mathcal{K}$ of the facet of the convex hull of $\Lambda(3)$, corresponding to the universal minimum associated with $A\in \Lambda(2)$ of the discrete potential $U_h(x,(1/\sqrt{6})\Lambda(3))$ (see \eqref{Lambda3PULB}) is symmetrical to $K_2$ with respect to the center of mass $m_1$ of $K_1$; i.e.,$\mathcal{K}=\overline{K}_2$.
\end{corollary}

\begin{figure}[htbp]
\includegraphics[width=3.3 in]{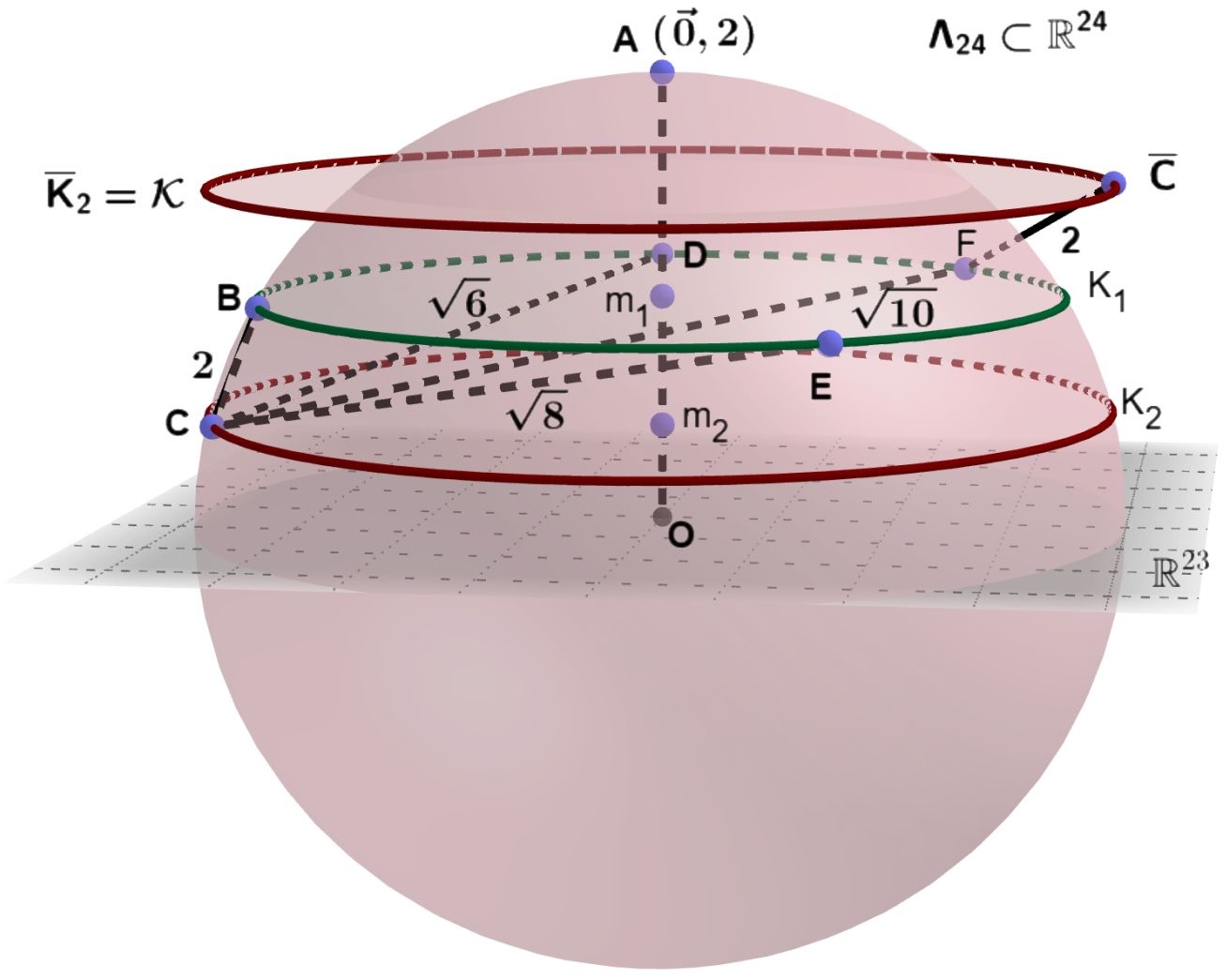}
\caption{The set of universal minima for $U_h(x,K_1)$: projection of $\widetilde{K}_2=K_2\cup \overline{K}_2$.}
\label{fig-4600_K2}
\end{figure}

\begin{proof} Let $A\in \Lambda(2)$ be associated with a fixed universal minimum of $U_h(x,(1/\sqrt{6})\Lambda(3))$ and $\mathcal{K}$ be the set of vertices of the corresponding facet; i.e.,$\mathcal{K} = S(A,2)\cap S(O,\sqrt{6}) \cap \Lambda_{24}$ (see Figure \ref{fig-4600_K2}). Let $\overline{C} \in \overline {K}_2\subset S(A,2)\cap S(O,\sqrt{6})$ be fixed and $C\in K_2$ be the symmetric point to $\overline {C}$ with respect to $m_1=A/2$. As $(C+\overline{C})/2=m_1=A/2$, we conclude that $\overline{C}=A-C$ and thus is a point from the Leech lattice. Then $\overline {C}\in \mathcal K$. Therefore, $\overline{K}_2\subset \mathcal{K}$. From \eqref{Lambda3PULB} and $|K_2|=47104$ we have that $| \mathcal{K}|=|\overline{K}_2|=47104$, which completes the proof.
\end{proof}

Examining Figure \ref{fig-4600_K2} we observe that the quadrangle $OCA\overline{C}$ is a parallelogram with sides $2$ and $\sqrt{6}$, from which we derive easily that the projection of $\overline{C}$ onto the hypersphere determined by $K_1$ splits the points of $K_1$ into the same four subcodes as $C$, and hence is a universal minimum antipodal (w.r.t. $m_1$) to the projection of $C$ (note that $|\overline{C}-F|=2=|C-B|$ and $|\overline{C}-E|=\sqrt{6}=|C-D|$).

We now show that the projections of $K_2\cup \overline{K}_2$ onto the hypersphere circumscribing $K_1$ are the only  universal minima of $U_h (x,K_1)$. Let $y$ be a universal minimum of $U_h (x,K_1)$ as embedded in the Leech lattice and let 
\[\widetilde{y}:=A/4+(\sqrt{15}/2)(y-A/2)/\sqrt{3}\]
be the projection of $y$ onto the hypersphere circubscribing $K_2$. From the PULB bound \eqref{PolarizationULB} we have 
\[\frac{(y-A/2)\cdot(z_i-A/2)}{\sqrt{3}\cdot \sqrt{3}}\in \left\{ \pm \frac{\sqrt{5}}{5}, \pm \frac{\sqrt{5}}{15} \right\}, \quad z_i \in K_1,\]
from which we derive that 
\[ (\widetilde{y}-A/4)\cdot z_i=(\widetilde{y}-A/4)\cdot (z_i -A/2)=(\sqrt{5}/2)(y-A/2)\cdot (z_i-A/2) \ \in  \left\{ \pm \frac{3}{2}, \pm \frac{1}{2} \right\}.\]
Therefore, we conclude that $\widetilde{y}\cdot z_i \in \{-1,0,1,2\}$, $i=1,\dots,4600$. Furthermore, $\widetilde {y}\cdot A=1$.

Let $L:={\rm ispan}(A,\{ z_i \})$ be the sub-lattice spanned by $A$ and the vectors of $K_1$ and denote its dual with $L^*:=\{ y\in \mathbb{R}^{24} \, :\, y\cdot a\in \mathbb{Z}, \ a\in L\}$. It is well-known that the index of $L$ is two \cite[Chapter 14]{CS}, thereby $|\Lambda_{24} : L|=|L^* : \Lambda_{24}|=2$. 
As $\widetilde{y}\cdot z_i \in \{-1,0,1,2\}$, $i=1,\dots,4600$, and $\widetilde{y}\cdot A=1$, we have that $\widetilde{y}\in L^*$.  We also observe that $A/2 \in L^*$. Since $A/2\not\in \Lambda_{24}$, we conclude that $L^*=\Lambda_{24}\cup (A/2+\Lambda_{24})$. 
Therefore, either $\widetilde{y}\in \Lambda_{24}$, in which case  $\widetilde{y}\in K_2$, or $A/2-\widetilde{y}\in \Lambda_{24}$. In the latter case, $A/2-\widetilde{y}$ is the antipode w.r.t. $A/4$ of vector $\widetilde{y}$, so $A/2-\widetilde{y}\in K_2$. Since the point $A/2+\widetilde{y}$ is symmetric to $A/2-\widetilde{y}$ with respect to $m_1=A/2$, it belongs to $\overline{K}_2$. Thus, the projections of $K_2\cup \overline{K}_2$ onto the hypersphere circumscribing $K_1$ are the only universal minima of $U_h(x,K_1)$.

We also derive that the union $\widetilde{K}_2$ of the $7$-designs $K_2$ and $\overline{K}_2$ (when projected onto $\mathbb{S}^{22}$) is a $7$-design and is separated into $4$ hyperplanes by the points of $K_1$ with $|\widetilde{K}_2|=94208$. Thus, $\widetilde{K}_2$ is a $4$-stiff configuration and therefore, PULB-optimal. Furthermore, every point of $K_1$ is a universal minimum of $\widetilde {K}_2$.

Next, we describe all universal minima of $U_h (x,\widetilde{K}_2)$. From Theorem \ref{PULB} for odd strength designs (case $\epsilon=0$), the nodes $\{ \alpha_i\}$, being zeros of a Gegenbauer polynomial, are symmetric about the origin, so we conclude that if such a design is PULB-optimal, then the collection of universal minima is an antipodal set. Thus, the universal minima of $U_h (x,\widetilde{K}_2)$ coincide with the universal minima of $U_h (x,K_2)$ and $U_h (x,\overline{K}_2)$ as positioned on the unit sphere $\mathbb{S}^{22}$, because in this case, $\overline{K_2}=-K_2$. Let $u$ be a universal minimum  of $U_h (x,K_2)$ (as positioned on the circumscribed hypersphere of $K_2$ centered at $A/4$ and embedded in $2\, \mathbb{S}^{23}$). Note that the cosines between $u-A/4$ and the points $z-A/4$, where $z\in K_2$, are in the set $ \{ \pm \sqrt{5}/5,\pm \sqrt{5}/15\}$. Thus, as $|u-A/4|=|z-A/4|=\sqrt{15}/2$, we have that
\[(u-A/4)\cdot (z-A/4)=(u- A/4)\cdot z \in \left\{ \pm \frac{3\sqrt{5}}{4}, \pm \frac{\sqrt{5}}{4} \right\}.\] 
The projection of $u$ onto the hypersphere circumscribing $K_1$ is given as $\widetilde{u}:=A/2+\sqrt{3}(u-A/4)/(\sqrt{15}/2)$. Clearly, 
\[ \widetilde{u}\cdot A=2,\quad \widetilde{u}\cdot z = A\cdot z/2+2(u-A/4)\cdot z/\sqrt{5} \in \{-1,0,1,2\},\ \ z\in K_2. \]
Below we prove a result similar to Lemma \ref{Span} that the points in $\{A\}\cup K_2$ generate the Leech lattice, which 
along with the self-duality of the Leech lattice implies that $\widetilde{u}\in K_1$ and will account for all $4600$ universal minima of $U_h(x,K_2)$. 

\begin{lemma}\label{Span2} The points of $K_2$ along with $A$ generate the Leech lattice.
\end{lemma}

\begin{proof} We recall the description of the 196560 vectors in $\Lambda(2)$ from Lemma \ref{lem-196560} and proceed with an explicit
description of $K_2$, omitting again the normalizing factor $1/\sqrt{8}$.
 Without loss of generality we may assume $A=[4,4,0^{22}]$. The points $Z\in K_2$ satisfy $|A-Z|=\sqrt{6}$, which is equivalent to $A\cdot Z=1$, from which we derive that the points in $K_2$ are as follows:
\begin{itemize}
\item[(1)] $ 2^{6} . 176$ vectors in $\Omega_1:=\{ [2,0,(\pm 2)^{7},0^{15}]\}$ and $ 2^{6} . 176$ vectors in $\Omega_2:=\{[0,2,(\pm 2)^{7},0^{15}]\}$ where there are even number of "$-$" signs (these are Leech points corresponding to the octads in $\mathcal{O}_{10}$ and $\mathcal{O}_{01}$, see Lemma \ref{O_1});
\item[(2)] $2 . 2^{10}$ vectors $[ 3,-1,(\mp 1)^{22}]$ or $[ -1,3,(\mp 1)^{22}]$, where the upper signs follow the Golay codewords with the first two digits equal to  one;
\item[(3)] $22 . 2^{10}$ vectors $[ 1,1,(\mp3)^{1},(\pm 1)^{21}]$, where the upper signs follow the words in the Golay codewords with the first two digits zero.\end{itemize}

Next, we verify that the vectors in $\sqrt{8}\Lambda(2)_2$ (see Lemma \ref{lem-196560}) belong to the integer span of the vectors in (1); i.e.,$\sqrt{8}\Lambda(2)_2\subset {\rm ispan}(\Omega_1,\Omega_2)$.  First, expand $\Omega_1$ and $\Omega_2$ to include $- \Omega_1$ and $- \Omega_2$, call for simplicity the new sets $\Omega _1$ and $\Omega_2$ too. Take an arbitrary Leech point $Y$ that has $\pm2$ in its first two positions, $6$ other $\pm2$ sprinkled among the rest of the $22$ positions, and zeroes elsewhere (denote the set of such points by $\Omega_3$). Let $\omega_Y$ denote the octad corresponding to $Y$. Clearly, $\omega_Y\in \mathcal{O}_{11}$ (see Lemma \ref{O_1}). Select a tetrad $T_1$ that covers the first position and any three of the six sprinkled positions described above and let $\{T_i\}_{i=1}^6$ be the associated sextet (see the discussion after Lemma \ref {octads-p}), where $\omega_Y=v_{T_1}+v_{T_2}$. 

Consider the point $B$ that matches the entries of $Y$ on $T_1$, has $\pm 2$ on $T_3$, and $0$ elsewhere. Clearly, we can select the number of negative signs on $T_3$ to match that on $T_1$. Then we have $B\in \Omega_1$ and $Y-B\in \Omega_2$. Since $Y=B+(Y-B)$, we conclude $Y\in {\rm ispan}(\Omega_1,\Omega_2)$, which implies that $\Omega_3 \subset {\rm ispan}(\Omega_1,\Omega_2)$.

Now take an arbitrary point $C\in \sqrt{8}\Lambda(2)_2$ that starts with $0,0$. Select four arbitrary positions, where it has $\pm2$ and call this $T_1$. This identifies the sextet $T_1,\dots,T_6$ uniquely and let $\omega_C=v_{T_1}+v_{T_2}$. Let $T_3$ be the (unique) tetrad that covers the first position and let $D$ be the Leech point that matches the entries of $C$ on $T_1$, has $\pm 2$ on $T_3$ with the same number of minus signs as on $T_1$, and $0$'s elsewhere. Then $D$ will be either in $\Omega_3$ or in $\Omega_1$, depending on whether $T_3$ covers position $2$ or not. The same will be true for $C-D$ and we get that $C=D+(C-D)$, which implies $C\in {\rm ispan}(\Omega_1, \Omega_3)\subset {\rm ispan}(\Omega_1, \Omega_2)$. Thus, $\sqrt{8}\Lambda(2)_2\subset {\rm ispan}(\Omega_1,\Omega_2)$.

Let $[x_1,x_2,\dots, x_{24}]$ be an arbitrary Leech point, where $x_i$ satisfy \eqref{LeechDef}. Similar to Lemma \ref{Span} we add a vector of type (3) if needed, so that all $x_i$ are even. As the octads generate the Golay code, by adding vectors of type $[(\pm 2)^8,0^{16}]$, we may reduce the situation to all $x_i$ divisible by $4$. By \cite[Theorem 10.24]{CS} the vectors $[2^8,0^{16}]$  span the vectors $[4^1, (-4)^1,0^{22}]$. Using integer combinations of $[4^1, (-4)^1,0^{22}]$, we further reduce the problem to only one non-zero $x_i$ that is a multiple of $8$. On the other hand $[4, 4,0^{22}]+[4, -4,0^{22}]=[8,0^{23}]$ 
and $[8,0^{23}]-2[4, (-4)^1,0^{22}]=[0, 8^1,0^{22}]$, which completes the proof of the Lemma.
\end{proof}

We summarize the results in the following theorem.

\begin{theorem}\label{4600PULBpair} For any potential $h$ with $h^{(8)}(t)> 0$, $t\in (-1,1)$, the codes $K_2=C_{47104}$ and $\widetilde{K}_2=K_2\cup  \overline{K}_2$ are PULB-optimal; i.e., they attain the bound \eqref{PolarizationULB} 
\begin{equation}\label{K2}
\begin{split}
 m^h (K_2)&=2816h(-\sqrt{5}/5) + 20736h(-\sqrt{5}/15) + 20736h(\sqrt{5}/15) + 2816h(\sqrt{5}/5),\\
 m^h (\widetilde{K}_2)&=5632h(-\sqrt{5}/5) + 41472h(-\sqrt{5}/15) + 41472h(\sqrt{5}/15) + 5632h(\sqrt{5}/5).
 \end{split}
\end{equation}
with universal minima at the points of the sharp code $K_1=C_{4600}=(23,4600,7)$.

Moreover, the projections onto the unit sphere $\mathbb{S}^{22}$  of $K_1$ and $\widetilde{K}_2$ form a maximal PULB-optimal pair $(K_1,\widetilde{K}_2)$ (and a fortiori universal polar dual pair) with normalized discrete potentials achieving the same extremal value
\begin{equation*}
\frac{m^h(K_1)}{4600}= \frac{m^h(\widetilde{K}_2)}{94208} \ \left( = \frac{m^h(K_2)}{47104}\right) .\end{equation*}
\end{theorem}

\begin{remark} For completeness, we illustrate the symmetry groups approach for counting the distinct orthogonal transforms sending one McLaughlin facet into another and thus the universal minima. The group of orthogonal transformations preserving $K_1$ is $Co_2$ and the subgroup preserving a McLaughlin facet is $McL$. The number of transformations then is found to be
\[ \frac{2^{18} .\, 3^6 . \, 5^3 .\, 7\, .\, 11\, .\, 23}{2^{7} .\, 3^6 .\, 5^3 .\, 7\, .\, 11}= 2^{11} .\, 23=47104 \]
Note that an orthogonal transformation that sends a McLaughlin facet to its antipodal does not preserve the Leech lattice and therefore, in this way we only account for the distinct facets corresponding to a $\triangle AOC$ with $C\in K_2$. As a facet determined by a universal minimum will correspond to a $\triangle AOC$ or its centrally-symmetric about $m_1$ $\triangle AO\overline{C}$ with $\overline{C} \in \overline{K}_2$, we confirm that the number of universal minima of $U_h (x,K_1)$ is $94208$.
\end{remark}

\subsection{The maximal PULB-optimal pair generated by the sharp code $C_{891}=(22,891,5)$ and the maximal antipodal code $C_{2816}=(22,2816,5)$ on $\mathbb{S}^{21}$.} \label{891code} In this subsection, we utilize the notation from Subsection \ref{4600code}, unless noted otherwise. Let us consider the sharp code $A_1:=C_{891}$. This code is 
unique up to isometry (see \cite{CK-2}), and when embedded in $\Lambda_{24}$, is a kissing configuration inside the sharp code $K_1=C_{4600}$ from Subsection \ref{4600code} as can be seen from the Energy ULB for absolutely monotone $h$-potentials associated with $K_1$ in Table \ref{EnergyULB_Table}:
\begin{equation}\label{K1EnergyULB} \mathcal{E}(23,4600;h)/4600\geq h(-1) + 891h(-1/3) + 2816h(0) + 891h(1/3) .
\end{equation}
We also derive from \cite[Theorem 8.2]{DGS} that the three subcodes $A_1$, $A_2:=C_{2816}$, and $A_3$, are $5$-designs, where $A_3$ is a second copy of the sharp code $(22,891,5)$, centrally symmetric to $A_1$ about the center of mass $m_1$ of $K_1$ (we remind the reader that $K_1$ is an antipodal code). As in the previous subsection, we shall refer to $A_1$, $A_2$, and $A_3$ as codes on the unit sphere  $\mathbb{S}^{21}$ and as configurations embedded in the Leech lattice, with the use becoming clear from the context.

\begin{figure}[htbp]
\includegraphics[width=3.6 in]{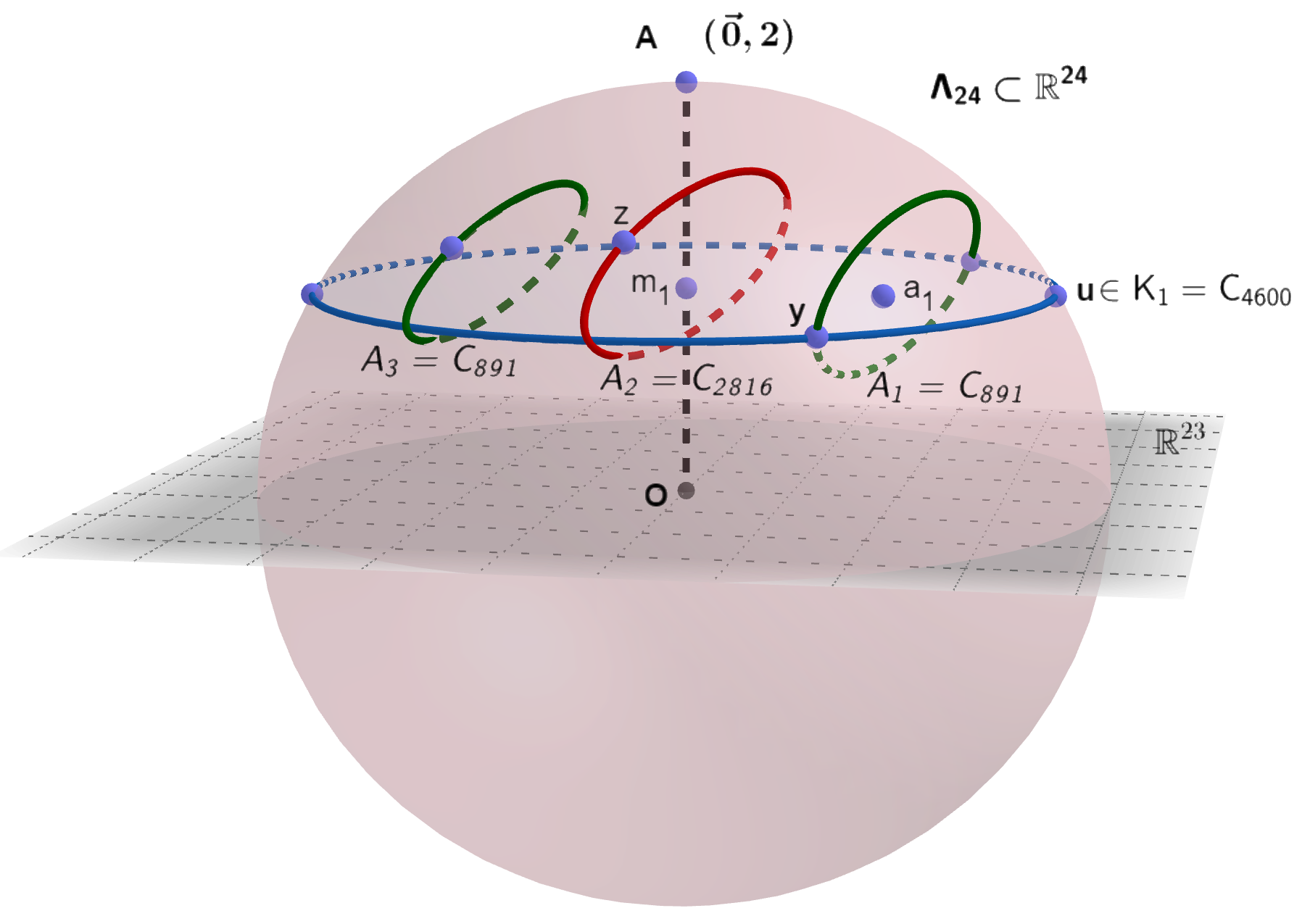}
\caption{The subcodes of $K_1$ determined by $u\in K_1$.}
\label{fig-891_G}
\end{figure}

The subcode $A_2$ of cardinality $|A_2|=2816$ is a remarkable configuration identified as a maximal antipodal code on $\mathbb{S}^{21}$  by one of the authors more than 30 years ago (see \cite[Theorem 6]{Boy}). As we shall show in Section 6, the $1408$ lines determined by the antipodal pairs of $A_2$ constitute a universally optimal code in the real projective space $\mathbb{P}\mathbb{R}^{21}$, a fact that seems to have been missed in the literature so far. 

First, we derive that the projections on $\mathbb{S}^{21}$ of the points in $A_2$ are universal minima of $U_h (x,A_1)$. As we shall see later, the minimal set includes two  additional sets, $A_2^\prime$ and $A_2^{\prime \prime}$, of the same cardinality. Fix arbitrary $z\in A_2$. Recall that the center of mass of $K_1$ is $m_1=A/2$, which is also the center of mass of $A_2$. Let $a_1$ be the center of mass of $A_1$. From the energy ULB formula \eqref{K1EnergyULB} written with repect to $u\in K_1$ fixed, we compute that 
\begin{equation}\label{a_1}
a_1=A/2+(u-A/2)/3=(u+A)/3.
\end{equation} Suppose $y\in A_1$.
 Then we are interested in the quantity (note that the radius of the hypersphere circumscribing $A_1$ is $|y-a_1|=\sqrt{8/3}$)
\[ \cos(\phi):=\frac{y-a_1}{|y-a_1|}\cdot \frac{z-A/2}{|z-A/2|}=\frac{(y-u/3-A/3)\cdot (z-A/2)}{\sqrt{8/3}\, . \sqrt{3}}=\frac{z\cdot y -1}{\sqrt{8}}.\]
The last equality holds because $A\cdot (y-a_1)=0$, $z\cdot A=2$, and $z\cdot u=1$. The possible  distances  $|z-y|$ are $ 2,\sqrt{6},\sqrt{8}$, which implies that $z\cdot y\in \{2,1,0\}$, respectively. Indeed, the triangle formed by $z$, its antipodal w.r.t. $m_1$ point $A-z$, and $y$ is a right triangle with diameter (hypothenuse) $|z-(A-z)|=2\sqrt{3}$ made of Leech lattice points and therefore, with sides of length at least $2$, which prohibits $|z-y|=\sqrt{10}$ as a possibility. Thus, $\cos(\phi)\in \{-1/\sqrt{8},0,1/\sqrt{8}\} $, the set of respective nodes from \eqref{PolarizationULB} for this case. This establishes that both, $A_1$ and $A_2$, are $3$-stiff. Thus, the projection of $z$ onto $\mathbb{S}^{21}$ (we project points in $A_1$ onto the hyperplane determined by $A_2$ and then projecting both sets of points onto a unit hypersphere of radius one) is a universal minimum for $U_h (x,A_1)$ (here $A_1$ is of course the configuration on $\mathbb{S}^{21}$).  Similarly, the projection of any fixed $y\in A_1$ is a universal minimum of $U_h(x,A_2)$. The same is true for the projection of any point in $A_3$, as it will be $-y$ for some $y\in A_1$.
\

Let $A_1:=\{y_i \}_{i=1}^{891}, A_2:= \{ z_j\}_{j=1}^{2816}$, and $y$ be an arbitrary minimum of the potential $U_h (x,A_1)$, where  $A_1$ is embedded in the Leech lattice and $y$ lies in the hypersphere circumscribed about $A_1$. Consider the sub-lattice $M:={\rm ispan}(A,u,A_1)$. The Smith normal form is ${\rm diag}(1,2^9,4^{13},8)$ (see Remark \ref{Mathematica}), which yields that its index is $|\Lambda_{24}:M|=4$. Since $A\cdot A=u\cdot u=4, A\cdot u=A\cdot y_i=u\cdot y_i=2$, we obtain that for any $a\in M$ we have $a\cdot A, a\cdot u \in 2\mathbb{Z}$.

The projection of $y$ onto the (equatorial) hypersphere circumscribing $A_2$ is 
\[ \widetilde{y}=\frac{A}{2}+\frac{\sqrt{3}(y-a_1)}{\sqrt{8/3}},\]
for which we derive that $\widetilde{y}\cdot A=2$, $\widetilde{y}\cdot u=1$, and
\[\widetilde{y}\cdot y_i=1+\frac{\sqrt{3}(y-a_1)\cdot (y_i-a_1)}{\sqrt{8/3}}+\frac{\sqrt{3}(y-a_1)\cdot (a_1-A/2)}{\sqrt{8/3}}+\frac{\sqrt{3}(y-a_1)\cdot (A/2)}{\sqrt{8/3}}\in \{2,1,0\}.\]
Note that the last two terms are $0$ and the second term takes on the values $\pm 1,0$. Therefore, $\widetilde{y}$ belongs to the dual sub-lattice $M^*:=\{ b\in \mathbb{R}^{24} \, :\, b\cdot a \in \mathbb{Z}, a\in M\}$. 

Since for any $z_j \in A_2$, $j=1,\dots ,2816$, we have $z_j \cdot u =1$, we conclude that $z_j \notin M$. Then $M$ is a proper sub-lattice of $L:={\rm ispan}(z_j,M)$ and as $z_j \cdot A=2$, the point $A$ forms only even dot products with vectors of $L$. Therefore, the lattice $L$ is a proper sub-lattice of $\Lambda_{24}$ (recall that $A$ forms dot product $1$ with points of $K_2\subset \Lambda _{24}$). Clearly, $|L:M|=|\Lambda_{24}:L|=2$ and we have the strict inclusions $M\triangleleft L\triangleleft \Lambda_{24} \triangleleft {L^*}\triangleleft M^*$ with index $2$ between any consecutive subgroups in the string. Note that  any $z_j$, $j=1,\dots,2816$, identifies the same sub-lattice $L$.

{\bf Case 1:} Suppose there is an index $j$, such that $\widetilde{y}\cdot z_j\in \mathbb{Z}$. Then $\widetilde{y}\in L^*$. Clearly, $A/2 \in L^*$ as well. As $A/2\not \in \Lambda_{24}$, we have that $L^* =\Lambda_{24} \cup (A/2 +\Lambda_{24})$. If $\widetilde{y}\in (A/2 +\Lambda_{24})$, then $\widetilde{y}-A/2  \in\Lambda_{24} $. However, $|\widetilde{y}-A/2|^2=3$, which would be a contradiction. Therefore, if $\widetilde{y}\in L^*$, then $\widetilde{y}\in \Lambda_{24}$, and thus $\widetilde{y}\in A_2$.

{\bf Case 2:} Suppose now that $\widetilde{y}\in M^*\setminus L^*$; i.e., none of the inner products of $ \widetilde{y}$ with $z_j$'s is integer.  Since $2z_j \in M$ ($z_j \in L\setminus M$ and $L=M\cup (z_j +M)$, so $z_j+z_j \in M$), we have $\widetilde{y}\cdot z_j\in (1/2)(2\mathbb{Z}+1)$. Recall that the diameter of the equatorial hypersphere where $z_j$'s lie is $2\sqrt{3}$ and that $z_j$ and $A-z_j$ form an antipodal pair with respect to $m_1=A/2$. Therefore, $|\widetilde{y}-z_j|^2+|\widetilde{y}-(A-z_j)|^2=12$, from which we easily derive the only possibilities are $|\widetilde{y}-z_j|^2\in \{1,3,5,7,9,11\}$ or
\[ \widetilde{y}\cdot z_j\in \left\{ -\frac{3}{2},-\frac{1}{2}, \frac{1}{2}, \frac{3}{2}, \frac{5}{2}, \frac{7}{2} \right\},\]
which yields for the cosine distribution
\[ \cos \phi =\frac{(\widetilde{y}-A/2)\cdot (z_j-A/2)}{\sqrt{3}\sqrt{3}} \in \left\{\pm \frac{1}{6},\pm\frac{1}{2}, \pm \frac{5}{6}\right\}.\]
Because $A_2$ is an antipodal code with respect to its center of mass $A/2$, we have symmetry of the frequencies, which we denote by $2\alpha, 2\beta,2\gamma$ for dot products $\pm \frac{1}{6},\pm\frac{1}{2}, \pm \frac{5}{6}$, respectively. If there were two distinct points $z_i$ and $z_j$ such that $|\widetilde{y}-z_i|=|\widetilde{y}-z_j|=1$, then we would have $|z_i-z_j|<2$, which is impossible.  Therefore, at most one $z_i$ satisfies $\left|\widetilde y-z_i\right|^2=1$. Thus, $\gamma=1$ or $\gamma=0$ are the only possibilities.

Assume first that $\gamma =1$. Then $\alpha=1407-\beta$. The projection of $A_2$ onto $\mathbb{S}^{21}$ is an antipodal spherical $5$-design, so applying Definition \ref{def-des} with $x=\widetilde{y}$ and $f(t)=t^2$ and $t^4$, we obtain
\begin{equation} \label{System_pq}
\begin{split} \frac{1}{22}&=\frac{1}{2816}\left[ 2\left( \frac{5}{6}\right)^2+ 2\beta\left( \frac{1}{2}\right)^2 +2\alpha\left( \frac{1}{6}\right)^2\right],\\
\frac{1\cdot3}{22\cdot24}&=\frac{1}{2816}\left[ 2\left( \frac{5}{6}\right)^4+ 2\beta\left( \frac{1}{2}\right)^4 +2\alpha\left( \frac{1}{6}\right)^4\right],
\end{split}
\end{equation}
respectively. The second equation yields $\beta=521/5\notin \mathbb Z$, which renders this case impossible.

Assume now that $\gamma =0$. Then $\alpha=1408-\beta$ and the analogs of equations \eqref{System_pq} become
\begin{equation*} \label{System_pq2}
\begin{split} \frac{1}{22}&=\frac{1}{2816}\left[  2\beta\left( \frac{1}{2}\right)^2 +2\alpha\left( \frac{1}{6}\right)^2\right], \\
\frac{1 \cdot 3}{22 \cdot 24}&=\frac{1}{2816}\left[  2\beta\left( \frac{1}{2}\right)^4 +2\alpha\left( \frac{1}{6}\right)^4\right],
\end{split}
\end{equation*}
which yields $\alpha=1296$ and $\beta=112$. This defines a split of the $2816$-code into four parallel hyperplanes, with the $112$-code being a facet of the convex polytope with vertices in $A_2$. Note, that the radius of the facet is $3/2$. As a matter of fact, the $112$-facet, along with the $162$-facet associated with the minimum $y$, and $u$ form a copy of the McLaughlin $275$-code. Indeed, we can show that the centers of mass of $112$-subcode of $A_2$ consisting of points closest to $\widetilde{y}$ lies on the same line with $u$ and and the center of mass of the $162$-subcode closest to $y$, from which we derive that the $22$-dimensional hyperplane determined by that $112$-subcode and $u$ has an intersection with the hypersphere circumscribing $A_1$, will contain the $162$-subcode of $A_1$. Observe that this $275$-code has no intersections with $A_3$, this is important information we will use later.

Next, we observe that $u/2\in M^*\setminus L^*$. Indeed, $u/2 \cdot A, u/2 \cdot u, u/2 \cdot y_i \in \mathbb{Z}$, so $u/2 \in M^*$. However, $u/2 \cdot z_j=1/2$ yields $u/2 \notin L^*$. Hence, $M^*$ splits into four cosets
\[ M^* = \Lambda_{24}\cup \left( A/2+\Lambda_{24}\right)\cup \left( u/2+\Lambda_{24}\right)\cup \left(u/2- A/2+\Lambda_{24}\right). \] 
From Case 1 we already derived there are $2816$ universal minima of $U_h (x,A_1)$ that when projected onto the Equatorial hypersphere circumscribing $A_2$, coincide with $A_2$. We also showed $\widetilde{y}\notin (A/2+\Lambda_{24})$, because $\widetilde{y}-A/2 \notin \Lambda_{24}$. 

Suppose now that $\widetilde{y}\in (u/2+\Lambda_{24})$; i.e.,$\widetilde{y}=u/2+v$, for some $v\in \Lambda_{24}$. As $|v|=|\widetilde{y}-u/2|=2$, $|v-A|=\sqrt{6}$, and $|\widetilde{y}-A/2|=\sqrt{3}=|v-(A-u)/2|$, the points $v$ will lie on the intersection of $\Lambda(2)$ with a sphere with center $(A-u)/2$ and radius $\sqrt{3}$. The point $A-u$ is antipodal to $u$ about the center of mass of $K_1$. Moreover, from $\sqrt{3}=|v-(A-u)/2|$ we conclude $(A-u)\cdot v=2$, or $|(A-u)-v|=2$; i.e.,$v$ belongs to the $2816$-facet of $K_2$ found in the PULB bound \eqref{K2}. On the other hand, each $v$ has four inner products $\{2,1,0,-1\}$ with points in $K_1$ (see Subsection \ref{4600code}). The vectors in $K_1$ having inner product $2$ with $v$ form a McLaughlin $275$-code containing $A-u$ and having empty intersection with $A_1$. This implies that the points of $A_1$ have three distinct inner products $\{\pm 1,0\}$ with $v$ (and hence inner products $\{2,1,0\}$ with $u/2+v$), and therefore, are split into three hyperplanes. The frequency of the respective inner products is uniquely determined by the quadrature formula \eqref{PolarizationULB}. Clearly, the set of $2816$ points $\{u/2+v \}$ lies on the same hypersphere and is disjoint with $A_2$.

The argument when $\widetilde{y}\in (u/2-A/2+\Lambda_{24})$ is similar. If $\widetilde{y} = u/2-A/2+v$, $v\in \Lambda_{24}$, then $|v|=|\widetilde{y} -u/2+A/2|=\sqrt{6}$. We also have that $|v-(A-u)|=|v-A|=2$. Therefore, in this case the collection of points $ v$ lies on a sphere with radius $\sqrt{3}$ and center $A-u/2$ and forms a $2816$-facet in $\overline{K}_2$ associated with $A-u$. Per Subsection \ref{4600code} these points $v$ correspond to universal minima of $U_h (x,K_1)$ and split the $K_1$ code into four hyperplanes, one of which has empty intersection with $A_1$. As a result we obtain a third (and distinct) set of $2816$ universal minima of $U_h(x,A_1)$.

\begin{figure}[htbp]
\includegraphics[width=3.4 in]{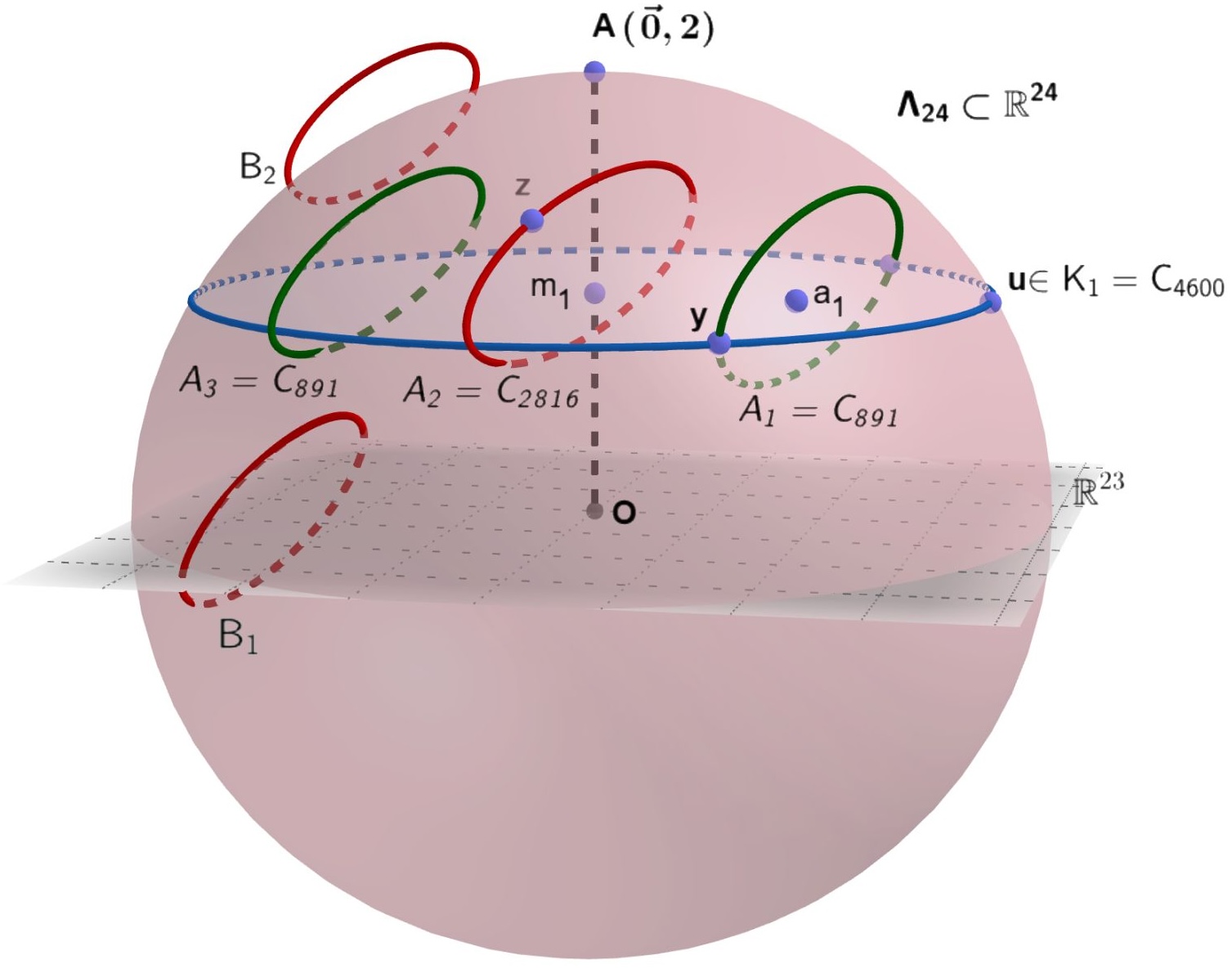}
\caption{The maximal PULB-pair $(A_1\cup A_3,A_2\cup A_2^\prime \cup A_2^{\prime \prime})$.}
\label{fig-891_GG}
\end{figure}

To summarize, we derive there are total of $3\cdot 2816=8448$ universal minima of $U_h(x,A_1)$,  determined as follows. Note that 
\[A_2=S(0,2)\cap S(A,2)\cap S(A-u,\sqrt{6})\cap \Lambda_{24},\]
and define the sets
\[ B_1:=S(0,2)\cap S(A-u,2)\cap S(A,\sqrt{6})\cap \Lambda_{24}, \quad B_2:= S(0,\sqrt{6})\cap S(A-u,2)\cap S(A,2)\cap \Lambda_{24}.\]
Then one third of the universal minima are coming from $A_2$, one third from $A_2^\prime:=u/2+B_1$, and one third from $A_2^{\prime \prime}:= u/2-A/2+B_2$. 

Next we determine the universal minima of $U_h (x, A_2\cup A_2^\prime \cup A_2^{\prime \prime})$. As $A_2\cup A_2^\prime \cup A_2^{\prime \prime}$ is a $5$-design, we  observe that if $y$ is such a universal minimum, then $I(y,A_2\cup A_2^\prime \cup A_2^{\prime \prime})= \{-1/\sqrt{8},0,1/\sqrt{8}\}$. Moreover, since all codes $A_2, A_2^\prime, A_2^{\prime\prime}$ are $5$-designs and $I(y,A_2), I(y,A_2^\prime), I(y,A_2^{\prime \prime})$ are all subsets of $\{-1/\sqrt{8},0,1/\sqrt{8}\}$, we conclude from Theorem \ref{PULB} that 
\[I(y,A_2)= I(y,A_2^\prime)= I(y,A_2^{\prime \prime})=\{-1/\sqrt{8},0,1/\sqrt{8}\},\]
and that $y$ is a universal minimum for each of the potentials $U_h(x,A_2)$, $U_h(x,A_2^\prime)$, and $U_h(x,A_2^{\prime\prime})$. 
Therefore, we shall examine only the universal minima of $U_h (x, A_2)$.

 We already proved that the points of $A_1$ and $A_3$ projected onto $\mathbb{S}^{21}$ form universal minima for the discrete $h$-potential $U_h(x,A_2)$. Without loss of generality we may assume the coordinate representation of $A, u$ above to be
\begin{equation*}\label{LeechCoord} 
A:=[4,4,0^{22}], \quad u:=[4,0,4,0^{21}],
\end{equation*}
where we shall omit the factor of $1/\sqrt{8}$ here and in what follows. The point $y_1:= [0,4,4,0^{21}]\in A_1$ induces a split $A_2=D_1\cup D_2\cup D_3$ as follows:
\begin{itemize}
\item[I.] $D_1:=\{[1,3,1,(\pm 1)^{21}]\}=\{z_1,\dots,z_{512}\}$, where the signs follow the Golay codewords with $-3$ and $+1$ corresponding to $1$ and $3$ and $-1$ corresponding to $0$ ($y\cdot z_i=2$);
\item[II.] $D_2:=\{[2,2,0,(\pm 2)^6,0^{15}] \}=\{z_{513},\dots, z_{2304}\}$, where $2$ and $\pm 2$ is placed at the Golay code octads with even number of negative signs ($y\cdot z_i=1$);
\item[III.] $D_3:=\{[3,1,-1,(\pm 1)^{21}]=\{z_{2305},\dots, z_{2816}\}$, where the signs follow the Golay codewords with $-3$ and $+1$ corresponding to $1$ and $3$ and $-1$ corresponding to $0$ ($y\cdot z_i=0$). 
\end{itemize}

To find that there are no other (universal) minima of $U_h(x,A_2)$, let $y$ be such a minimum on the hypersphere circumscribing $A_2$ as embedded in $\Lambda_{24}$. Then 
\[ \frac{(y-A/2)\cdot(z_j-A/2)}{\sqrt{3} \, \sqrt{3}} \in \left\{\pm \frac{1}{\sqrt{8}},0\right\},\ j=1,\dots, 2816.\]
Denote its projection onto the hypersphere circumscribing $A_1$ with $\widetilde{y}$ (recall that $a_1$ is the center of mass of $A_1$, see \eqref{a_1}), i.e.
\[ \widetilde{y}=a_1+\frac{\sqrt{8/3}(y-A/2)}{\sqrt{3}} = \frac{A+u}{3}+\frac{\sqrt{8}(y-A/2)}{\sqrt{3}\, \sqrt{3}} . \]
Straightforward calculation show that
\[ (\widetilde{y}-a_1)\cdot (z_j-A/2) =\widetilde{y}\cdot z_j -1 \in \{\pm 1,0\},\]
from which we get $\widetilde{y}\cdot z_j \in \{0,1,2\}$. Let $M:={\rm ispan}(A,u,A_2)$. Since $\widetilde{y}\cdot A=\widetilde{y}\cdot u=2$, we have that $\widetilde{y}\in M^*$. The Smith normal form of $M$ is ${\rm diag}(1,2^{10},4^{12},24)$, so $|\Lambda_{24}:M|=6$. Clearly, $M\triangleleft L$. We already know $A/2\in L^*\subset M^*$. We verify directly that $a_1=(A+u)/3\in M^*$. Therefore, we obtain the coset decomposition
\[ M^*=\bigcup_{j=0}^{1}\bigcup_{k=0}^{2}\left(\frac{jA}{2}+\frac{k(A+u)}{3}+\Lambda_{24}\right).\] 
Next we determine that
\[ \left| \widetilde{y}-\frac{jA}{2}-\frac{k(A+u)}{3}\right|^2=4+j^2+\frac{4k^2}{3}-2j-\frac{8k}{3}+2kj=4+2kj-2j+j^2+\frac{4k(k-2)}{3},\]
which is an even integer if and only if $j=0$ and $k=0$ or $k=2$. 

When $j=k=0$, we have $\widetilde{y}\in \Lambda_{24}$, so $\widetilde{y}\in A_1$. When $j=0$ and $k=2$, we conclude that the antipodal with respect to $a_1$ point $2a_1-\widetilde{y}$ is in the Leech lattice or $A_1$. This is equivalent that the projection of $y$ onto the hypersphere circumscribing $A_3$ belongs to the Leech lattice.  This concludes the construction of a one-to-one mapping between the universal minima of $U_h(x,A_2)$ and $A_1\cup A_3$.

We combine the results in the following theorem.

\begin{theorem}\label{891PULBpair} For any potential $h$ with $h^{(6)}(t)> 0$, $t\in (-1,1)$, the codes $A_2$ and $D:=A_2\cup A_2^\prime\cup A_2^{\prime \prime}$ are PULB-optimal; i.e., they attain the bound \eqref{PolarizationULB} 
\begin{equation*}\label{A2}
\begin{split}
 m^h (A_2)&\geq 512h(-1/\sqrt{8}) + 1792h(0) + 512h(1/\sqrt{8}),\\
  m^h (D)&\geq 1536h(-1/\sqrt{8}) + 5376h(0) + 1536h(1/\sqrt{8})
 \end{split}
\end{equation*}
with universal minima at the points $C:=C_{891}\cup (-C_{891})$ of sharp code $A_1=C_{891}=(22,891,5)$ and its antipodal $A_3=-C_{891}$.

Moreover, $(C,D)$ is a maximal PULB-optimal pair, and thus universal polar dual pair, with normalized discrete potentials achieving the same extremal value
\begin{equation*}
\frac{m^h(C)}{1782}=\frac{m^h(D)}{8448}\ \left(=\frac{m^h(C_{891})}{891}=\frac{m^h(C_{2816})}{2816}\right) .\end{equation*}
\end{theorem}

\subsection{The maximal PULB-optimal pair generated by the sharp code $C_{552}=(23,552,5)$ and the derived code $C_{11178}=(23,11178,5)$} \label{SharpCode552} \label{552} The PULB line of Table \ref{PolarizationULB_Table} for the sharp code $(24,196560,11)$ exhibits seven derived codes and from Theorem \ref{derived_codes_thm} we conclude that each of them is a spherical $5$-design. According to Subsection \ref{level-1-leech}, if $w$ is a fixed point in $\Lambda(3)$, then the two subcodes of $\Lambda(2)$ closest to $w$ are the collection $\mathcal{B}:=C_{552}$ of $552$ points  that are at distance $2$ from $w$ and the collection $\mathcal{C}:=C_{11178}$ of $11178$ points at distance $\sqrt{6}$ from $w$ (see Figure \ref{fig-552}). We remark that $\mathcal{B}$ is an antipodal tight $5$-design attaining the Delsarte-Goethals-Seidel bound.

\begin{figure}[htbp]
\includegraphics[width=3.4 in]{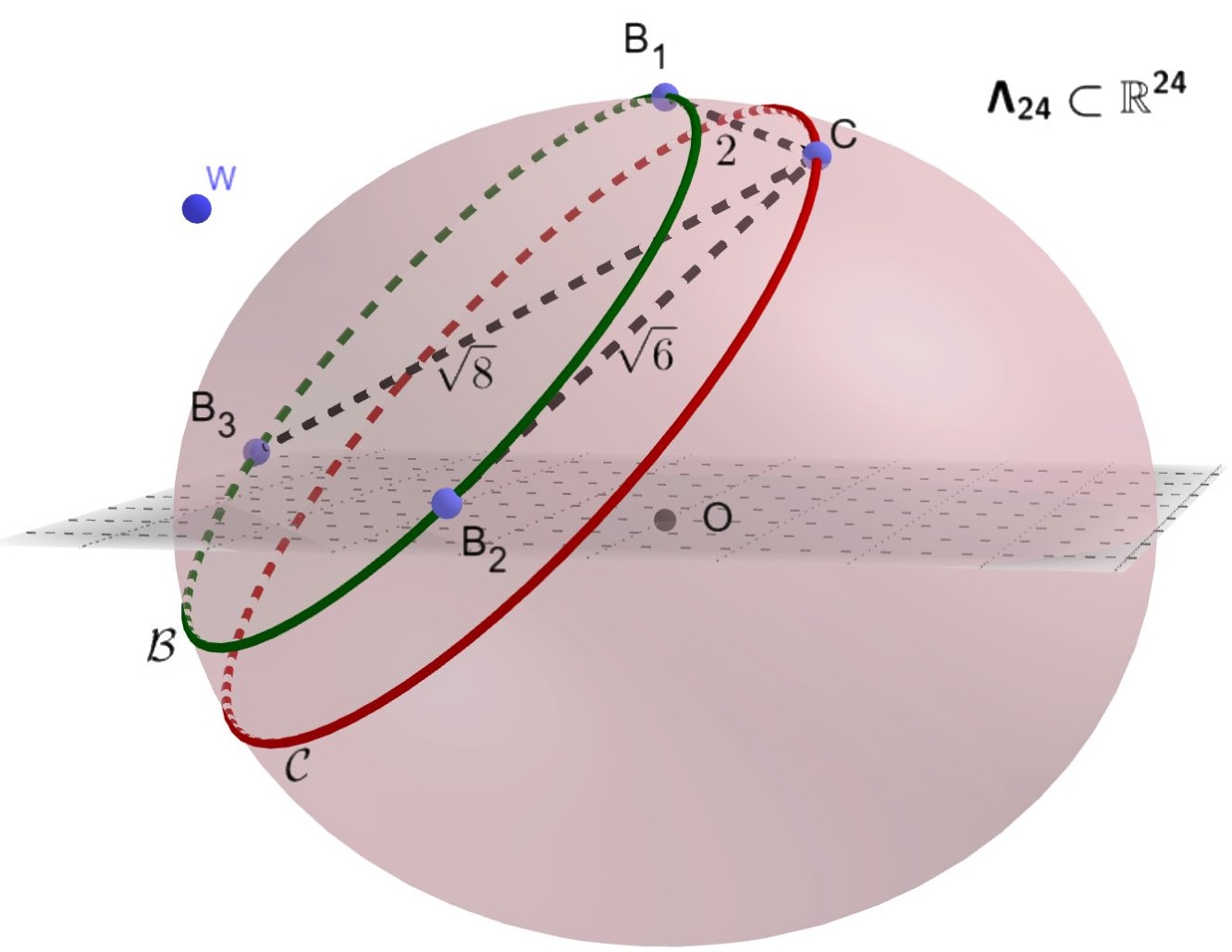}
\caption{Illustration of the stiffness of the derived codes $\mathcal{B}$ and $\mathcal{C}$ in
$\Lambda_{24}$.}
\label{fig-552}
\end{figure}

Let $\overline{\mathcal{C}}:=w- \mathcal{C}$ be the centrally symmetric to $\mathcal{C}$ configuration with respect to the center of mass $w/2$ of $\mathcal{B}$. We shall prove that the projections onto $\mathbb{S}^{22}$ of the points of the code $\widetilde{\mathcal{C}}:=\mathcal{C}\cup \overline{\mathcal{C}}=(23,22356,5)$ are all of the universal minima of the discrete potential of the respective projection of $\mathcal{B}$ onto $\mathbb{S}^{22}$ and vice versa. We shall denote the said projections onto $\mathbb{S}^{22}$ with $\mathcal{B}$ and $\widetilde{\mathcal{C}}$ as well, where the use will become clear from the context.

Let us fix points $B_1, B_2, B_3\in \mathcal{B}$ and $C\in \mathcal{C}$, where $|C-B_1|=2$, $|C-B_2|=\sqrt{6}$, $|C-B_3|=\sqrt{8}$, which implies that $C\cdot B_1=2$, $C\cdot B_2=1$, and $C\cdot B_3=0$. Since $\mathcal{B}$ is antipodal with respect to its center of mass $w/2$, we have $w-\mathcal{B}=\mathcal{B}$. As $2$ is the greatest inner product between any distinct points of $\Lambda(2)$, the least inner product between $C$ and points of $\mathcal{B}=w-\mathcal{B}$ would equal $0$ (note that $C\cdot w=2$). 

The centers of mass of $\mathcal{B}$ and $\mathcal{C}$ are $w/2$ and $w/3$, respectively. Thus, we can compute the normalized inner products 
\[ \cos(\phi)=\frac{(B_i-w/2)\cdot(C-w/3)}{|B_i-w/2|.|C-w/3|}=\frac{B_i\cdot C - w\cdot C/2}{\sqrt{\sfrac{5}{2}}\sqrt{\sfrac{10}{3}}}=\frac{\sqrt{3}(B_i\cdot C-1)}{5} \quad \in \quad \left\{ \pm \frac{\sqrt{3}}{5},0 \right\},\]
as $B_i\cdot C$ takes on values $\{ 2,1,0\}$ for $i=1,2,3$, respectively. Other dot products for $B_i \cdot C$ are impossible as the values of the cosine above are symmetric about the origin 
($\mathcal{B}$ is antipodal about $w/2$), so if $B_i \cdot C\leq-1$, then there is 
$B_j$ such that $B_j\cdot C\geq 3$ or $|B_j-C|^2\leq 2$, but Leech lattice has no vectors of squared length $2$ (note $B_i \not= C$). 

We observe that the cosines are the roots of $P_3^{(23)}(t)=t(25t^2-3)/22$, and hence the nodes of the Gauss-Gegenbauer quadrature. The uniqueness of the quadrature yields the weights and hence we obtain the PULB
\begin{equation}\label{552-code_PULB}
m^h (\mathcal{B})\geq 100 h\left( -\sqrt{3}/5\right) +352h(0)+100 h\left(\sqrt{3}/5\right),
\end{equation}
which holds true for any potential $h$ with $h^{(6)}(t)> 0$, $t\in (-1,1)$. Further, every point of $\mathcal{C}$ projected onto $\mathbb{S}^{22}$ is a universal minimum of $U_h(x,\mathcal{B})$. Since $\mathcal{C}$ is a spherical $5$-design and, by reciprocity, the projection of any point $B\in\mathcal{B}$ onto $\mathbb{S}^{22} $ splits $\mathcal{C}$ into three parallel hyperspheres, we obtain that $\mathcal{C}$ is $3$-stiff, and hence PULB-optimal as well.

To compute the total number of universal minima of $U_h (x,\mathcal{B})$, consider the sub-lattice $L$ spanned by $\mathcal{B}$; i.e.,$L:={\rm ispan}(\mathcal{B})$. Utilizing the Smith normal form ${\rm diag}(1,2^{11},4^{11},24)$, as described in the last paragraph of Subsection \ref{lattbasics}, we obtain for the index  $|\Lambda_{24}:L|=3$, which implies the index for the dual $|L^* : \Lambda_{24}|=3$. Indeed, by selecting $w=[5,1^{23}]$, we then determine $\mathcal{B}=A\cup B\cup C$, where 
\begin{itemize}
\item[] $A:=\{ [4,(4)^1,0^{22}]\}$, $|A|=23$;
\item[] $B:=\{ [2,2^7,0^{15}]\}$ following the Golay octads that start with $1$, $|B|=253$;
\item[] $C:=\{ w-x : x\in A\cup B\}$.
\end{itemize}
We find that the Smith normal form is ${\rm diag}(d_1,\dots,d_{24})={\rm diag}(1,2^{11},4^{11},24)$, which yields $|\Lambda_{24}:L|=3$.

Clearly, $w\in L$ ($w$ is the sum of two antipodal vectors of $\mathcal{B}$ with respect to its center of mass) and $w/3\in L^*$. Therefore,
\begin{equation}\label{ispanB}
L^*=\Lambda_{24}\cup \left( w/3+\Lambda_{24} \right)\cup \left( 2w/3+\Lambda_{24} \right).\end{equation}
On the other hand, as in Subsection \ref{4600code} we compute that if $y$ is a universal minimum of $U_h (x,\mathcal{B})$, then $(y-w/2)/|y-w/2|$ has dot-products $\{ \pm \sqrt{3}/5,0\}$ with the vectors $(b-w/2)/|b-w/2|$, $b\in \mathcal{B}$, and as in Subsections \ref{4600code} and \ref{891code} the  projection 
\[\widetilde{y}:=w/3+2(y-w/2)/\sqrt{3}\]
 of $y$ onto the hypersphere circumscribing $\mathcal{C}$ has integer inner products $\{0,1,2\}$ with vectors of $\mathcal{B}$ and hence, belongs to $L^*$. If $\widetilde{y} \in \Lambda_{24}$ then $\widetilde{y}\in \mathcal{C}$. 

Since $|\widetilde{y}-w/3|^2=10/3$, we conclude that $\widetilde{y}\notin w/3+\Lambda_{24} $. 

Finally, suppose that $\widetilde{y}\in 2w/3+\Lambda_{24}$. Then the antipodal point of $\widetilde{y}$ w.r.t. the center of mass $w/3$ of $\mathcal{C}$ (which is $2w/3-\widetilde{y}$) is in the Leech lattice. This  implies that the projection 
\[\overline{y}:=2w/3+ 2(y-w/2)/\sqrt{3}=w/3+\widetilde{y}=w-(2w/3-\widetilde{y})\] 
of $y$ onto the hypersphere circumscribing $\overline{\mathcal{C}}$ is in the Leech lattice or equivalently belongs to $\overline{\mathcal{C}}$. As each point in $\overline{\mathcal{C}}$ splits $\mathcal{B}$ into three subcodes embedded into parallel hyperplanes, we obtain a one-to-one correspondence between the universal minima of $U_h (x,\mathcal{B})$ and the vectors in $\widetilde{\mathcal{C}}=\mathcal{C}\cup \overline{\mathcal{C}}$; i.e., $U_h(x,\mathcal{B})$ has no other minima but the ones in the projection of $\widetilde{\mathcal{C}}$ onto the hypersphere circumscribing $\mathcal{B}$.

\begin{figure}[htbp]
\includegraphics[width=3.4 in]{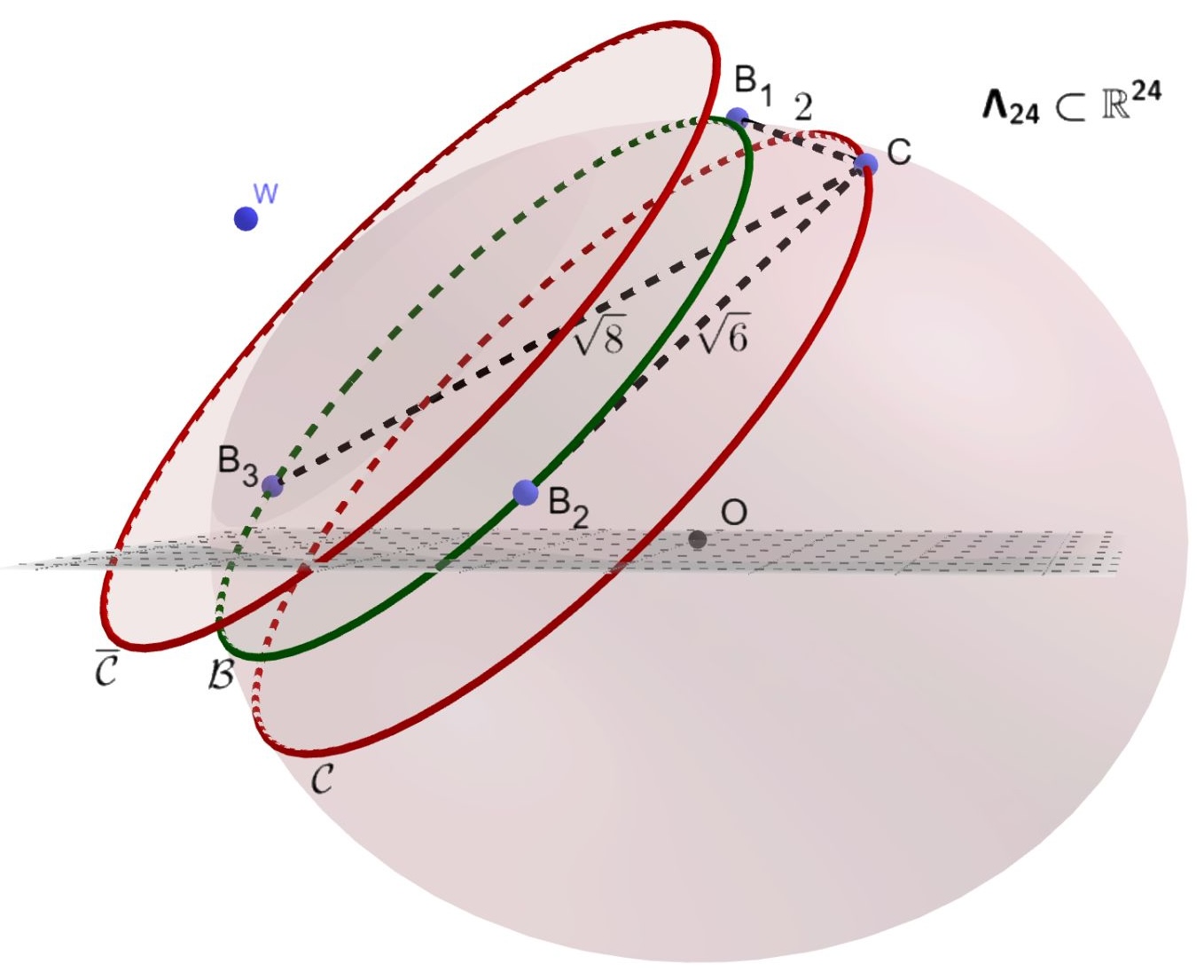}
\caption{The maximal PULB-optimal pair $(\mathcal{B}, \widetilde{\mathcal{C}})$, where $\widetilde{\mathcal{C}}=\mathcal{C}\cup \overline{\mathcal{C}}$.}
\label{fig-552-5}
\end{figure}

We next derive that the potential $U_h(x,\widetilde{\mathcal{C}})$ (where $\widetilde{\mathcal{C}}$ is considered as a configuration on $\mathbb{S}^{22}$) has exactly $552$ minima. Our approach is similar to the one in Subsection \ref{4600code}. Since the projections of $\mathcal{C}$ and $\overline{\mathcal{C}}$ on $\mathbb{S}^{22}$ are $5$-designs (odd strength), antipodal to each other, from the symmetry the set of global minima of $U_h(x,\widetilde{\mathcal{C}})$ is the same as that of $U_h(x,\mathcal{C})$, so it suffices to consider only the number of universal minima of $U_h(x,\mathcal{C})$. We know that any point of $\mathcal{B}$ projected onto the hypersphere circumscribing $\mathcal{C}$ is a universal minimum of $U_h(x,\mathcal{C})$, which yields that the number of universal minima is at least $552$. 

Suppose now that $y$, as positioned on the hypersphere $\{x\in \mathbb{R}^{24} : w\cdot x=2\}\cap (2\mathbb S^{23})$ circumscribing $\mathcal{C}$, is a universal minimum of $U_h (x,  \mathcal{C})$. The projection $\widetilde{y}$ of $y$ onto the hypersphere circumscribing $\mathcal{B}$ is given as
\[ \widetilde{y}:=\frac{w}{2}+\frac{y-w/3}{|y-w/3|}\frac{\sqrt{10}}{2}=\frac{w}{2}+\frac{\sqrt{3}}{2}(y-w/3).\]
From the PULB \eqref{PolarizationULB} applied to the code $C_{11178}=(23,11178,5)$; i.e.,$n=23$ and $\tau=5$, we have that for any $z\in \mathcal{C}$
\[ \frac{(y-w/3)\cdot (z-w/3)}{|y-w/3| |z-w/3|}=\frac{(y-w/3)\cdot z}{10/3}\in \left\{ \pm \frac{\sqrt{3}}{5},0 \right\},\]
which easily implies that $\widetilde{y}\cdot z\in \{2,1,0\}$. Let $L:={\rm ispan}(\mathcal{C})$, then $\widetilde{y} \in L^*$. The Smith normal form is computed as ${\rm diag}(1,2^{11},4^{11},16)$, therefore we have that $|L^* : \Lambda_{24}|=2$ (an alternative proof is given in Lemma 5.9 of arXiv:2512.25037v1). Clearly, $w/2 \in L^* \setminus \Lambda_{24}$, therefore 
\[ L^* = \Lambda_{24} \cup \left( w/2 +  \Lambda_{24}\right). \]  
From $|\widetilde{y}-w/2|^2=5/2$, we get that $\widetilde{y}\notin w/2+ \Lambda_{24}$.  Therefore, we conclude that $\widetilde{y} \in \Lambda_{24}$ and thus, $\widetilde{y} \in \mathcal{B}$.

We summarize the findings in the following theorem.

\begin{theorem}\label{552PULBpair} For any potential $h$ with $h^{(6)}(t)> 0$, $t\in (-1,1)$,  the codes $C_{11178}:=\mathcal{C}$ and $C_{22356}=C_{11178}\cup (-C_{11178}):=\widetilde{\mathcal{C}}$ on $\mathbb{S}^{22}$ are PULB-optimal; i.e., they attain the bound \eqref{PolarizationULB} 
\begin{equation}\label{11178}
\begin{split}
 m^h (C_{11178})\geq 2025h(-\sqrt{3}/5) + 7128h(0) +  2025h(\sqrt{3}/5),\\
 m^h (C_{22356})\geq 4050h(-\sqrt{3}/5) + 14256h(0) +  4050h(\sqrt{3}/5)
 \end{split}
\end{equation}
with the set of universal minima being the sharp code $\mathcal{B}=:C_{552}=(23,552,5)$.

Moreover, $(C_{552},C_{22356})$ is a maximal PULB-optimal pair, and thus universal polar dual pair, with normalized discrete potentials achieving the same extremal value
\begin{equation*}
\frac{m^h(C_{552})}{552}= \frac{m^h(C_{22356})}{22356} \left( = \frac{m^h(C_{11178})}{11178}\right).\end{equation*}
\end{theorem}

\subsection{The maximal PULB-optimal pair generated by the symmetrized McLaughlin code $C_{550}=(22,550,5)$ and the symmetrized Equatorial derived code $C_{7128}:=(22,7128,5)$ of $C_{11,178}$} \label{SymmetrizedMcLaughlin}  We shall use the set-up and notations from the previous section. Let $b\in \mathcal{B}$ be fixed and denote with $b^\prime$ the antipodal to $b$ point about $w/2$, the center of mass of $\mathcal{B}$. Then the Energy ULB in Table \ref{EnergyULB_Table} indicates that $b$ splits $\mathcal{B}=\{b,C_{275},C_{275}^\prime,b^\prime\}$, where $C_{275}$ and $C_{275}^\prime$ are two copies of the McLaughlin code symmetric about $w/2$.
In Theorem \ref{E8-pairs-1-5}(b) the code $C_{126}$ is split into two Schl\"{a}fli derived codes $C_{27}$ antipodal to each other, and the Equatorial derived code $C_{72}$ (see \eqref{C126PULB}). We then obtain the symmetrized Schl\"{a}fli code $C_{54}$ in Theorem \ref{E8-pairs-1-5}(c) as projections of the two copies of $C_{27}$ onto the circumscribed hypersphere of universal polar dual counterpart $C_{72}$. While we shall project analogously here to obtain the first code in the universal polar dual pair, the symmetrized McLaughlin code $C_{550}$, the difference in our approach here comes from the fact that unlike Theorem \ref{E8-pairs-1-5}(c), the universal polar dual counterpart of $C_{550}$ lies on the Equatorial hyperspheres of $\mathcal{C}$ and $\overline{\mathcal{C}}$, the universal polar dual of $\mathcal{B}$.

Let $H$ be the perpendicular bisector hyperplane of the segment $\overline{bb^\prime}$ and let $S$ be the Equatorial hypersphere obtained as the intersection of $H$ and the circumscribed hypersphere of $\mathcal{B}$. We have that the codes $C_{275}$ and $C_{275}^\prime$ lie in parallel to $H$ hyperplanes passing through their respective centers of mass
\begin{equation}\label{McLcenter} m_{275}:=\frac{w}{2}+\frac{1}{5}\left(b-\frac{w}{2}\right)=\frac{2w+b}{5},\quad \quad m_{275}^\prime:=\frac{w}{2}-\frac{1}{5}\left(b-\frac{w}{2}\right)=\frac{3w-b}{5}.
\end{equation}  
Let $x\in C_{275}$ and $x^\prime \in C_{275}^\prime$ be arbitrary. Their projections onto $S$ are given by
\begin{equation}\label{McLxbar} \overline{x}:=\frac{w}{2}+\frac{\sqrt{10}}{2}\frac{x-m_{275}}{|x-m_{275}|},\quad \quad \overline{x^\prime}:=\frac{w}{2}+\frac{\sqrt{10}}{2}\frac{x^\prime-m_{275}^\prime}{|x^\prime-m_{275}^\prime|}.
\end{equation}  
Denote the two codes on $S$  formed by these projections with $\overline{C}_{275}$ and $ \overline{C}_{275}^\prime$, Then $C_{550}:=\overline{C}_{275}\cup \overline{C}_{275}^\prime$ is a symmetrized McLaughlin code, lies on $S$, and is a spherical $5$-design as an antipodal code (both copies are already tight $4$-designs). This will be our first code in the universal polar dual pair.

To construct the second code, we observe that from \eqref{11178}, the vector $b$ splits the points of $C_{11178}$ into three subcodes $C_{2025}$, $C_{7128}$, $C_{2025}$ at distances $2$, $\sqrt{6}$, and $\sqrt{8}$, respectively. If $c\in C_{7128}$ is arbitrary, then $b\cdot c=1$. We note that the vector 
\begin{equation} \label{C11178cbar}
 \overline{c}=\frac{w}{2}+\frac{\sqrt{10}}{2}\frac{c-w/3}{\sqrt{10/3}}
 \end{equation}
belongs to $S$. Indeed, it is easy to see that $(b-w/2)\cdot(\overline{c}-w/2)=0$, so the projection of $c$ onto the circumscribing hypersphere of $\mathcal{B}$ is already in the Equatorial hypersphere (with $b$ and $b^\prime$ being the poles). In a similar fashion we determine that if $c^\prime \in C_{7128}^\prime:=w-C_{7128}$, then 
\[
\overline{c^\prime}=\frac{w}{2}+\frac{\sqrt{10}}{2}\frac{c^\prime-2w/3}{\sqrt{10/3}}
\]
also belongs to $S$. Denote the codes formed by those projections with $\overline{C}_{7128}$ and $\overline{C}_{7128}^\prime$ and let $C_{14256}:=\overline{C}_{7128}\cup \overline{C}_{7128}^\prime$. We shall show that $(C_{550},C_{14256})$ form a universal polar dual pair.

To do this we first show that the two are PULB-optimal pairs. Let us determine the cosine between $\overline{x}-w/2$ and $\overline{c}-w/2$ first (see \eqref{McLxbar} and \eqref{C11178cbar})
\begin{equation} \label{550cos}
\cos(\phi)=\frac{(\overline{x}-w/2)\cdot (\overline{c}-w/2)}{(\sqrt{10}/2)(\sqrt{10}/2)}=\frac{(x-m_{275})\cdot (c-w/3) }{|x-m_{275}|\sqrt{10/3}} \in \left\{\pm \frac{1}{\sqrt{8}},0\right\}.
\end{equation}
Indeed, we compute easily from the Energy ULB that $|x-m_{275}|=(\sqrt{24}/5)\sqrt{10}/2=2\sqrt{3/5}$, and derive \eqref{550cos} from $x\cdot c\in \{2,1,0\}$, $(x-m_{275})\cdot w=0$, and $m_{275}\cdot c=1$ (see \eqref{McLcenter}).

As both $C_{550}$ and $C_{14256}$ are symmetrical about $w/2$, \eqref{550cos} holds when $\overline{x}^\prime$ and/or $\overline{c}^\prime$ are substituted for $\overline{x}$ or $\overline{c}$, respectively. Since the cosines are precisely the zeros of $P_3^{(22)}(t)$, we conclude that the two codes form a PULB-optimal pair, provided $C_{14256}$ is a $5$-design, which we show next. 

In Subsection \ref{SharpCode552} we assumed  $w=[5,1^{23}]$ and itemized the coordinates of $\mathcal{C}$ accordingly (the factor $1/\sqrt{8}$ was omitted). With $b=[4,4,0^{22}]$ the coordinates of $C_{7128}$ are:
\begin{itemize}
\item[1.] $B:=\{ [0,2,2^7,0^{15}]\}$ following the Golay octads that start with $[0,1,\dots]$, $|B|=176$;
\item[2.] $C:=\{ [2,0,2^5,(-2)^2,0^{15}]\}$ following the Golay octads that start with $[1,0,\dots]$, $|C|={7 \choose 2}\cdot 176=3696$;
\item[3.] $D:=\{ [ 3,-1,(-1)^{10},1^{12}]\}$ following the Golay dodecads that start with $[1,1,\dots]$, $|D|=616$ (see \cite[Table 10.2]{CS});
\item[4.] $E:=\{ [ 1,1,3^1,(-1)^{7},1^{14}]\}$ following the Golay octads that start with $[0,0,\dots]$, $|E|=8\cdot 330=2640$ (see \cite[Table 10.1]{CS}).
\end{itemize}

We compute the cosines within $C_{7128}$ as $\{\alpha_j\}=\{1,2/5,1/10,-1/5,-1/2\}$ with respective frequencies 
$\{r_j\}=\{1,750,3500,2625,252\}$. Since the code is at least $3$-design by Theorem \ref{derived_codes_thm}, we conclude that $C_{7128}$ is distance regular by \cite[Theorem 7.4]{DGS}. As such, we can compute the first four moments using the above distance distribution 
as
\[ M_i^{22}(C_{7128})=7128 \sum_{j=0}^4 r_j P_i^{(22)}(\alpha_j)=0, \quad i=1,2,3,4,\]
which verifies that $C_{7128}$ is a spherical $4$-design. As a matter of fact $M_5^{22}(C_{7128})=0$, which yields that $C_{7128}$ is a $5$-design. The symmetrization $C_{14256}$ is an antipodal spherical $5$-design.

We next show the maximality of this PULB-pair. The entry corresponding to $C_{275}$ in Table \ref{EnergyULB_Table} yields that the inner products of $C_{550}\subset \mathbb{S}^{21}$ are $\pm 1, \pm 1/4,\pm 1/6$ with frequencies $1,112, 162$, respectively. As the code is distance regular (see \cite[Theorem 7.4]{DGS}), if we fix an arbitrary vertex it will identify associated copies of $C_{275}$ and its antipode $C_{275}^\prime$ on the unit sphere. Next we fix two arbitrary antipodal  points $b, b^\prime$ in $\mathcal{B}$ and let $C_{275}=\{x_i\}_{i=1}^{275}$ and $C_{275}^\prime=\{x_i^\prime\}_{i=1}^{275}$ be the associated McLaughlin codes found in $\mathcal{B}$. Since $(22,275,4)$ is unique up to isometry we may embed the code $C_{550}$ into the Equatorial hypersphere $S$ determined by $b$ and $b^\prime$, so that it coincides with respective projections $\overline{C}_{275}$ and $\overline{C}_{275}^\prime$. 

\begin{figure}[htbp]
\includegraphics[width=3.4 in]{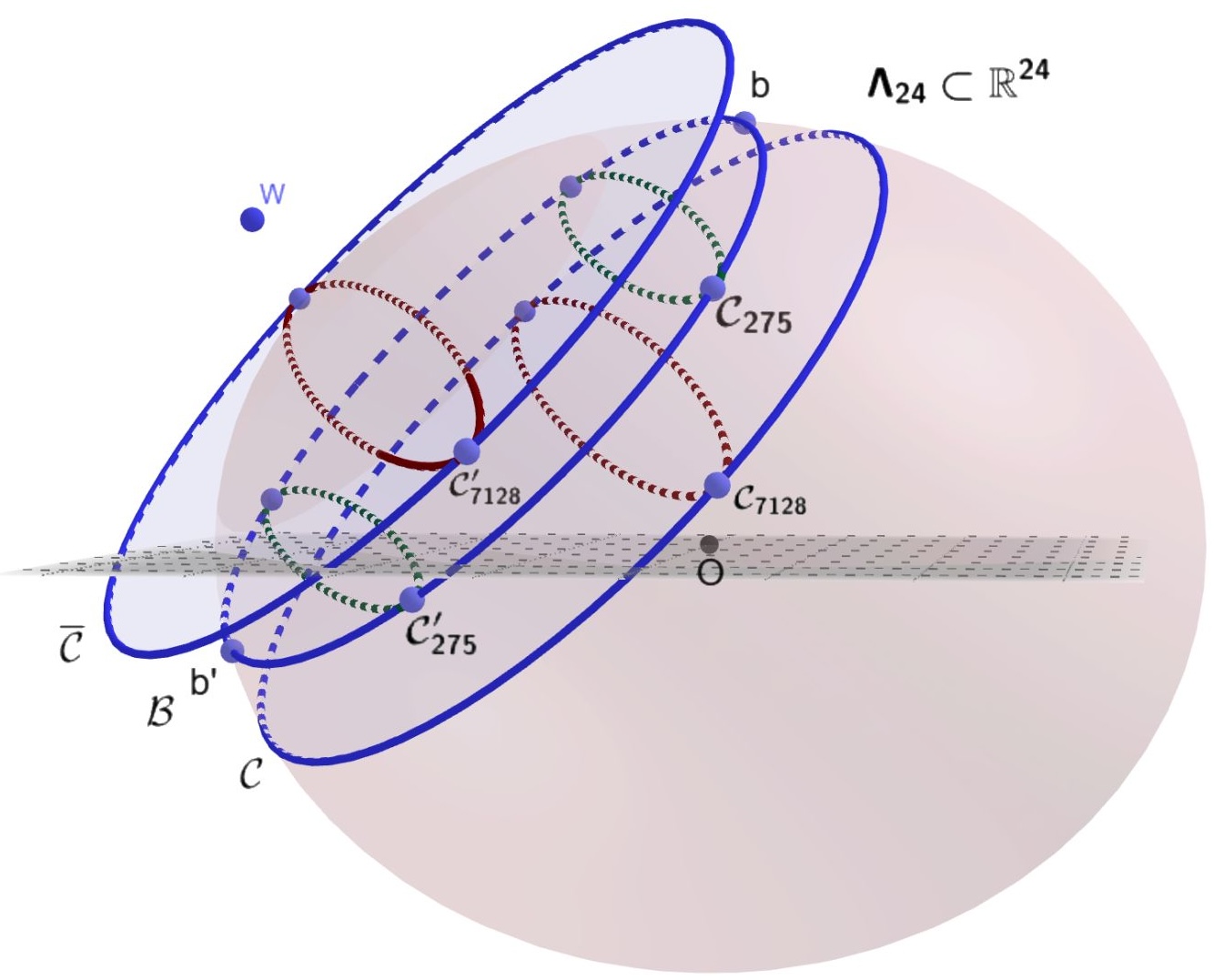}
\caption{The universal polar dual pair generated by the symmetrized McLaughlin  $C_{550}=(22,550,5)$ in $\Lambda_{24}$.}
\label{fig-550}
\end{figure}

Let $y\in S$ be an arbitrary universal minimum of $U_h(x,C_{550})$. We want to show that there are no more than the already found $14256$ universal minima coming from $C_{7128}$ and $C_{7128}^\prime$. In particular,
\[ \frac{(y-w/2)\cdot (\overline{x}_i-w/2)}{\sqrt{10}/2\sqrt{10}/2} \in \left\{\pm \frac{1}{\sqrt{8}}, 0\right\},\ \frac{(y-w/2)\cdot (\overline{x}_i^\prime-w/2)}{\sqrt{10}/2\sqrt{10}/2} \in \left\{\pm \frac{1}{\sqrt{8}}, 0\right\}, \ i=1,\dots, 275.\]
Project $\overline{x}_i$ onto the circumscribed hypersphere of $C_{275}$ and $y$ onto the circumscribed sphere of $C_{7128}$
\[ x_i:=m_{275}+\frac{\overline{x}_i-w/2}{\sqrt{10}/2}\frac{\sqrt{24}}{5}\frac{\sqrt{10}}{2},\quad \widetilde{y}:=\frac{w}{3}+\frac{y-w/2}{\sqrt{10}/2}\sqrt{\frac{10}{3}} ,\]
from which we compute $(\widetilde{y}-w/3)\cdot (x_i-m_{275}) \in \{\pm 1,0\}.$
Since $w\cdot (x_i-m_{275})=0$, from \eqref{McLcenter}, $\widetilde{y}\cdot w=2$, and $\widetilde{y}\cdot b=1$, we derive $\widetilde{y}\cdot x_i \in \{0,1,2\}$. Similarly $\widetilde{y}\cdot x_i^\prime \in \{0,1,2\}$. 

Recall the sub-lattice $L={\rm ispan}(\mathcal{B})$ and its dual $L^*$, considered in the previous subsection. The coset expansion of $L^*$ is given in \eqref{ispanB}. Clearly, $\widetilde{y}\in L^*$ (note that $\widetilde{y}\cdot b^\prime=1$), from which we conclude in an analogous manner that if $\widetilde{y}\in \Lambda_{24}$, then $\widetilde{y}\in C_{7128}$, or if $\widetilde{y}\in 2w/3+\Lambda_{24}$, then the antipode $2w/3- \widetilde{y}$ relative to the center of mass $w/3$ is in $C_{7128}$ (the case $\widetilde{y}\in w/3+\Lambda_{24}$ is again rejected). Thus, the total number of universal minima of $U_h(x,C_{550})$ is  indeed $14256$. 

Let now $y$ be a universal minimum of $U_h(x,C_{14256})$ as embedded in $S$. Recall that $C_{14256}=\overline{C}_{7128} \cup \overline{C}_{7128}^\prime$, where $\overline{C}_{7128}$ and $\overline{C}_{7128}^\prime$ are the projections onto $S$ of the respective Equatorial derived codes $C_{7128}$ and $C_{7128}^\prime$ of $\mathcal{C}$ and $\overline{\mathcal{C}}$ and $b$ and $b^\prime$ are the poles. Let
\[ \widetilde{y}:=m_{275}+\frac{\sqrt{24}}{5}(y-w/2) \]
be the projection of $y$ onto $C_{275}$ as embedded in $\mathcal{B}$ (see Figure \ref{fig-550}). Since $y$ is a universal minimum we have
\[ \frac{(y-w/2)\cdot (\overline{c}-w/2)}{(\sqrt{10}/2)(\sqrt{10}/2)}\in \left\{ \pm \frac{1}{\sqrt{8}},0\right\},\]
from which using \eqref{C11178cbar} one obtains
\[ \frac{5}{\sqrt{24}}\frac{(\widetilde{y}-m_{275})\cdot (c-w/3)}{(\sqrt{10}/2)\sqrt{10/3}}=\frac{\widetilde{y}\cdot c-m_{275}\cdot c}{\sqrt{8}} \in \left\{ \pm \frac{1}{\sqrt{8}},0\right\}. \]
We have that $c \cdot m_{275}=c\cdot (2w+b)/5=1$, from which we derive that $\widetilde{y}\cdot c\in \{0,1,2\}$. Let $M:={\rm ispan}(w,b,C_{7128})$.  Since $\widetilde{y}\cdot w=3$ and $\widetilde{y}\cdot b=2$, we have $\widetilde{y}\in M^*$. The Smith normal form of $M$ is ${\rm diag}(1,2^{11},4^{11},40)$, which implies $|\Lambda_{24}:M|=|M^* : \Lambda_{24}|=5$. We compute $w\cdot(2w+b)/5=3$ and $b\cdot (2w+b)/5=2$ and we already found $c\cdot (2w+b)/5=1$. This implies $(2w+b)/5 \in M^*$ and since for any $x_i\in C_{275}$ we have $x_i\cdot (2w+b)/5=8/5$, so $(2w+b)/5\notin \Lambda_{24}$, we can identify the coset representation of $M^*$
\[ M^* =\bigcup_{k=0}^4 \left( \frac{k(2w+b)}{5}+\Lambda_{24}\right).\] 
To determine the cosets $\widetilde{y}$ may belong to we find
\[ \left| \widetilde{y}- \frac{k(2w+b)}{5} \right|^2= |\widetilde{y}|^2+\frac{k^2|2w+b|^2}{25}-\frac{2k\widetilde{y}\cdot (2w+b)}{5}=4+\frac{8k(k-2)}{5}.\]
The latter is even integer if and only if $k=0$, in which case $\widetilde{y}\in C_{275}$, or $k=2$, in which case the antipode of $\widetilde{y}$ with respect to $m_{275}=(2w+b)/5$ is in the Leech lattice; i.e., $2m_{275}-\widetilde{y}\in \Lambda_{24}$. This completes the proof of the following theorem. 

\begin{theorem}\label{550PULBpair} For any potential $h$ with $h^{(6)}(t)> 0$, $t\in (-1,1)$, the symmetrized McLaughlin code $C_{550}$ and the union of the Equatorial derived code $C_{7128}$ of $\mathcal{C}$ and its antipode $C_{7128}^\prime$ are PULB-optimal; i.e., they attain the bound \eqref{PolarizationULB} 
\begin{equation*}\label{550}
\begin{split}
 m^h (C_{550})\geq 100h(-\frac{1}{\sqrt{8}}) + 350h(0) +  100h(\frac{1}{\sqrt{8}}),\\
 m^h (C_{14256})\geq 2592h(-\frac{1}{\sqrt{8}}) + 9072h(0) +  2592h(\frac{1}{\sqrt{8}})
 \end{split}
\end{equation*}
with the set of universal minima of $U_h(x,C_{550})$ being $C_{14256}=(22,14256,5)$ and vice versa.

Moreover, the projections onto the unit sphere $\mathbb{S}^{21}$  of $C_{550}$ and $C_{14256}$ form a maximal PULB-optimal pair, and consequently universal polar dual pair denoted with $(C_{550},C_{14256})$, with normalized discrete potentials achieving the same extremal value
\begin{equation*}
\frac{m^h(C_{550})}{550}= \frac{m^h(C_{14256})}{14256}=\frac{m^h(C_{7128})}{7128}.
\end{equation*}
The last equality holds because $C_{7128}$ is a $5$-design as we already verified above, and therefore, is $3$-stiff code. As a matter of fact, the following PULB bound holds
\[m^h (C_{7128})\geq 1296h\left(-\frac{1}{\sqrt{8}}\right) + 4536h(0) +  1296h\left(\frac{1}{\sqrt{8}}\right).\]
\end{theorem}

\subsection{The maximal PULB-optimal pair generated by the sharp code $\mathcal{F}_1=(22,100,3)$ and the maximal (in the projective space) antipodal code $\mathcal{F}_2=(22,352,3)$} \label{HigmanSims} We now consider the three derived codes of $\mathcal{B}$ in \eqref{552-code_PULB} determined by a fixed $z\in \mathcal{C}$ and denote them as $\mathcal{F}_i$, $i=1,2,3$, with $|\mathcal{F}_1|=|\mathcal{F}_3|=100$ and $|\mathcal{F}_2|=352$. The sharp code $\mathcal{F}_1$ corresponds to the Higman-Sims graph srg$(100,22,0,6)$ (cf. \cite{BrSRG,HiSi,Me}). Theorem \ref{derived_codes_thm} implies all three derived codes are spherical $3$-designs in the corresponding hyperspheres. As $\mathcal{B}$ is antipodal (when projected on $\mathbb{S}^{22}$), the actual configuration in the Leech lattice is symmetric about the center of mass $w/2$, which implies that $\mathcal{F}_1$ and $\mathcal{F}_3$, even though not having central symmetry, are symmetric to each other about $w/2$. 

\begin{figure}[htbp]
\includegraphics[width=3.5 in]{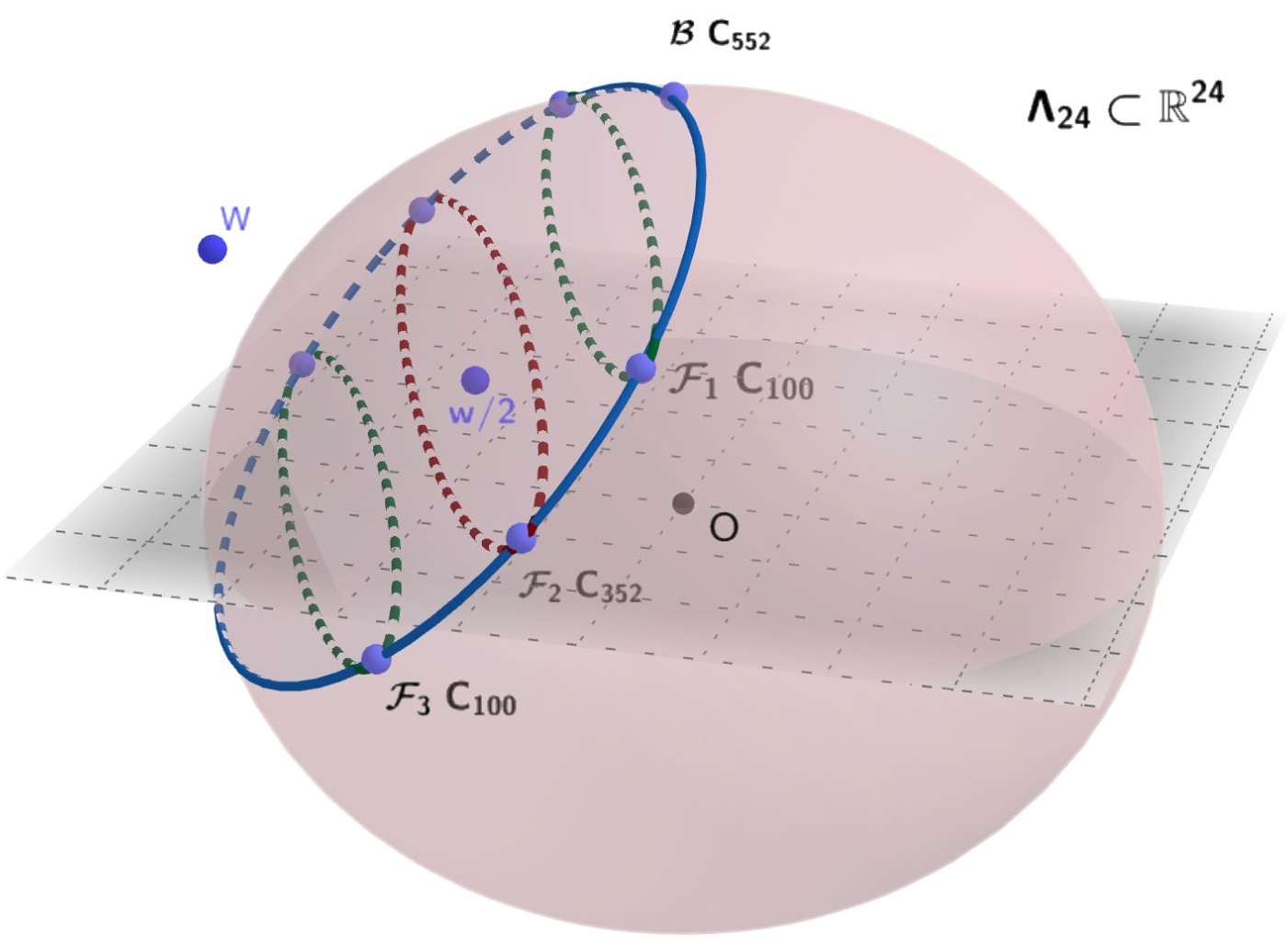}
\caption{The set of universal minima of the sharp code $(22,100,3)$ in $\Lambda_{24}$.}
\label{fig-100}
\end{figure}

We note that the center of mass of $\mathcal{F}_2$ is $w/2$ (recall $\mathcal{F}_2$ is embedded in the equator hypersphere) and that  $\mathcal{F}_2$ is centrally-symmetric about $w/2$. The $176$ antipodal pairs forming an equiangular set of lines were described and shown to be optimal with respect to the so-called relative bound by Delsarte, Goethals, and Seidel in \cite[Example 5.7]{DGS75}. As a projective code coming from a strongly regular graph, it was listed by Waldron in \cite{W2009} and as universally optimal in real projective space $\mathbb{R}\mathbb{P}^{21}$ by Cohn, Kumar, and Minton \cite{CKM16}. 

Let $e\in \Lambda(3)$ with $|e-w|=2$. From Subsection \ref{SharpCode552} we can determine the center of mass $m_{100}$ of $\mathcal{F}_1$ as 
\begin{equation} \label{m100} m_{100}=\frac{w}{2}+\frac{\sqrt{3}}{5}\frac{\sqrt{10}}{2}\frac{e-2w/3}{\sqrt{\sfrac{10}{3}}} =
\frac{w}{2} +\frac{3}{10}\left(e-\frac{2w}{3}\right)=\frac{3(w+e)}{10}.\end{equation}

Given $x\in \mathcal{F}_1$ and $y\in \mathcal{F}_2$ we can compute the cosine between $x-m_{100}$ and $y-w/2$ as follows:
\[ \cos(\phi)=\frac{x-m_{100}}{|x-m_{100}|}\cdot \frac{y-w/2}{|y-w/2|} = \frac{(x-w/2)\cdot(y-w/2)}{(\sfrac{\sqrt{22}}{5}\sfrac{\sqrt{10}}{2})(\sfrac{\sqrt{10}}{2})}=\frac{x\cdot y-3/2}{\sqrt{22}/2},\]
where we use the fact that $x\cdot w=y\cdot w=3$, $|x-m_{100}|=\sqrt{1-3/25}\sqrt{10}/2=\sqrt{11/5}$, and $|y-w/2|=\sqrt{10}/2$. If $|y-x|=2$, then we have $x\cdot y=2$ and $\cos(\phi)=1/\sqrt{22}$. As $w-y$ is symmetrical to $y$ about $w/2$, it belongs to $\mathcal{F}_2$ and in this case $\cos(\phi)=-1/\sqrt{22}$. Since $|x-y|\geq 2$, then $x\cdot y\leq 2$ for any $y\in \mathcal{F}_2$, and in particular for $w-y$, which implies that $x\cdot y\geq 1$. As the inner products of points in the Leech lattice are integers, we obtain $x\cdot y \in \{1,2\}$. Thus, we conclude that the projection of $y$ onto $\mathbb{S}^{21}$ allows us to embed the sharp code $\mathcal{F}_1$ into two parallel hyperplanes, showing $\mathcal{F}_1$ is $2$-stiff  (see the corresponding row in Table \ref{PolarizationULB_Table}), and hence PULB-optimal with the projection point of $y$ being a universal minimum. If we consider the union $\mathcal{F}_1\cup \mathcal{F}_3$, the symmetry about $w/2$ will imply that the projection point has inner products $\pm 1/\sqrt{22}$ with this $200$-point code. As the union is clearly a $3$-design, $\mathcal{F}_1\cup \mathcal{F}_3$ will be PULB-optimal as well, with the same universal minima of the discrete potential $U_h(x,\mathcal{F}_1\cup \mathcal{F}_3)$ as that of $U_h(x,\mathcal{F}_1)$ and $U_h (x,\mathcal{F}_3)$.

On the other hand, the projections of the points in $\mathcal{F}_1\cup \mathcal{F}_3$ onto a hypersphere circumscribing $\mathcal{F}_2$ split $\mathcal{F}_2$ in two parallel hyperplanes, and since it is a spherical $3$-design on $\mathbb{S}^{21}$, the code $\mathcal{F}_2$ is 2-stiff and PULB-optimal with these projections being universal minima of $U_h(x,\mathcal{F}_2)$. 

As the set of universal minima of $U_h(x,\mathcal{F}_1\cup \mathcal{F}_3)$ is the same as that of $U_h(x,\mathcal{F}_1)$, we shall focus on the latter. A universal minimum of $U_h(x,\mathcal{F}_1)$ splits $\mathcal{F}_1$ into two  isometric subcodes of cardinality $50$. They are spherical two-distance sets each forming a Hoffman-Singleton strongly regular graph ${\rm srg}(50, 7,0,1)$. Note that recently it was shown in  \cite{CLL} that this code is an optimal (best packing) spherical code on $\mathbb{S}^{21}$. 

From \cite[Sec. 10.31]{BrSRG} we know there are $352$ such splits for a total of $704$ Hoffman-Singleton subgraphs, which suggests we have as many universal minima. We shall provide a direct proof that the number of universal minima is exactly $704$, thus establishing an alternative proof of this fact. Moreover, our approach allows for obtaining explicit formulas for the coordinates of the universal minima.

We already know that the $352$ vectors in $\mathcal{F}_2$ when projected onto the hypersphere circumscribing $\mathcal{F}_1$, are universal minima. This accounts for $176$ of the said above splits and $352$ of the Hoffman-Singleton subgraphs/subcodes of the Higman-Sims graph/code. From Subsection \ref{SharpCode552} we know that $\mathcal{F}_1$ is associated with a certain universal minimum of $U_h(x,\mathcal{B})$, which is a projection onto the respective hypersphere of either a vector in $\mathcal{C}$ or $\overline{\mathcal{C}}$ (see Subsection \ref{SharpCode552}). Without loss of generality assume this is a vector $e\in \overline{\mathcal{C}}$. As a matter of fact this is the Leech lattice construction of the Higman-Sims graph given in \cite[Sec. 10.31]{BrSRG}, where $w=(1/\sqrt{8})[5,1,\dots,1]$ and $e=(1/\sqrt{8})[1,5,1,\dots,1]$ and the Higman-Sims code is given as one point $(1/\sqrt{8})[4,4,0^{22}]$, $22$ points $(1/\sqrt{8})[1,1,(-3)^1,1^{21}]$, and $77$ points $(1/\sqrt{8})[2,2,2^6,0^{16}]$, where the two's are placed in the $77$ octads starting with two $1$'s (see Lemma \ref{octads-p}(b)). 

Observe, that 
\[ C_{100}:=\mathcal{F}_1=S(w,2)\cap S(e,2)\cap S(0,2)\cap \Lambda_{24}, \ \ C_{352}:=\mathcal{F}_2=S(w,2)\cap S(e,\sqrt{6})\cap S(0,2)\cap \Lambda_{24},\]
and let us define
\[ C_{352}':=\mathcal{F}_4:=S(w,\sqrt{6})\cap S(e,2)\cap S(0,2)\cap \Lambda_{24}.\]

\begin{figure}[htbp]
\includegraphics[width=3.5 in]{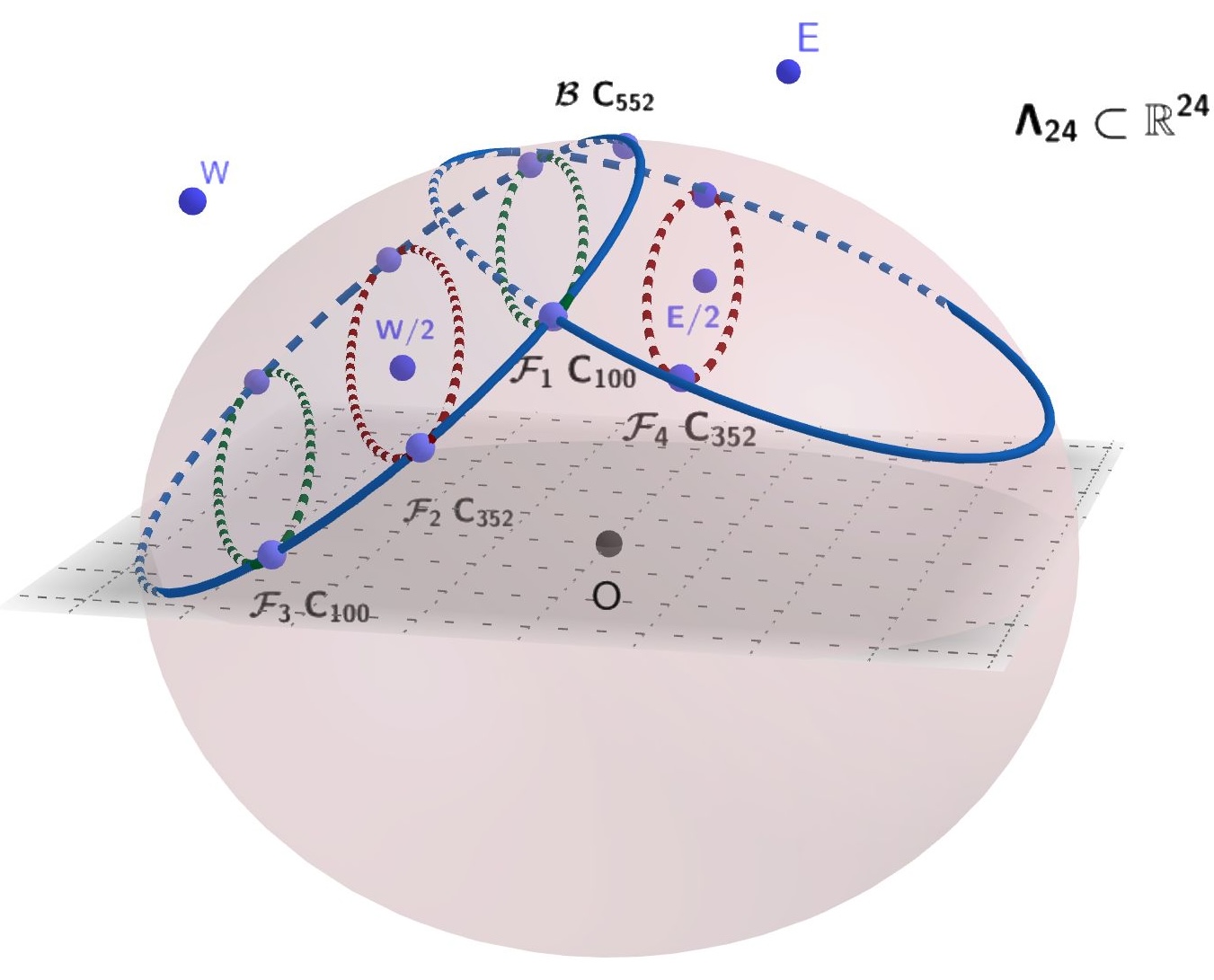}
\caption{The universal polar dual pair generated by the Higman-Sims code.}
\label{110-352PD}
\end{figure}

The symmetry implies that $|\mathcal{F}_2|=|\mathcal{F}_4|$. Just as in the beginning of this subsection we establish that the projections of the $352$ points in $\mathcal{F}_4$ onto the hypersphere determined by $\mathcal{F}_1$ will split the Higman-Sims code into two Hoffman-Singleton subcodes and thus will be universal minima. We shall first prove that these points are distinct from the previously found minima from $\mathcal{F}_2$. Indeed, let $a\in \mathcal{F}_2$ and $b\in \mathcal{F}_4$. We calculate the respective projections as 
\[ \widetilde{a}=m_{100}+\frac{R_{\mathcal{F}_1} (a-w/2)}{R_{\mathcal{F}_2}},\quad \widetilde{b}=m_{100}+\frac{R_{\mathcal{F}_1 }(b-e/2)}{R_{\mathcal{F}_4}},\]
where $m_{100}$ is the center of mass of $\mathcal{F}_1$ and $R_{\mathcal{F}_i}$ denotes the radius of the circumscribing hypersphere of $\mathcal{F}_i$. In particular, we can compute that $R_{\mathcal{F}_2}=R_{\mathcal{F}_4}=\sqrt{5/2}$ and $R_{\mathcal{F}_1}=\sqrt{11/5}$. If $\widetilde{a}=\widetilde{b}$, then $a-w/2=b-e/2$, or $w/2-e/2 \in \Lambda_{24}$. However, $a\cdot w=3$ and $a\cdot e=2$, which implies that $(w/2-e/2)\cdot a =1/2 \notin \mathbb{Z}$. The derived contradiction shows that the respective projection sets are disjoint. Thus, we have accounted for $704$ universal minima obtained as projections of $\mathcal{F}_2$ and  $\mathcal{F}_4$ onto the hypersphere circumscribing  $\mathcal{F}_1$.

We next show that there are no other universal minima of $U_h (x,\mathcal{F}_1)$. Let $y$ be an arbitrary minimum. Then 
\[\frac{(y-m_{100})\cdot (z_i-m_{100})}{\sqrt{11/5}\sqrt{11/5}}\in \left\{ \pm \frac{1}{\sqrt{22}} \right\}, \quad \forall  z_i\in \mathcal{F}_1,\] and the projection of $y$ onto the hypersphere circumscribing $\mathcal{F}_2$ is given by
\[\widetilde{y}=\frac{w}{2}+\frac{\sqrt{5/2}(y-m_{100})}{\sqrt{11/5}}.\]
Let $M:={\rm ispan}(w,\mathcal{F}_1)$ be the sub-lattice spanned by the vectors in $\mathcal{F}_1$ and $w$. Utilizing the coordinate representation of these vectors above we obtain the Smith normal form as ${\rm diag}(1,2^{10},4^{12},24)$, yielding $|\Lambda_{24}:M|=|M^*:\Lambda_{24}|=6$. We compute that $\widetilde{y}\cdot w=3$ and verify using \eqref{m100} that $(\widetilde{y}-w/2)\cdot m_{100}$ (note $w\cdot (y-m_{100})=0$), which yields for all $z_i \in \mathcal{F}_1$ 
 
\[ \widetilde{y}\cdot z_i -\frac{3}{2}=(\widetilde{y}-w/2)\cdot z_i=(\widetilde{y}-w/2)\cdot (z_i-m_{100})\in \left\{\pm \frac{1}{2}\right\}.\] 
Thus, $\widetilde{y}\in M^*$. The vectors $w/3$ and $w/2-e/2$ are also in $M^*$ with respective orders $3$ and $2$. Then the sup-lattice $M^*$ splits into six cosets
\[M^*=\bigcup_{i=0}^1 \bigcup_{k=0}^2 \left( \frac{kw}{3}+\frac{j(w-e)}{2}+\Lambda_{24} \right).\]
Since 
\[ \left|\widetilde{y}-\frac{kw}{3}-\frac{j(w-e)}{2} \right|^2= 4+\frac{2k^2}{3}+j^2-2k-j+\frac{2kj}{3}=4-2k+\frac{2k(k+j)}{3},\]
$\widetilde{y}$ may belong to only three cosets, namely $\Lambda_{24}$, $(w-e)/2+\Lambda_{24}$, and $7w/6-e/2+\Lambda_{24}$. If $\widetilde{y}\in \Lambda_{24}$, then $\widetilde{y}\in \mathcal{F}_2$. If $\widetilde{y}=w/2-e/2+g$, for some 
$g\in \Lambda_{24}$, then $|\widetilde{y}-w/2|=|g-e/2|=\sqrt{5/2}$, so $g\in \mathcal{F}_4$. Note that the collection of vectors $w/2-e/2+g$, $g\in \mathcal{F}_4$, belongs to the hypersphere circumscribing $\mathcal{F}_2$, but is disjoint from $\mathcal{F}_2$, because $w/2-e/2 \notin \Lambda_{24}$.

We shall reject the last case, $\widetilde{y} \in (7w/6-e/2+\Lambda_{24})$. To do so, we recall that the set of universal minima is antipodal, which implies that $w-\widetilde{y}$ (antipode of $\widetilde{y}$ with respect to $w/2$) is also a projection of a universal minimum of $U_h(x,\mathcal{F}_1)$ and itself has to belong to an admissible coset. Since $\widetilde{y}\notin \Lambda_{24}$ and $\widetilde{y}\notin (w/2-e/2+\Lambda_{24})$, then $w-\widetilde{y} \notin \Lambda_{24}$ and $w-\widetilde{y} \notin (w/2-e/2+\Lambda_{24})$ (recall that $w\in \Lambda_{24}$). So, we conclude that $w-\widetilde{y} \in (7w/6-e/2+\Lambda_{24})$. Adding $\widetilde{y}$ and $w-\widetilde{y}$ we reach a contradiction that $w/3 \in \Lambda_{24}$. Thus, any universal minimum of $U_h(x,\mathcal{F}_1)$ projected onto the circumscribed sphere of $\mathcal{F}_2$ lies in either of the disjoint sets $\mathcal{F}_2$ or $ w/2-e/2+\mathcal{F}_4$.

We first note that both codes, $\mathcal{F}_2$ and $\mathcal{F}_4$, as projected on $\mathbb{S}^{21}$ are spherical $3$-designs. Moreover, as we established that the points of both $\mathcal{F}_2$ and $\mathcal{F}_4$ have inner products $\pm 1/\sqrt{22}$ with points in $\mathcal{F}_1$, we conclude that $\mathcal{F}_2 \cup \mathcal{F}_4$ is $2$-stiff by reciprocity. If $y$ is a universal minimum for $U_h(x,\mathcal{F}_2\cup\mathcal{F}_4)$, then $I(y,\mathcal{F}_2\cup\mathcal{F}_4)\subseteq \{\pm 1/\sqrt{22}\}$ (see \eqref{I(x,C)} for the definition of the inner product set $I(x,C)$). This yields that $I(y,\mathcal{F}_2)\subseteq \{\pm 1/\sqrt{22}\}$ and $I(y,\mathcal{F}_4)\subseteq \{\pm 1/\sqrt{22}\}$, implying that $y$ is a universal minimum for both potentials $U_h (x,\mathcal{F}_2)$ and $U_h (x,\mathcal{F}_4)$. 

Suppose now that $y$ is a universal minimum of $U_h(x,\mathcal{F}_2)$ as embedded in the Leech lattice, that is 
\begin{equation} \label{F_2Eq}
\frac{(y-w/2)\cdot (x_j-w/2)}{(\sqrt{10}/2)(\sqrt{10}/2)}\in \left\{ \pm \frac{1}{\sqrt{22}}\right\}, \quad \forall  x_j\in \mathcal{F}_2.\end{equation}
In the Leech lattice construction of the Higman-Sims code above, the point $e$ identifies the universal minimum for $U_h(x,C_{552})$ associated with $\mathcal{F}_1$ and it lies in $\overline{\mathcal{C}}$ (see Subsection \ref{SharpCode552}). 
 The projection of $y$ onto the hypersphere circumscribing $\mathcal{F}_1$ is given by
\[\widetilde{y}=m_{100}+\frac{(y-w/2)\sqrt{11/5}}{\sqrt{10}/2}.\]
 Let $M:={\rm ispan}(e,\mathcal{F}_2)$. We shall show that $\widetilde{y}\in M^*$. Clearly, $\widetilde{y}\cdot e=3$. Utilizing \eqref{F_2Eq} we compute
 \[ (\widetilde{y}-m_{100})\cdot (x_j-w/2)= \frac{(y-w/2)\cdot (x_j-w/2)\sqrt{11/5}}{\sqrt{10}/2}=\pm \frac{1}{2}, \]
which implies
\[ \widetilde{y}\cdot x_j = \frac{3}{10}(w+e)\cdot x_j\pm \frac{1}{2} \in \{ 1,2\}.\]
 
 The Smith normal form of $M$ is ${\rm diag}(1,2^{11},4^{11},80)$, implying $|\Lambda_{24} : M|=|M^* : \Lambda_{24}|=10$. We observe that $e/2$ and $(w+e)/5$ both belong to $M^*$. Indeed, $x_i\cdot e/2=1$, $x_i\cdot(w+e)/5=1$, $i=1,\dots,352$, $e\cdot e/2=3$, and $e\cdot (w+e)/5=2$. This shows the representation of $M^*$ as union of ten cosets
 \[ M^*=\left[ \bigcup_{k=0}^4 \left( \frac{k(w+e)}{5}+\Lambda_{24}\right)\right]\bigcup \left[ \bigcup_{k=0}^4  \left( \frac{e}{2}+ \frac{k(w+e)}{5}+\Lambda_{24}\right)\right].\]
We first consider the possibility that $\widetilde{y}$ belongs to the first five cosets. If so, then
\[ \left| \widetilde{y}-\frac{k(w+e)}{5}\right|^2= 4+\frac{20k^2}{25}-2k\frac{\widetilde{y}\cdot (w+e)}{5}=4+\frac{4k(k-3)}{5}\]
has to be an even integer, which is only possible when $k=0$ or $k=3$. In the first case $\widetilde{y} \in \Lambda_{24}$, or $\widetilde{y} \in \mathcal{F}_1$. In the second case, the antipodal to $\widetilde{y}$ point in the hypersphere circumscribing $\mathcal{F}_1$ belongs to the Leech lattice, namely $(2m_{100}-\widetilde{y}) \in \Lambda_{24} $ or $(2m_{100}-\widetilde{y}) \in \mathcal{F}_1$.

We easily discard the remaining cosets from 
\[ \left| \widetilde{y}-\frac{e}{2} - \frac{k(w+e)}{5}\right|^2= \frac{5}{2}+\frac{20k^2}{25}-2k\frac{(\widetilde{y}-e/2)\cdot (w+e)}{5}=\frac{5}{2}+\frac{2k(2k-1)}{5}.\]
which clearly is not an even integer.

This implies that there are no other universal minima and the following theorem holds.

\begin{theorem}\label{100PULBpair} Denote the corresponding projections of $\mathcal{F}_i\subset \Lambda_{24}$, $i=1,2,3,4$, onto  $\mathbb{S}^{21}$ with $\mathcal{F}_i$ as well. Then for any potential $h$ with $h^{(4)}(t)> 0$, $t\in (-1,1)$, the codes $\mathcal{F}_2$, $\mathcal{F}_4$, and $\mathcal{F}_2 \cup \mathcal{F}_4$ are PULB-optimal; i.e., they attain bound \eqref{PolarizationULB} 
\begin{equation*}\label{F2}
\begin{split}
 m^h (\mathcal{F}_i)&\geq 176h\left( -1/\sqrt{22}\right)+176h\left(1/\sqrt{22}\right), \quad i=2,4,\\
  m^h (\mathcal{F}_2\cup \mathcal{F}_4)&\geq 352h\left( -1/\sqrt{22}\right)+352h\left(1/\sqrt{22}\right)
 \end{split}
\end{equation*}
with universal minima at the points of the sharp code $\mathcal{F}_1=(22,100,3)$ and its antipode $\mathcal{F}_3$.

Moreover, $(\mathcal{F}_1\cup \mathcal{F}_3,\mathcal{F}_2\cup \mathcal{F}_4)$ is a maximal PULB-optimal pair (and a fortiori universal polar dual pair) with normalized discrete potentials achieving the same extremal value
\begin{equation*}
\frac{m^h(\mathcal{F}_2\cup \mathcal{F}_4)}{704}= \frac{m^h(\mathcal{F}_1\cup\mathcal{F}_3)}{200} \left( =\frac{m^h (\mathcal{F}_2)}{352}=\frac{m^h (\mathcal{F}_1)}{100}\right).\end{equation*}
\end{theorem}

\subsection{The maximal PULB-optimal pair generated by the sharp codes $C_{112}:=(21,112,3)$ and $C_{162}:=(21,162,3)$} \label{112-162} Our last PULB-optimal pair of sharp codes embedded in the Leech lattice is generated by the two sharp codes $C_{112}$ and $C_{162}$. Both are derived codes found in the McLaughlin code $C_{275}:=(22,275,4)$ and induce strongly regular graphs, respectively ${\rm srg}(112,30,2,10)$ and $ {\rm srg}(162,56,10,24)$ (see  \cite[Section 10.34 and 10.48]{BrSRG}). Let us consider all three codes as embedded in $\mathcal{B}=C_{552}=(23,552,5)$ (see Subsection \ref{SharpCode552}). The distances between distinct Leech lattice points of $\mathcal{B}$ are $\{2,\sqrt{6},\sqrt{10}\}$, where $\sqrt{10}$ is only between diametral points in $\mathcal{B}$. As $C_{275}$ does not include any diametral points, the distinct distances within the code will be $2$ and $\sqrt{6}$, where a distance of $\sqrt{6}$ indicates the adjacency between points considered as vertices of the corresponding graphs.

\begin{figure}[htbp]
\includegraphics[width=3.5in]{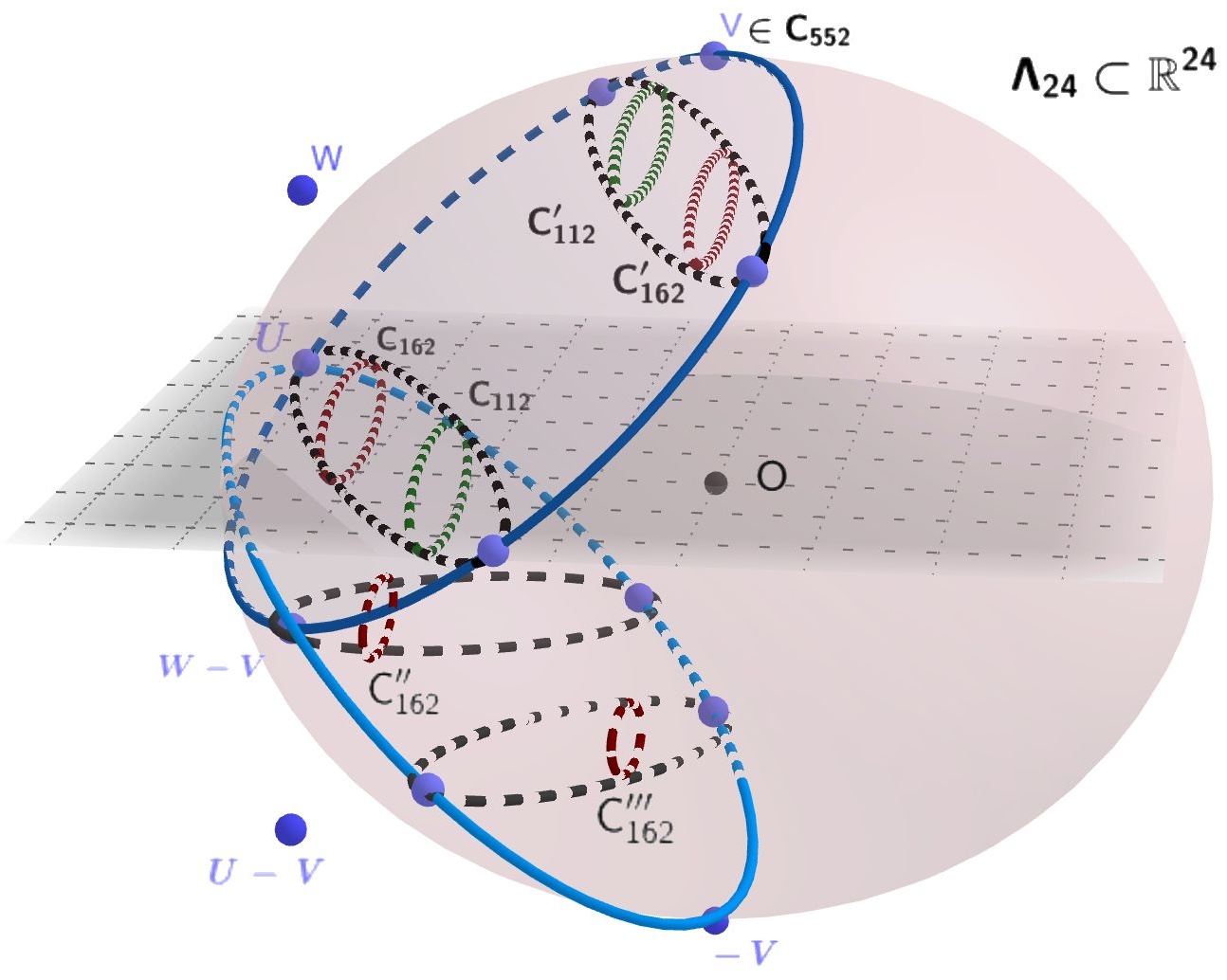}
\caption{Polarization pair $(21,112,3)$ and $(21,162,3)$ in $\Lambda_{24}$.}
\label{fig-112}
\end{figure}

Fix a point $u\in C_{275}$ and let $C_{112}$ be the corresponding derived code associated with the choice of $u$. Fix $y\in C_{112}$. Recall that the McLaughlin code $C_{275}$ is a strongly regular graph itself ${\rm srg}(275,112,30,56)$. So, there are $112$ points in $C_{275}$ that are at distance $\sqrt{6}$ from $y$. One of these points is $u$ itself. As the degree of the subconstituent ${\rm srg}(112,30,2,10)$ is $30$, there are $30$ points in $C_{112}$ that are at distance $\sqrt{6}$ to $y$ and $81$ points in $C_{162}$, which are at distance $\sqrt{6}$ away from $y$. Hence, these $81$ points lie on one hyperplane. The other $81$ points in $C_{162}$ necessarily are at a distance $2$ from $y$ and thus also lie on a hyperplane. Both hyperplanes are orthogonal to the projection of $y$ onto the  hypersphere containing $C_{162}$, and hence are parallel. This implies that $C_{162}$ is a $2$-stiff code with the said projection being a universal minimum of $U_h (x,C_{162})$. A similar argument implies that $C_{112}$ is a $2$-stiff code and that the projections of points of $C_{162}$ are universal minima of the discrete potential $U_h (x,C_{112})$. 

While this shows that $(C_{112},C_{162})$ is a PULB-optimal pair, finding the associated maximal PULB-optimal pair is much more complex. From the symmetry it is clear that the antipodal codes $-C_{112}$ and $-C_{162}$ have to be included into the respective sets of minima. While this will suffice for the universal minima set associated with $C_{162}$, we will see that the universal minima set associated with $C_{112}$ includes another pair of antipodal copies isometric  to $C_{162}$. 

For the purpose of establishing maximality of the PULB-pair we recall the coordinate representation of $C_{275}$ given in \cite[Section 10.61]{BrSRG}. For brevity, we shall omit the factor $(1/\sqrt{8})$ from all coordinate representations below. Given $w:=[5,1,1,1^{21}]$ in $\Lambda(3)$ and $v:=[4,4,0,0^{21}]$ in $\Lambda(2)$ the $275$ vectors of $C_{275}$ have inner product $3$ with $w$ and $1$ with $v$. There are $22$ points of type $[1,1,(-3)^1,1^{21}]$, $77$ of type $[3,-1,1^{16},(-1)^6]$, and $176$ of type $[2,0,2^7,0^{15}]$. We note that the vectors of $\Lambda(2)$ that have inner products $3$ with $w$ and $2$ with $v$, respectively, form another copy $w-C_{275}$ of the McLaughlin code embedded in $\mathcal{B}$ that is antipodal to $C_{275}$ w.r.t. $w/2$.

Fixing $u:=[1,1,-3,1^{21}]\in C_{275}$, the code $C_{162}$ consists of vectors $\{ x_j\}_{j=1}^{162}$ having inner product $2$ with $u$, namely $21$ points of type $[1,1,1,(-3)^1,1^{20}]$, $21$ of type $[3,-1,-1,(-1)^5,1^{16}]$, and $120$ of type $[2,0,0,2^7, 0^{14}]$, while $C_{112}$ consists of vectors $\{ z_i \}_{i=1}^{112}$ having inner product $1$ with $u$, that is $56$ vectors of type $[3,-1,1,(-1)^6, 1^{15}]$ and $56$ vectors of type $[2,0,2,2^6, 0^{15}]$. The center of mass of $\mathcal{B}$ is $w/2$ and the radius of the circumscribing hypersphere $S(0,2)\cap S(w,2)$ is $R_\mathcal{B}=R_{552}=\sqrt{5/2}$. We compute the centers of mass of $C_{275}$, $C_{162}$, and $C_{112}$ to be 
\begin{equation}\label{mass_centers}
m_{275}= \frac{3w-v}{5}, \quad m_{162}= \frac{3w+u-v}{6}, \quad m_{112}=\frac{3w-u-v}{4},
\end{equation} 
respectively, and the radii of their circumscribed hyperspheres to be 
\[ R_{275}=2\sqrt{\frac{3}{5}}, \quad R_{162}=\sqrt{\frac{7}{3}},\quad R_{112}=\frac{3}{2}.\]

Let $y$ be any universal minimum of $U_h(x,C_{112})$ taken on the hypersphere of minimal dimension circumscribing  $C_{112}$ (as embedded in the Leech lattice).  The corresponding entry of $C_{112}$ in Table \ref{PolarizationULB_Table} yields that 
\begin{equation}\label{112y} \frac{(y-m_{112})\cdot (z_i-m_{112})}{(3/2)  (3/2) } \in \left\{ \pm \frac{1}{\sqrt{21}}\right\}, \quad i=1,\ldots,112. \end{equation}

Let $\widetilde{y}$ be the projection of $y$ onto the hypersphere of minimal dimension circumscribing $C_{162}$ (as embedded in the Leech lattice)
\[ \widetilde{y}=m_{162}+ \frac{y-m_{112}}{3/2}\sqrt{\frac{7}{3}}.\]
Let $M:={\rm ispan}(w,v,u,\{ z_i\}_{i=1}^{112})$ be the sub-lattice of $\Lambda_{24}$ generated by these $115$ vectors. First, we observe that $\widetilde{y}\cdot u =2$, $\widetilde{y}\cdot v=1$, and $\widetilde{y}\cdot w=3$. Next, we calculate using \eqref{mass_centers} and \eqref{112y} that for every $i=1,\ldots,112$,
\begin{eqnarray*}
\pm \frac{1}{2}&=&(\widetilde{y}-m_{162})\cdot (z_i-m_{112}) = (\widetilde{y}-u)\cdot (z_i-m_{112})\\
&=&\widetilde{y}\cdot z_i-\widetilde{y}\cdot m_{112}-u\cdot z_i+u\cdot m_{112}=\widetilde{y}\cdot z_i -\frac{3}{2},
\end{eqnarray*}
which implies $\widetilde{y}\cdot z_i\in \{1,2\}$ for all $i=1,\dots,112$.  We conclude that $\widetilde{y}\in M^*$.

The Smith Normal form of $M$ is ${\rm diag}(1,2^{11},4^{10},12,24)$, which yields $|\Lambda_{24}:M|=9$. Recall that the sub-lattice $L={\rm ispan}(\mathcal{B})$ considered in Subsection \ref{SharpCode552} has $|\Lambda_{24} : L|=3$, and since $M\triangleleft L$, we derive that $|L:M|=3$. Direct computation implies that the vectors $w/3$ and $(u-v)/3$ belong to $M^*$, and that $w/3\in L^*$ (see Subsection \ref{SharpCode552}) and $(u-v)/3\notin L^*$. Indeed, $x_j \in L$ for any $x_j\in C_{162}$, but $x_j \cdot (u-v)/3=1/3\notin \mathbb{Z}$. Therefore, we can identify the nine cosets generated by $\Lambda_{24}$ in $M^*$ as follows:
\[ M^*=\bigcup_{j=-1}^1 \bigcup_{k=-1}^1 \left(\frac{j(u-v)}{3}+\frac{kw}{3}+\Lambda_{24}\right). \]
Next we evaluate
\[ \left| \widetilde{y}-\frac{j(u-v)}{3}-\frac{kw}{3} \right|^2=4+\frac{2j^2}{3}+\frac{2k^2}{3}-\frac{2j}{3}-2k=4 -2k+\frac{2(j^2-j+k^2)}{3}, \ j,k=-1,0,1.\]
So, if $\widetilde{y}$ is in a $(j,k)$-coset, then $(j^2-j+k^2)/3$ has to be an integer, which is true for  the following four cases. A key role in considering the four cases is played by the following subspace
\begin{equation}\label{G} G:=\left( {\rm span}\{u,v,w\}\right) ^{\perp}.\end{equation}
\vskip 2mm
\noindent \underline{Case 1:} $j=k=0$: In this case $\widetilde{y} \in \Lambda_{24}$, or equivalently $\widetilde{y} \in C_{162}$. 
\vskip 2mm
\noindent \underline{Case 2:} $j=1, k=0$: Assume $\widetilde{y}=\frac{u-v}{3}+g$, $g\in \Lambda_{24}$. Then $\widetilde {y}-\frac{u-v}{3}\in \Lambda_{24}$. Denote $C_{162}':=w-C_{162}$ and $m_{162}^\prime:=w-m_{162}$, and let
$$
H:=\{x\in \mathbb R^{24} : w\cdot x=3,\ v\cdot x=1,\text{ and }u\cdot x=2\}=m_{162}+G
$$
and
$$
H':=\{x\in \mathbb R^{24} : w\cdot x=3,\ (w-v)\cdot x=1,\text{ and }(w-u)\cdot x=2\}=m_{162}^\prime+G.
$$
Then the hypersphere of minimal dimension circumscribing $C_{162}$ is $H\cap S(0,2)$ with $\widetilde{y}\in H\cap S(0,2)$ and $C_{162}=H\cap S(0,2)\cap \Lambda_{24}$. Note that $m_{162}=w/2+(u-v)/6$ and $m_{162}^\prime=w/2-(u-v)/6$. It is not difficult to see that $H-\frac {u-v}{3}=H'=w-H$. Moreover, for every $x\in H\cap S(0,2)$, we have $\left|x-\frac {u-v}{3}\right|^2=4$ and for every $x\in S(0,2)$ such that $w\cdot x=3$, there holds $\left|w-x\right|^2=4$. Consequently,
$$
H\cap S(0,2)-\frac {u-v}{3}=H'\cap S(0,2)=w-H\cap S(0,2).
$$
Therefore,
$$
\widetilde {y}-\frac{u-v}{3}\in H'\cap S(0,2)\cap \Lambda_{24}=(w-H\cap S(0,2))\cap \Lambda_{24}=w-H\cap S(0,2)\cap \Lambda_{24},
$$
and, consequently, $\widetilde {y}-\frac{u-v}{3}\in w-C_{162}=C'_{162}$. Thus, $\widetilde {y}\in \frac{u-v}{3}+C'_{162}$. Since $C'_{162}\subset w-H\cap S(0,2)$, we have $\frac{u-v}{3}+C'_{162}\subset H\cap S(0,2)$. Since $\left|\frac {u-v}{3}+w\right|^2=\frac {20}{3}\notin 2\mathbb Z$, the codes $C_{162}$ and $\frac{u-v}{3}+C'_{162}$ are disjoint subsets of $H\cap S(0,2)$.
  
Recall that for any $x_j \in C_{162}$, we have $x_j \cdot z_i\in \{1,2\}$, so if $g_j:=w-x_j$ then $g_j\cdot z_i\in \{2,1\}$. This implies that points of $C'_{162}$ and, hence of $\frac {u-v}{3}+C_{162}'$ identify a new (disjoint) collection of $162$ universal minima (note that $(u-v)\cdot z_i=0$ for all $i=1,\dots, 112$).
\vskip 2mm
\noindent \underline{Case 3:} $j=-1, k=1$:  Suppose $\widetilde{y}=-(u-v)/3+w/3+g$, $g\in \Lambda_{24}$. Since $\widetilde{y}\cdot z_i \in \{1,2\}$, $u\cdot z_i=v\cdot z_i=1$, and $w\cdot z_i=3$, we have that $g\cdot z_i\in \{0,1\}$, so $g$ splits $C_{112}$ into two subcodes embedded in parallel hyperplanes. We also compute directly that $g\in \Lambda(2)$. From \eqref{mass_centers} we have
\[ |\widetilde{y}-m_{162}|=\sqrt{7/3}=\left| g-\frac{u-v}{3}+\frac{w}{3}-\frac{3w+u-v}{6}\right|=\left| g-\left( \frac{u-v}{2}+\frac{w}{6}\right) \right|.
\]
As a matter of fact $\widetilde{y}-m_{162}=g-m_{162}^{\prime\prime}$, where $m_{162}^{\prime\prime}:=(u-v)/2+w/6$. The collection of points 
\[|g-m_{162}''|=\sqrt{7/3}, \quad g\in ((u-v)/3-w/3+H\cap S(0,2))\cap \Lambda_{24},\] 
is a copy $C_{162}''$ of the $162$-code, which is obtained by exchanging the roles of the length $\sqrt{6}$ vectors $u-v$ and $w$. Indeed, the code $C_{162}$ is obtained as the intersection below by selecting the vectors in this order $\{0,w,v,u\}$
\[  C_{162}=S(0,2)\cap S(w,2)\cap S(v,\sqrt{6}) \cap S(u,2)\cap \Lambda_{24} = H\cap S(0,2)\cap \Lambda_{24}, \]
while the code $C_{162}''$ is obtained from the vectors $\{ 0, u-v, -v, w-v\}$ in this order
\[ C_{162}'':=S(0,2)\cap S(u-v,2)\cap S(-v,\sqrt{6}) \cap S(w-v,2)\cap \Lambda_{24}. \]
Similar to Case $2$ we define 
 \[ H^{\prime\prime}:=m_{162}^{\prime\prime}+G.\]
Then $C_{162}^{\prime\prime}=H^{\prime\prime} \cap S(0,2)\cap \Lambda_{24}$ and $\widetilde {y}\in -\frac{u-v}{3}+\frac{w}{3}+C''_{162}\subset H\cap S(0,2)$.
\vskip 2mm

 \noindent \underline{Case 4:} $j=-1, k=-1$: In this case we shall choose the representative $\widetilde{y}=-(u-v)/3+2w/3+g$, $g\in \Lambda_{24}$, which guarantees $g\in \Lambda(2)$. Indeed,
 \[ |g|^2=|\widetilde{y}+(u-v)/3-2w/3|^2=|\widetilde{y}|^2+\frac{|u-v|^2}{9}+\frac{4|w|^2}{9}+\frac{2\widetilde{y}\cdot (u-v)}{3}-\frac{4\widetilde{y}\cdot w}{3}-\frac{4(u-v)\cdot w}{3}=4.\]
We compute
 \[ |\widetilde{y}-m_{162}|=\sqrt{7/3}=\left| g-\frac{(u-v)}{3}+\frac{2w}{3}-\frac{w}{2}-\frac{u-v}{6}\right|=\left| g-\left( \frac{u-v}{2}-\frac{w}{6}\right) \right| =|g-m_{162}'''|,\]
 where $m_{162}^{\prime\prime\prime}:=(u-v)-m_{162}^{\prime\prime}=(u-v)/2-w/6$. We define
 \[ H^{\prime\prime\prime}:=m_{162}^{\prime\prime\prime}+G,\]
which identifies a fourth copy $C_{162}''':=(u-v)-C_{162}''$. Note that $\widetilde {y}\in -\frac{u-v}{3}+\frac{2w}{3}+C'''_{162}\subset H\cap S(0,2)$, for a total of $648$ possible universal minima, all distinct as belonging to four different cosets. 

Next we show that \textcolor{blue}{all} the vectors in $C_{162}''$ and $C_{162}'''$ identify universal minima for $U_h(x,C_{112})$. For this purpose we determine the coordinates in $C_{162}''':=\{y_j\}_{j=1}^{162}$ (recall that $C_{112}=\{z_i\}_{i=1}^{112}$ consists of $56$ vectors of type $[3,-1,1,(-1)^6, 1^{15}]$ and $56$ vectors of type $[2,0,2,2^6, 0^{15}]$, the factor $1/\sqrt{8}$ being omitted). We shall prove that
\begin{equation}\label{162Minima} 
\frac{(y_j-m_{162}''' )\cdot (z_i -m_{112})}{\left( \sqrt{7/3} \right) (3/2)} = \frac{ y_j \cdot z_i -y_j \cdot m_{112}}{\left( \sqrt{21}/2\right)} \in \left\{ \pm \frac{1}{\sqrt{21}} \right\},\end{equation}
where we used the fact that $(z_i - m_{112}) \cdot m_{162}'''=0$. Indeed, $m_{162}'''=(u-v)/2-w/6 \in {\rm span}\{u,v,w\}$ and $(z_i-m_{112} ) \in G$ (see \eqref{G}).

The vectors $y_j$ have to belong to one of the three types $(A)$, $(B)$, and $(C)$ of Lemma \ref{lem-196560} and have inner products $y_j \cdot (-v)=2$, $y_j \cdot (u-w)=2$, and $y_j \cdot (u-v)=3$. As $-v=[-4,-4,0,0^{21}]$ and $u-w=[-4,0,-4,0^{21}]$ we obtain that there are $5\times 21=105$ vectors of type $(A)$, namely $[-2,-2,-2,(\pm 2)^5,0^{16}]$ following octads starting with $111$. Indeed, from $(u-v)=[-3,-3,-3,1^{21}]$ we conclude there is exactly $1$ additional negative sign in the last $21$ coordinates. There are $56$ type $(B)$ vectors $[-3,-1,-1,(-1)^6,1^{15}]$, which follow the octads starting with $011$, and one additional type $(C)$ vector $[0,-4,-4,0^{21}]$.

We first consider the vector $c:=[0,-4,-4,0^{21}]$, which clearly has inner products with vectors in $C_{112}$ in $\{-1,0\}$. From the coordinate representation $m_{112}=\left[ \frac{5}{2},-\frac{1}{2},\frac{3}{2},\left( \frac{1}{2}\right)^{21} \right]$ we have $a\cdot m_{112} = -1/2$, which yields \eqref{162Minima} in this case.

Next, fix a vector of type  $b_j:=[-3,-1,-1,(-1)^6,1^{15}]$. Recall that two distinct octads may have only $0, 2, 4$ intersections ($1$'s at the same position).  When scalar multiplied with a vector of type $[3,-1,1,(-1)^6, 1^{15}]\in C_{112}$, the respective octads already share $1$ at the second position, so they will have $1$ or $3$ intersections in the last $21$ coordinates. In the first case the inner product will be $-1$ and in the second case $0$. As $b_j\cdot m_{112}=-1/2$ (recall that all vectors have a factor of $1/\sqrt{8}$), we conclude \eqref{162Minima}. The vectors of type $[2,0,2,(2)^6,0^{15}]\in C_{112}$ already have a shared $1$ with $b_j$ in the corresponding octad representation at position $3$, so in this case there may be $1$ or $3$ shared $1$'s in positions $4,\dots,24$. If there is $1$ intersection, then the inner product is $0$, otherwise it is $-1$, which along with  $b_j\cdot m_{112}=-1/2$ exhausts this case.

Finally, fix a vector $a_j:=[-2,-2,-2,(-2)^1,(2)^4,0^{16}]$. We have  $a_j\cdot m_{112}=-1/2$. When computing the inner products with vectors of type $[3,-1,1,(-1)^6, 1^{15}]\in C_{112}$, we already have two intersections at positions $1$ and $2$. If there are no other intersections the inner product is $0$. When there are two additional intersections, there are two cases to consider depending on the position of the fourth negative sign in $a_j$. If $-2$ is in a shared intersection, the inner product is $0$, otherwise it is $-1$, and so \eqref{162Minima} follows. Lastly, when taking the inner product with vectors of type $[2,0,2,(2)^6,0^{15}]\in C_{112}$ we obtain $-1$ in the case of no additional intersections. In the case of two additional intersections and both positive signs we get an inner product of $0$ and if one of the intersection positions has a negative sign, the inner product is $-1$. This completes the proof of \eqref{162Minima}. 

Using that $C_{162}''=(u-v)-C_{162}'''$ and that $m_{162}''=(u-v)-m_{162}'''$ we can derive from \eqref{162Minima} that for any $e\in C_{162}''$ we have 
$$\frac{(e-m_{162}'')\cdot (z_i -m_{112})}{\left( \sqrt{7/3}\right)(3/2)}=\pm \frac{1}{\sqrt{21}}.$$

As the corresponding $\widetilde{y}$ vectors belong to disjoint cosets, we have found the projections of all $648$ vectors in the  minimal set of $U_h(x,C_{112})$ onto the circumscribed sphere of $C_{162}$ as the set
\begin{equation}\label{4C162}
D:=C_{162}\cup \left(\frac{u-v}{3}+C_{162}'\right)\cup \left(-\frac{u-v}{3}+\frac{w}{3} + C_{162}''\right) \cup \left(-\frac{u-v}{3}+\frac{2w}{3}+C_{162}'''\right) .
\end{equation}

Next, we consider the collection of minima of the potential $U_h(x,D)$ as projected on $\mathbb{S}^{20}$. All four subcodes of $D$ in \eqref{4C162} are $3$-designs, which implies that $D$ is a $3$-design. Since the vectors of $D$ have inner products $\pm 1/\sqrt{21}$ with points in $C_{112}$ (as projected on $\mathbb{S}^{20}$), we conclude that $D$ is $2$-stiff and hence PULB-optimal. If $y$ is a (universal) minimum of $U_h(x,D)$, then $I(y,D)=\{\pm 1/\sqrt{21}\}$, which implies that $I(y,C_{162})\subseteq \{\pm 1/\sqrt{21}\}$. Theorem \ref{PULB} now implies that equality holds and $y$ is a universal minimum of $U_h (x,C_{162})$. 

We next determine the set of universal minima associated with $C_{162}$. Fix a universal minimum $y$ of $U_h(x,C_{162})$ lying on the circumscribed hypersphere of $C_{162}$. Then
\[ \frac{(y-m_{162})\cdot (x_j-m_{162})}{(\sqrt{7/3})  (\sqrt{7/3}) } \in \left\{ \pm \frac{1}{\sqrt{21}}\right\}, \quad j=1,\ldots,162. \]
Let $\widetilde{y}$ be the (geodesic) projection of $y$ onto the circumscribed hypersphere of $C_{112}$, i.e.
\[ \widetilde{y}=m_{112}+ \frac{3(y-m_{162})}{2\sqrt{7/3}}.\]
We have $\widetilde{y}\cdot w=3$, $\widetilde{y}\cdot u=1$, and $\widetilde{y}\cdot v=1$. We compute that
\begin{eqnarray*}
\pm \frac{1}{2}&=&(\widetilde{y}-m_{112})\cdot (x_j-m_{162}) = (\widetilde{y}-u)\cdot (x_j-m_{162})\\
&=&\widetilde{y}\cdot x_j-\widetilde{y}\cdot m_{162}-u\cdot x_j+u\cdot m_{162}=\widetilde{y}\cdot x_j -\frac{3}{2},
\end{eqnarray*}
so $\widetilde{y}\cdot x_j \in \{1,2\}$, $j=1,\dots,162$.
Let $M:={\rm ispan}(w,v,u,\{x_j\}_{j=1}^{162})$ be the sub-lattice determined by these $165$ vectors. Clearly, $\widetilde{y} \in M^*$. 

From the Smith Normal form ${\rm diag}(1,2^{10},4^{12},24)$ we find the index $|\Lambda_{24}:M|=6$. We note that $M\triangleleft L$ and $|\Lambda_{24}:L|=3$, where $L={\rm ispan}(\mathcal{B})$ is the sub-lattice considered in Subsection \ref{SharpCode552}. We directly verify that the two vectors $w/3 $ and $2m_{112}-w=(w-u-v)/2$ belong to $M^*$. As a matter of fact, we already know $w/3 \in L^*$. Since $z_i \cdot(w-u-v)/2=1/2\notin \mathbb{Z}$, we have $(w-u-v)/2 \in M^*\setminus L^*$. Therefore, the six equivalence classes in this case are
\begin{eqnarray*} 
M^*&=&\Lambda_{24}\cup \left(\frac{w}{3}+\Lambda_{24}\right)\cup\left(\frac{2w}{3}+\Lambda_{24}\right)\\
& & \cup \left(\frac{w-u-v}{2}+\Lambda_{24}\right)\cup\left(\frac{w-u-v}{2}+\frac{w}{3}+\Lambda_{24}\right) \cup \left(\frac{w-u-v}{2}+\frac{2w}{3}+\Lambda_{24}\right).
\end{eqnarray*}
Clearly, if $\widetilde{y}\in \Lambda_{24}$ then $\widetilde{y}\in C_{112}$. From $|\widetilde{y}-w/3|^2=4+2/3-2=8/3$, $|\widetilde{y}-2w/3|^2=4+8/3-4=8/3$,  we can discard the other two classes in the first row. We evaluate 
\[|\widetilde{y}-w/3-(w-u-v)/2|^2=8/3+1-(w-u-v)\cdot (\widetilde{y}-w/3)=8/3,\]
and 
\[|\widetilde{y}-2w/3-(w-u-v)/2|^2=8/3+1-(w-u-v)\cdot (\widetilde{y}-2w/3)=8/3,\]
rejecting two more classes. Finally, $|\widetilde{y}-(w-u-v)/2|^2=4+1-2\widetilde{y}\cdot(w-u-v)/2=4$. Suppose $\widetilde{y}=(w-u-v)/2+g$, $g\in \Lambda_{24}$. Then 
\[ 
\begin{split}
\widetilde{y}-m_{112}&=g+\frac{w-u-v}{2}-\frac{3w-u-v}{4}\\
&=g-\left(w-\frac{3w-u-v}{4}\right)=g-(w-m_{112}) =g-m_{112}',
\end{split}\]
where $m_{112}':=w-m_{112}$ is the center of mass of $C_{112}' := w- C_{112}$,  which is the antipodal image of $C_{112}$ with respect to $w/2$. Clearly, $|\widetilde{y}-m_{112}|=|g-m_{112}'|=3/2$. Thus, $g\in C_{112}'$, so $\widetilde{y}\in (w-u-v)/2+C_{112}'$.

Note that the circumscribed hyperspheres of $C_{112}$ and $C_{112}'$ are $(m_{112}+G)\cap S(0,2)$ and $(m_{112}'+G)\cap S(0,2)$, respectively.  Let
\begin{equation}\label{2C112}
C:=C_{112} \cup \left( \frac{w-u-v}{2}+C_{112}' \right).
\end{equation}
It is straightforward that for any point  $g\in C_{112}'$ we have $(w-u-v)/2+g=2m_{112}-z_i$ for some $i$, so each such vector identifies a universal minimum of $U_h(x,C_{162})$.

This establishes the maximality of the PULB-optimal pair $(C,D)$ generated by the sharp codes  $C_{112}$ and $C_{162}$, which we summarize in the following theorem.

\begin{theorem}\label{112PULBpair} Let $C$ and $D$ be as defined in \eqref{2C112} and \eqref{4C162} and denote their projections onto $\mathbb{S}^{20}$ with the same letters. Then for any potential $h$ with $h^{(4)}(t)> 0$, $t\in (-1,1)$, the codes  $C=(21,224,3)$ and  $D=(21,648,3)$, are PULB-optimal; i.e.,they attain the bound \eqref{PolarizationULB} 
\[ \begin{split}
 m^h (C)&\geq 112h\left( -1/\sqrt{21}\right)+112h\left(1/\sqrt{21}\right), \\
  m^h (D)&\geq 324h\left( -1/\sqrt{21}\right)+324h\left(1/\sqrt{21}\right).
 \end{split} \]

Moreover, $(C,D)$ forms a maximal PULB-optimal pair (and hence a universal polar dual pair) with normalized discrete potentials achieving the same extremal value
\[ \frac{m^h(C)}{224}= \frac{m^h(D)}{648} \left( =\frac{m^h (C_{112})}{112}=\frac{m^h (C_{162})}{162}\right). \]
\end{theorem}

\section{A new universally optimal code in the real projective space $\mathbb{RP}^{21}$ }\label{B1408}

Strictly speaking, the code $A_2=(22,2816,5)$ from Section 5.3 is not new and even might be well known to the specialists in lattices. 
However, its universal optimality as a code in the real projective space was not mentioned neither by Levenshtein in his tables from \cite{Lev92} 
(1992) and \cite{Lev} (1998) nor by
Cohn, Kumar, and Minton \cite{CKM16} (2016, when the concept of universal optimality by Cohn and Kumar \cite{CK} was already in use). 

The optimality of $A_2$ among all antipodal spherical codes on $\mathbb{S}^{21}$ of cardinality $2816$ follows from the Levenshtein bound \cite{Lev82} on codes on $\mathbb{S}^{n-1}$ with bounded 
modulus of the inner product from 1982. 
In slightly different setting of the linear programming approach, it was first explicitly noticed by the second author in 1993 (cf. Theorem 6 in \cite{Boy}). 

\begin{theorem} \label{b1408} Let $B_{1408}$ be the projective code of 1408 lines through the origin corresponding to the pairs of antipodal points of $A_2$. Then $B_{1408}$ is universally optimal in $\mathbb{RP}^{21}$. 
\end{theorem}

\begin{proof} The code $B_{1408}$ belongs to the real projective space $\mathbb{RP}^{21}$ with distances and "inner products" as described by Levenshtein 
\cite[Section 9]{Lev92}, \cite[Section 6.2]{Lev} (see also Section 8 in \cite{CK}). Since the inner products of $A_2$ as a spherical code are $-1$, $\pm 1/3$, and $0$, the substitution $\sigma =2t^2-1$ \cite[page 77]{Lev92} gives "inner products" $\sigma=-1$ and $-7/9$ of $B_{1408}$ as a projective code. Then $B_{1408}$ attains the Levenshtein bound for maximal codes in $\mathbb{RP}^{21}$ and is, therefore, universally optimal via Theorem 8.2 from \cite{CK}.
\end{proof}  

\begin{remark} The code $B_{1408}$ considered as a graph, where two vertices are connected whenever their inner product is $-1$, is a Conway graph of $1408$ vertices, which is a strongly regular graph ${\rm srg}(1408,567,246,216)$ (see \cite[Subsection 10.81]{BrSRG}). It is the second subconstituent of the Conway graph of $2300$ vertices \cite[Subsection 10.88]{BrSRG}, whose construction in the Leech lattice is obtained by considering as vertices the $2300$ antipodal with respect to the center of mass pairs of points of $K_1$, where two vertices (pairs) are connected when the inner products of any two representatives (of the pairs) are even.
\end{remark}

The problem of optimality of $A_2$ as a spherical code is open. The Levenshtein bound for (usual) spherical codes (see, for example,
Equation (6.12) in \cite{Lev}; to be applied with $n=22$, $k=6$, and $s=1/3$) gives $\approx 3513.36$ which is far from 2816.
Moreover, it cannot be improved by linear programming (this can be proved via the so-called test functions as in \cite[Section 4]{BDHSS}; 
see also \cite[Section 6]{CLL}). The test functions work as defined in 1996 in \cite{BDB} (see also Theorem 5.47 in \cite{Lev}).  

We conclude this section with a corollary of Theorem \ref{b1408} about the universal optimality of $A_2$ among all antipodal codes.

\begin{corollary} The code $A_2$ is universally optimal among all antipodal codes on $\mathbb{S}^{21}$ of cardinality $2816$; i.e.,for any antipodal $C\subset \mathbb{S}^{21}$, $|C|=2816$, we have that for any absolutely monotone potential $h$
$$ E_h (C)\geq E_h (A_2).$$
\end{corollary}

\noindent{\bf Acknowledgment.} The research of PB and MS was supported, in part, by Bulgarian NSF grant KP-06-N72/6-2023. The research of PD was supported, in part, by the Lilly Endowment. 
\vskip 2mm

\noindent{\bf AI Declaration.} The authors used Chat GPT 5.5 to verify the proofs and improve grammatical content. It was used neither to construct new proofs nor to generate new theorems.

\end{document}